\documentclass[12pt]{article}
\usepackage[dvips]{graphicx}

\usepackage[titletoc,toc,title]{appendix}
\usepackage{amstext}
\usepackage{enumerate}
\usepackage{epsf}
\usepackage{psfrag}
\usepackage{wrapfig}
\usepackage[english]{babel}
\usepackage{type1cm}
\usepackage{array}
\usepackage{colortbl}
\usepackage{titlesec}
\usepackage{algorithmicx}
\usepackage{algorithm}
\usepackage[noend]{algpseudocode}
\usepackage{mathtools}
\usepackage{color,rotating,psfrag,amsfonts}
\usepackage[dvips]{epsfig}
\usepackage{subfigure}
\usepackage{amsmath}
\usepackage{amssymb}
\usepackage{amsthm}
\usepackage{epic}
\usepackage[margin=1cm]{caption}
\usepackage[titletoc,toc,title]{appendix}
\usepackage{bbm,bm}
\usepackage{tikz}

\usepackage{booktabs}
\usepackage{siunitx}
\usepackage{pgfplots}
\usepackage{titlesec}
\usepackage[normalem]{ulem}

\usepackage{tikz}
\usetikzlibrary{calc}

\DeclareMathOperator*{\argmax}{arg\,max}
\DeclareMathOperator*{\argmin}{arg\,min}
\newcommand{\overbar}[1]{\mkern 1.5mu\overline{\mkern-1.5mu#1\mkern-1.5mu}\mkern 1.5mu}

\def\Put(#1,#2)#3{\leavevmode\makebox(0,0){\put(#1,#2){#3}}}

\newcommand {\bu} {\mathbf{u}}
\newcommand {\bv} {\mathbf{v}}

\newcommand {\bY} {\mathbf{Y}}

\newcommand{\E}[1]{\ensuremath{\mathbb{E}\left(#1\right)}}
\newcommand{\var}[1]{\ensuremath{\mathbb{V}\left( #1\right)}}
\newcommand{\cov}[2]{\ensuremath{\mathbb{C}\left( #1,#2\right)}}
\newcommand{\Pp}[1]{\ensuremath{\mathbb{P}\left( #1\right)}}

\newtheorem{lemma}{Lemma}

\begin{document}

\title{Adaptive stratified sampling for non-smooth problems}

\author{Per Pettersson$^{a}$, Sebastian Krumscheid$^{b}$ \\
$^{a}$ NORCE Norwegian Research Centre, N-5838 Bergen, Norway \\
$^{b}$ RWTH Aachen University, 52072 Aachen, Germany}

\date{}

\maketitle

\begin{abstract}

  Science and engineering problems subject to uncertainty are
  frequently both computationally expensive and feature non-smooth
  parameter dependence, making standard Monte Carlo too slow, and
  excluding efficient use of accelerated uncertainty quantification
  methods relying on strict smoothness assumptions. To remedy these
  challenges, we propose an adaptive stratification method suitable
  for non-smooth problems and with significantly reduced variance
  compared to Monte Carlo sampling. The stratification is iteratively
  refined and samples are added sequentially to satisfy an allocation
  criterion combining the benefits of proportional and optimal
  sampling. Theoretical estimates are provided for the expected
  performance and probability of failure to correctly estimate
  essential statistics.  We devise a practical adaptive stratification
  method with strata of the same kind of geometrical shapes,
  cost-effective refinement satisfying a greedy variance reduction
  criterion. Numerical experiments corroborate the theoretical
  findings and exhibit speedups of up to three orders of magnitude
  compared to standard Monte Carlo sampling.

\end{abstract}

\section{Introduction}
\label{sec:intro}

Many complex engineering problems are non-smooth functions of physical
parameters with unknown values that can be modeled as random
variables. Problems where the solution or its derivative is non-smooth
or even discontinuous in uncertain physical parameters include
high-speed flow in computational fluid dynamics~\cite{Hosder_etal_08},
porous media flows~\cite{Christie_etal_06} and
transport~\cite{Tartakovsky_Broyda_11}, weather and climate
predictions~\cite{Palmer_00}, and
geomechanics~\cite{Pereira_etal_16}.
Due to the non-smooth parameter dependence, quantifying the
uncertainty in the solution of these problems commonly involves
repeatedly sampling from the parameter domain. As a single evaluation
of the underlying model is in general very computationally demanding,
efficient sampling is essential to keep the total number of samples at
a minimum to obtain some prescribed statistical tolerance
accuracy. For classic Monte Carlo based sampling methods, for example,
the error decays as $CN^{-1/2}$, where $C$ is typically well
approximated by a constant depending on the problem at hand, and $N$
is the number of samples. Increasing the number of samples by two
orders of magnitude to get a single order of magnitude error reduction
is often numerically intractable, in particular when the model at hand
describes a complex physical problem. Variance reduction techniques
aim at reducing the overall computational cost by providing an
estimator with reduced variance compared to the Monte Carlo variance
$CN^{-1/2}$ for fixed $N$; see, e.g., \cite[Ch.~9]{Kroese_etal_11} or
\cite[Ch.~V]{Asmussen_Glynn_07} and the references therein. Improving
upon the canonical Monte Carlo rate of $-1/2$, for example by using
quasi-Monte Carlo sampling \cite{Niederreiter_1992, Dick2013} or
spectral methods \cite{Xiu_2010}, typically requires a regular (i.e.,
differentiable) dependence of the quantity of interest with respect
to the uncertain parameters. Conversely, general-purpose variance
reduction techniques that are effective when there is a non-smooth
parameter dependence aim at decreasing the constant $C$.
Variance reduction via generalized (approximate) control variates,
such as the Multi-Level and Multi-Index Monte Carlo methods, has
become a very popular approach for a wide range of applications due to
its computational efficiency; see, e.g., \cite{Cliffe_etal_11,
  Giles_2015, Haji-Ali_etal_16} for a general account and
\cite{Muller_etal_13,Krumscheid2018,Giles_2019,Pisaroni_2020} for
variants tailored to, and applications for, particular non-smooth
quantities of interest. These methods are in particular suitable when
a hierarchy of correlated models of different degree of fidelity can
be established, e.g., multiple physical grids of different
resolution. For problems such as those modeling fractured porous media
for example, it is however often prohibitive to introduce conforming grids of
different degrees of refinement.

An alternative means to obtain variance reduction that does not rely
on the concept of a hierarchy of different models (for instance when
only a single unstructured grid is available) is offered by stratified
sampling methods \cite{Asmussen_Glynn_07,Kroese_etal_11}. The idea is
straightforward: the stochastic domain is partitioned into disjoint
subsets, so-called strata, and a suitable number of samples are drawn
from each stratum. A quantity of interest can be computed as a
function of local mean values of each stratum. The number of samples
can be chosen differently from standard Monte Carlo sampling, which
offers the potential to achieve an estimator with significantly
reduced standard error. The optimal number of samples per stratum is a
function of the local variance and the size of the
stratum. Collectively, this leads to a variance reduction compared to
Monte Carlo sampling. Thus, the number of times an often expensive
numerical simulator needs to be solved can be reduced with significant
computational speedup as a result.
The existing literature is mainly concerned with either the
proportional or the optimal allocation of samples given a fixed
stratification of the stochastic domain. In this work, we introduce a
\emph{hybrid} sample allocation rule as a linear combination of
proportional and optimal allocation. Moreover, the optimal
stratification of the stochastic domain is not assumed known
\emph{apriori}.  We therefore devise an original method where we start
from a single stratum and adaptively subdivide the stochastic domain
into finer strata, while assigning new samples to asymptotically
satisfy a sampling distribution defined by some prescribed allocation
rule.  
%

Stratified sampling is usually restricted to low dimensions due to
exponential growth in the number of strata with increasing
dimension. This problem can be overcome by using, e.g., latin
hypercube sampling where only the marginal distributions are
stratified~\cite{Kroese_etal_11}.
Here we use an adaptive strategy to attenuate the dimensional
restriction.
It is noteworthy, however, that
while adaptivity enables to use stratified sampling for moderate
dimensions, 
its efficiency will nonetheless degenerate in higher dimensions.
In~\cite{Etore_Jourdain_10} the authors introduce an asymptotically
optimal stratified sampling estimator for a fixed stratification using
the standard empirical variance estimators to update the stratum
parameters. Adaptivity of the stratification was suggested
in~\cite{Press_Farrar_90}, where recursive adaptive stratification was
used to bisect a parallelepiped domain based on the maximum of the
squares of the differences between minimum and maximum sample values
in the tentative stratification.  An adaptive stratification algorithm
using a sequentially updated stratification matrix determining stratum
boundaries was proposed in~\cite{Etore_etal_11}. Refined stratified
sampling where single-sample strata are sequentially bisected, and the
stratification enriched by a single-sample stratum at the time was
introduced in~\cite{Shields_etal_15}. This framework was later
combined with hierarchical Latin hypercube sampling to target very
high-dimensional problems~\cite{Shields_16}.


Adaptive stratification is a special
case of adaptive decomposition of the stochastic
domain. In~\cite{DeLuigi_Maire_10} the authors describe an adaptive
Quasi Monte Carlo method on a hypercube stratification with bisections
of strata maximizing an error indicator. The
work~\cite{Witteveen_Iaccarino_12} introduces a simplex tessellation
of stochastic parameter space to discretize irregular
domains. Targeting discontinuous functions, the authors
of~\cite{Jakeman_etal_13} suggested a stochastic domain decomposition
based on solution discontinuities identified via polynomial
annihilation. A Voronoi tessellation of the random space based on
random samples, together with localized surrogate modeling to capture
solution discontinuities, has been proposed in
\cite{Rushdi_etal_2017}. Also based on polynomial annihilation for
discontinuity detection is the work~\cite{Gorodetsky_Marzouk_14} using
support vector machines for building stochastic surrogate models.

The paper is structured as follows. Section~\ref{sec:basic-strat-samp}
describes the classical stratified sampling estimator. Moreover, it
introduces the novel hybrid stratified sampling
estimator. Section~\ref{sec:adapt-strat-samp} presents the new
adaptive stratification algorithm together with a variance minimizing
stratum splitting strategy as well as a probabilistic robustness
analysis due to exact stratum variances being replaced by sample
estimates. In Section~\ref{sec:practical-stratification}, we discuss
practical aspects of the proposed adaptive sampling algorithm,
including choices of stratum shapes and efficient strata splitting
strategies. Numerical results are presented in
Section~\ref{sec:num-res}, where the method is tested on a hierarchy
of different problems, ranging from synthetic test cases to those
describing simplified geomechanics during CO$_2$ injection into an
aquifer. The paper ends with a discussion and outlook in
Section~\ref{sec:concl}.


\section{Stratified Sampling Estimators} 
\label{sec:basic-strat-samp}
We consider a quantity of interest $Q$ that is given as a
scalar-valued measurable function $f\colon\mathbb{R}^n\to\mathbb{R}$
of an $n$-dimensional random vector
$\bY = (Y_1,\dots, Y_n) \in \mathfrak{U}\subset\mathbb{R}^{n}$, that
is, $Q = f(\bY)$ constitutes a random variable on some probability
space $(\Omega,\mathcal{F},\mathbb{P})$. Specifically, throughout this
work we assume that $Q\in
L^2(\Omega,\mathcal{F},\mathbb{P})$. Furthermore, we suppose that the
components $Y_k$ ($k=1,\hdots,n$) are mutually independent random
variables with finite variance and known, but not necessarily
identical probability distributions. If $F_k$ denotes the cumulative
distribution function of $Y_k$, $k=1,\hdots,n$, then the random
variables $Y_k$ and $F^{-1}_{Y_k}(U_{k})$ for a uniformly distributed
random variable $U_k$ on $[0, 1]$ have the same distribution in view
of the inverse probability integral transform. One may thus write
$Y_k = F^{-1}_{Y_k}(U_{k})$ for $k=1,\hdots,n$ with $U_1,\hdots, U_n$
mutually independent and identically distributed (i.i.d.),
$U_1 \sim U[0,1]$. Without loss of generality, we may therefore assume
that the stochastic domain $\mathfrak{U}$ is the hypercube
$\mathfrak{U} = [0, 1]^n$ and that $Y_k\sim U[0,1]$ for $k=1,\dots,n$,
viewing the particular choices of $F_k$ as part of the ``model'' $f$,
which, throughout this work, is assumed to be ``complicated'' and
computationally expensive to evaluate.

\subsection{Basic definition and properties}
\label{sec:basic-strat-samp:def}
Stratified sampling is an estimation technique to approximate
$\E{Q} = \E{f(\bY)}$ with reduced variance compared to classic Monte
Carlo sampling \cite{Asmussen_Glynn_07,Kroese_etal_11}. The variance
reduction is achieved by dividing the stochastic domain $\mathfrak{U}$
into multiple disjoint regions, so-called strata, aiming at reducing
the variation in each stratum. Specifically, let $\mathcal{S}$ be a
stratification of the domain $\mathfrak{U}$, in the sense that
$\mathfrak{U} = \cup_{S\in\mathcal{S}}S $ and $S\cap T = \emptyset$
for $S,T\in\mathcal{S}$ with $S\not=T$. 
For each
stratum $S\in\mathcal{S}$ we define $Q_S\colon\Omega\to\mathbb{R}$ as
the random variable with distribution of $Q = f(\bY)$
conditioned upon $\bY\in S$, that is, the distribution
$ \mathbb{P}_{Q_S}$ of $Q_S$ is given by
\begin{equation*}
  A\mapsto \mathbb{P}_{Q_S}(A) \equiv \mathbb{P}(Q_S\in A) = \frac{\mathbb{P}(Q\in A, \bY\in S)}{\mathbb{P}(\bY\in S)}= \frac{\mathbb{P}\bigl(f(\bY)\in A, \bY\in S\bigr)}{\mathbb{P}(\bY\in S)}\;,
\end{equation*}
for any $A\in\mathcal{F}$.  Let $p_S := \mathbb{P}(\bY\in S)$ denote
the measure (or ``size'') of $S$. Then
\begin{equation*}
  \E{Q} = \E{f(\bY)} = \sum_{S\in\mathcal{S}} p_S \E{Q_S} =  \sum_{S\in\mathcal{S}} p_S \E{f(\bY)|\bY\in S}
\end{equation*}
by the law of total probability. The stratified sampling estimator
$\hat{Q}$ of $\E{Q}$ is obtained by estimating the expected value
$\E{Q_S}$ in each stratum $S$ by a Monte Carlo average based on
$N_S\in\mathbb{N}$ i.i.d.\ samples of $Q_S$:
\begin{equation}
  \label{eq:ss_estimator}
  \hat{Q} := \sum_{S\in\mathcal{S}}  \frac{p_S}{N_S}\sum_{j=1}^{N_S}Q_{S}^{(j)}\;,\quad Q_{S}^{(j)}\sim\mathbb{P}_{Q_S}\quad
  1\le j\le N_S\;.
\end{equation}
Consequently, $\hat{Q}$ relies on a total of
$N := \sum_{S\in\mathcal{S}} N_S$ number of samples and constitutes an
unbiased estimator of $\E{Q}$. 
For~\eqref{eq:ss_estimator} to provide a practical estimator, the measure
$p_S$ has to be known for every stratum $S\in\mathcal{S}$ and,
furthermore, it needs to be possible to sample from the distribution
of ${Q_S}$. Throughout this work, we will assume that this is the case;
see Sect.~\ref{sec:practical-stratification} for possibilities how to
realize this in practice.  The estimator's variance is then given by
\begin{equation} \label{eq:ss_estimator:var:general}
  \var{\hat{Q}} = \sum_{S\in\mathcal{S}}\frac{p_S^2 \sigma_S^2}{N_S} = \frac{1}{N}\sum_{S\in\mathcal{S}}\frac{p_S^2 \sigma_S^2}{N_S/N},
\end{equation}
where $\sigma_S^2 := \var{Q_S} = \var{f(\bY)|\bY\in S}$.
If the numbers of samples $N_S$ are selected such that
$\lim_{N\to\infty}\frac{N}{N_S} < \infty$ for all $S\in\mathcal{S}$,
then $\lim_{N\to\infty} N\, \var{\hat{Q}} < \infty$ and the stratified
sampling estimator $\hat{Q}$ satisfies the central limit theorem, in
the sense that
\begin{equation}\label{eq:strat:basic:clt}
   \frac{\hat{Q} - \E{Q}}{\sqrt{\var{\hat{Q}}}} \underset{N\to\infty}{\Rightarrow}\mathcal{N}(0,1)\;, 
\end{equation}
see, e.g., \cite{Asmussen_Glynn_07} and the references therein. In
practice, when $N$ is sufficiently large, this asymptotic normality can
be used to report also an approximate confidence region for the point
estimate $\hat{Q}$. In that case, the natural variance estimator for
\eqref{eq:ss_estimator:var:general} is $\hat{V}$, which is obtained by
replacing the unknown variances $\sigma_S^2$ in each stratum by the
empirical variances $\hat{\sigma}_S^2$. For example, let
$z_{(1+p)/2} := \Phi^{-1}(\frac{1+p}{2})$ be the $(1+p)/2$ quantile of
the standard normal distribution, then
\begin{equation}
  \label{eq:basic:strat:clt:ci}
    p\approx \mathbb{P}\left( \left\vert \hat{Q} - \E{Q} \right\vert \le z_{\frac{1+p}{2}}{\var{\hat{Q}}}^{1/2}\right)\;,
  \end{equation}
  in view of the asymptotic normality of $\hat{Q}$.

\subsection{Proportional and optimal sample allocation rules}
\label{sec:basic-strat-samp:prop-opt-allocation}
To make the stratified sampling
estimator practical, one has to select a rule for the number of
samples $N_S$ in each stratum $S\in\mathcal{S}$. There are two popular
choices, namely proportional allocation and optimal sample allocation,
that is the number of samples are chosen according to
\begin{equation}\label{eq:allo_rules}
  N^{\text{prop}}_S := p_S N\quad\text{and}\quad
  N^{\text{opt}}_S := \frac{p_S\sigma_S}{\sum_{S\in \mathcal{S}}p_S \sigma_S} N\;,
\end{equation}
respectively.  In practice these rules are, of course, only used up to
integer rounding. For the discussion in this section, we will, however,
consider the number of samples in strata as continuous variables for
simplicity.

The latter allocation rule in ~\eqref{eq:allo_rules} is optimal
in the sense that it provides the stratification estimator with the
smallest variance for a given stratification $\mathcal{S}$ with a total of $N$ samples. 
Specifically, the estimator's variance ${V}_{\text{opt}}$ using
optimal sample allocation $N_S = N_S^{\text{opt}}$ and the
estimator's variance ${V}_{\text{prop}}$ using proportional sample
allocation $N_S = N_S^{\text{prop}}$ satisfy,
\begin{equation}\label{eq:allo_vars}
  {V}_{\text{opt}}= \frac{1}{N} {\Bigl(\sum_{S\in\mathcal{S}}p_S \sigma_S\Bigr)}^2
    \le 
    \frac{1}{N} \sum_{S\in\mathcal{S}}p_S \sigma_S^2
    = {V}_{\text{prop}} \le \frac{\var{Q}}{N}\;,
\end{equation}
where the first inequality is a consequence of Jensen's inequality and
the second inequality follows from the law of total variance. Notice
that the preceding display indicates that stratification with either
sample allocation rule never increases the variance compared to
classic Monte Carlo sampling with $N$ samples.

While the optimal sample allocation rule 
provides an estimator with minimal variance, its practical implementation faces
the difficulty that the strata standard deviations $\sigma_S$ are typically unknown.
Possible remedies include
estimating the standard deviations using a pilot run, see, e.g.,
\cite{Kroese_etal_11}, or via an adaptive procedure by sequentially
allocating samples, such as the algorithm introduced
in~\cite{Etore_Jourdain_10} which asymptotically achieves the optimal
variance ${V}_{\text{opt}}$.  We reiterate that in this work we are
concerned with the case where $f$ represents a complex model that may
be computationally expensive to evaluate. Pilot runs to estimate the
variances in each stratum are therefore not affordable. 

\subsubsection{Probability of misestimating the estimator's variance}\label{sec:basic-strat-samp:prop-opt-allocation:concentration}
In addition to potential practical difficulties, the choice of the
sample allocation rule may also affect the accuracy of the
stratification estimator's variance estimate, which is, for example,
used in \eqref{eq:basic:strat:clt:ci} to offer approximate confidence
intervals. 
In practice, the natural variance estimators for the stratification estimator using
proportional and optimal sample allocation are
\begin{equation}\label{eq:basic:var:est:classic}
  \hat{V}_{\text{prop}} := \frac{1}{N} \sum_{S\in\mathcal{S}}p_S \hat{\sigma}_S^2
  \quad\text{and}\quad
  \hat{V}_{\text{opt}} := \frac{1}{N} {\Bigl(\sum_{S\in\mathcal{S}}p_S \hat{\sigma}_S\Bigr)}^2\;,
\end{equation}
respectively, where $\hat{\sigma}^2_S $ denotes the empirical variance 
\begin{equation*}
  \hat{\sigma}^2_S := \frac{1}{N_S-1}\sum_{j=1}^{N_S}{\left( Q_S^{(j)} - \overbar{Q}_S \right)}^2\;,\quad \overbar{Q}_S := \frac{1}{N_S}\sum_{j=1}^{N_S}Q_{S}^{(j)}\;,
\end{equation*}
in stratum $S$ based on $N_S $ i.i.d.~samples
$Q_{S}^{(1)}, \dots, Q_{S}^{(N_S)}$ and
$\hat{\sigma}_S \equiv \sqrt{\hat{\sigma}^2_S}$ is the corresponding
empirical standard deviation. The variance estimates in
\eqref{eq:basic:var:est:classic}, in particular
$\hat{V}_{\text{opt}}$, may have practical limitations for strata with
small sample sizes $N_S$, because $\hat{\sigma}_S$ is not an unbiased
estimator for the standard deviation $\sigma_S$. In fact,
$\hat{\sigma}_S$ is only an asymptotically unbiased estimator in
contrast to $\hat{\sigma}^2_S$, which is an unbiased estimator for the
variance $\sigma_S^2$. The Lemma below states concentration
inequalities for quantifying the deviations of both of the
stratification estimator's variance estimates from the theorized
values. Their proof is based on arguments similar to the ones used for
classic Chernoff bounds; see, e.g., \cite[Chap.~2]{Vershynin_18}.
\begin{lemma}\label{lem:basic:strat:concentration}
  For a given stratification $\mathcal{S}$, let
  $\hat{V}_{\text{prop}}$ and $\hat{V}_{\text{opt}}$ be as
  in~\eqref{eq:basic:var:est:classic}.  Suppose that for every stratum
  $S\in\mathcal{S}$ there exist finite constants $0\le M_S$, such that
  the empirical standard deviations satisfy $\hat\sigma_S \le M_S$
  almost surely. Then the variance estimator for proportional sample
  allocation satisfies
\begin{equation*}
  \mathbb{P}\Bigl(\bigl\vert \hat{V}_{\text{prop}} - {V}_{\text{prop}} \bigr\vert \ge \vartheta \Bigr)
  \le 2 \exp{\left(- \frac{2 \vartheta^2 N^2}{\sum_{S\in\mathcal{S}} {p_S}^2 {M_S}^4}\right)}\;,
\end{equation*}
for any $\vartheta > 0$. Conversely, the variance estimator for
optimal sample allocation satisfies
\begin{equation*}
  \mathbb{P}\Bigl(\bigl\vert \hat{V}_{\text{opt}} - {V}_{\text{opt}} \bigr\vert \ge \vartheta \Bigr)
  \le 2 \exp{\left(-\frac{2{\left(\lvert B\rvert - \vartheta N\right)}^2}{{\left(\sum_{S\in\mathcal{S}} p_S M_S\right)}^4}\right)}\;,
  \end{equation*}
  for any $\vartheta >\frac{\lvert B\rvert}{N}$, where
  $B := \sum_{S,T\in\mathcal{S}, S\not= T}p_S p_T \bigl(b_S b_T + b_S
  \sigma_T + b_T \sigma_S\bigr)$ denotes the estimator's bias with
  $b_S := \mathbb{E}(\hat\sigma_S) - \sigma_S$ for $S\in\mathcal{S}$.
\end{lemma}
\begin{proof}
  Let $\vartheta\ge0$. It follows from Markov's inequality that
  \begin{equation}
    \begin{aligned}
      \mathbb{P}\Bigl(\bigl\vert \hat{V} - \mathcal{V} \bigr\vert \ge \vartheta \Bigr)
      &= \mathbb{P}\bigl(N(\hat{V} - \mathcal{V} ) \ge  \vartheta N\bigr) + \mathbb{P}\bigl(N(\hat{V} - \mathcal{V} ) \le  -\vartheta N\bigr)\\
      &= \mathbb{P}\bigl( e^{t N(\hat{V} - \mathcal{V} )} \ge  e^{t\vartheta N}\bigr) + \mathbb{P}\bigl( e^{- t N(\hat{V} - \mathcal{V} )} \ge  e^{t \vartheta N}\bigr)\\
      & \le e^{- t\vartheta N} \Bigl[\E{e^{t N(\hat{V} - \mathcal{V} )}} + \E{e^{- t N(\hat{V} - \mathcal{V} )}} \Bigr]\;,
    \end{aligned}\label{eq:basic:bound:var:est}
  \end{equation}
  for any $t>0$, where $V$ and $\hat{V}$ are placeholders for any
  stratified sampling estimator's true variance and its estimated
  variance, respectively.

  We begin with the case of proportional sample allocation. That is,
  we use \eqref{eq:basic:bound:var:est} with
  $\hat{V} \equiv \hat{V}_{\text{prop}}$ and
  ${V} \equiv {V}_{\text{prop}}$. As the variance estimators
  $\hat\sigma_{S}^2$ and $\hat\sigma_{T}^2$ are independent for strata
   if $S\not= T$, it follows from Hoeffding's
  lemma and the hypotheses that
\begin{equation*}
  \E{e^{\pm t N(\hat{V} - \mathcal{V} )}} = \prod_{S\in\mathcal{S}} \E{e^{\pm t({\hat\sigma}_S^2 - {\sigma_S^2})p_S} }
  \le \prod_{S\in\mathcal{S}} e^{\frac{t^2 p_S^2}{8} M_S^4} = e^{\frac{t^2}{8} \sum_{S\in \mathcal{S}}p_S^2\,M_S^4}\;,
\end{equation*}
since $\E{\hat{\sigma}_S^2} = \sigma_S^2$ for all
$S\in\mathcal{S}$. Consequently, we obtain the bound
\begin{equation*}
  \mathbb{P}\Bigl(\bigl\vert \hat{V} - \mathcal{V} \bigr\vert \ge \vartheta \Bigr)
  \le 2 e^{\left(\frac{t^2}{8} \sum_{S\in \mathcal{S}}p_S^2\,M_S^4 - t\vartheta N\right)}\;.
\end{equation*}
for any $t>0$.  Minimizing the right-hand side over $t>0$ yields the
claim.

For the case of optimal sample allocation, we proceed similarly. In
fact, using 
$\hat{V} \equiv \hat{V}_{\text{opt}}$ and
${V} \equiv {V}_{\text{opt}}$, Hoeffding's lemma yields
\begin{equation*}
  \E{e^{\pm t N(\hat{V} - \mathcal{V} )}} = \E{e^{\pm t \sum_{S,T\in\mathcal{S}}p_S p_T(\hat\sigma_S\hat\sigma_T - \sigma_S\sigma_T)} }
  \le e^{\pm t B + \frac{t^2}{8} {\left(\sum_{S\in\mathcal{S}} p_S M_S\right)}^4}
\end{equation*}
in this case, where we have used that
\begin{equation*}
  -{\Bigl(\sum_{S\in\mathcal{S}} p_S \sigma_S\Bigr)}^2\le  N\bigl(\hat{V} - \mathcal{V}\bigr) \le
{\Bigl(\sum_{S\in\mathcal{S}} p_S M_S\Bigr)}^2 -
{\Bigl(\sum_{S\in\mathcal{S}} p_S \sigma_S\Bigr)}^2\;.
\end{equation*}
 Here, the bias term $B$ is given as
\begin{equation*}
  B := N \E{\hat{V}-V} = \sum_{S,T\in\mathcal{S}, S\not= T}p_S p_T \bigl(b_S b_T + b_S \sigma_T + b_T \sigma_S\bigr)\;,
\end{equation*}
with $b_S := \mathbb{E}(\hat\sigma_S) - \sigma_S$. 
Combining these bounds with
\eqref{eq:basic:bound:var:est} eventually yields
\begin{equation*}
  \mathbb{P}\Bigl(\bigl\vert \hat{V} - {V} \bigr\vert \ge \vartheta \Bigr)
  \le \min_{t>0} 2 e^{t \left(\lvert B\rvert - \vartheta N\right) + \frac{t^2}{8} {\left(\sum_{S\in\mathcal{S}} p_S M_S\right)}^4}
  = 2 \exp{\left(-\frac{2{\left(\lvert B\rvert - \vartheta N\right)}^2}{{\left(\sum_{S\in\mathcal{S}} p_S M_S\right)}^4}\right)}\;,
\end{equation*}
provided that $\vartheta > \lvert B\rvert/N$, thus completing the
proof.
\end{proof}
The Lemma above shows that when using optimal sample allocation, the
natural stratification estimator's variance estimator is affected by a
bias term for small samples sizes $N$, which originates from the
biased strata standard deviation estimators. In contrast, proportional
allocation is not affected by such a bias. Moreover, a sufficient
condition for the hypothesis that the empirical standard deviations
satisfy $\hat\sigma_S\le M_S< \infty$ almost surely in every stratum
$S\in\mathcal{S}$ is that the function $f$ is bounded on every
stratum. 
The concentration inequalities stated in
Lemma~\ref{lem:basic:strat:concentration} indicate that the right-hand
sides can be decreased for a stratification $\mathcal{S}$ that isolates
highly varying regions of $f$ in strata of small measure. The adaptive
procedure introduced in this work, which will be detailed in
Sect.~\ref{sec:adapt-strat-samp}, takes advantage of this observation.
Finally, it is noteworthy that
Lemma~\ref{lem:basic:strat:concentration} may be strengthened by using
sharper bounds on the empirical standard deviations. For the purpose
of this work and to highlight the effects of biased strata standard
deviation estimators, the version presented here is sufficient,
however.

\subsection{Hybrid sample allocation rules}
\label{sec:basic-strat-samp:hyrbid-allocation}
Instead of using either proportional or optimal allocation for the
stratification procedure presented in this work, it will be useful to
consider sample allocation rules that are ``between'' both
versions. Specifically, let $\alpha\in[0,1]$ and consider the hybrid
sample allocation rule
\begin{equation}
    N_S^{\alpha} := (1-\alpha) N_S^{\text{prop}} + \alpha N_S^{\text{opt}}= p_S N \left( 1 + \alpha (\overline{\sigma}_S-1) \right) \;,
  \label{eq:allo_rules:hybrid}
\end{equation}
where
\begin{equation*}
  \overline{\sigma}_S := \frac{\sigma_S}{\sum_{T\in\mathcal{S}}p_T\sigma_T}\;,
\end{equation*}
so that that number of total samples
$N = \sum_{S\in\mathcal{S}}N_S^{\alpha}$ is as before. Notice that the
hybrid allocation rule contains, in particular, both proportional
allocation ($\alpha=0$) and optimal allocation ($\alpha = 1$) as
special cases. For the parameterized sample allocation rule
$N_S^{\alpha}$, $\alpha\in[0,1]$, we define the family of hybrid
stratified sampling estimators in the natural way by
  \begin{equation}
  \label{eq:ss_estimator:hybrid}
  \hat{Q}_{\alpha} := \sum_{S\in\mathcal{S}}  \frac{p_S}{N_{S}^{\alpha}}\sum_{j=1}^{N_{S}^{\alpha}}Q_{S}^{(j)}\;,\quad Q_{S}^{(j)}\sim\mathbb{P}_{Q_S}\;,
\end{equation}
for all $1\le j\le N_S^{\alpha}$ and $S\in\mathcal{S}$. The variance
of the estimator $ \hat{Q}_{\alpha}$ is thus given by
\begin{equation} \label{eq:ss_estimator:var:hybrid}
  V_\alpha := \var{\hat{Q}_\alpha} =  \frac{1}{N}\sum_{S\in\mathcal{S}}\frac{p_S \sigma_S^2}{1 + \alpha (\overline{\sigma}_S-1)}\;.
\end{equation}
An immediate consequence is that the variance of the hybrid stratified
sampling estimator satisfies the bound
\begin{equation} \label{eq:hybrid:var:bound:simple}
  V_\alpha \le \min\left\{\frac{V_0}{1-\alpha}, \frac{V_1}{\alpha}\right\}
\end{equation}
for any $\alpha\in [0,1]$, since
  $ 1 + \alpha (\overline{\sigma}_S-1) \ge \alpha \overline{\sigma}_S$ as well
  as $ 1 + \alpha (\overline{\sigma}_S-1) \ge 1-\alpha $. Moreover,
direct calculations show that
\begin{equation*}
  V_0 - V_\alpha = \frac{\alpha}{N}\sum_{S\in\mathcal{S}}\frac{p_S \sigma_S^2(\overline{\sigma}_S-1)}{1 + \alpha (\overline{\sigma}_S-1)}\;,\quad\text{and}\quad
  V_\alpha-V_1 = \frac{1-\alpha}{N}\sum_{S\in\mathcal{S}}\frac{p_S \sigma_S^2(\overline{\sigma}_S-1)}{\overline{\sigma}_S\left(1 + \alpha (\overline{\sigma}_S-1)\right)}\;,
\end{equation*}
so that one expects that $V_\alpha\approx V_0$ and $V_\alpha\approx V_1$
for $\alpha\approx 0$ and $\alpha\approx 1$, respectively.

For notational convenience, we introduce the vector notations
\begin{equation*}
  \boldsymbol{\sigma} := {(\sigma_S)}_{S\in\mathcal{S}}\in\mathbb{R}^{\lvert \mathcal{S}\rvert}\quad\text{and}\quad \boldsymbol{p} := {(p_S)}_{S\in\mathcal{S}}\in\mathbb{R}^{\lvert \mathcal{S}\rvert}\;,
\end{equation*}
so that
$ \overline{\sigma}_S = \frac{\sigma_S}{\langle \boldsymbol{p},
  \boldsymbol{\sigma}\rangle}$, using the standard Euclidean inner
product in $\mathbb{R}^{\lvert \mathcal{S}\rvert}$. For any
$\alpha\in[0,1]$, we define the variance constant
$C_\alpha\colon \mathbb{R}^{\lvert \mathcal{S}\rvert} \to \mathbb{R}$
as a function of $\boldsymbol{\sigma}$ via
\begin{equation*}
  C_\alpha(\boldsymbol{\sigma}) :=  \langle \boldsymbol{p}, \boldsymbol{\sigma}\rangle \sum_{S\in\mathcal{S}}\frac{p_S \sigma_S^2}{\alpha \sigma_S + (1-\alpha)\langle \boldsymbol{p}, \boldsymbol{\sigma}\rangle }\;, 
\end{equation*}
so that
the hybrid stratification estimator's variance in \eqref{eq:ss_estimator:var:hybrid}
can be written as
\begin{equation*}
  V_\alpha = \frac{C_{\alpha}(\boldsymbol{\sigma})}{N}\;.
\end{equation*}
Moreover, the hybrid stratified sampling estimator $\hat{Q}_{\alpha}$
satisfies the central limit theorem, cf.~\eqref{eq:strat:basic:clt}, 
\begin{equation*}
   \sqrt{N}\left(\hat{Q}_\alpha - \E{Q}\right) \underset{N\to\infty}{\Rightarrow}\mathcal{N}\left(0, C_{\alpha}(\boldsymbol{\sigma})\right)\;,
\end{equation*}
for any $\alpha\in [0,1]$, provided that
\begin{equation}\label{eq:strat:basic:clt:condition}
  \alpha \sigma_S + (1-\alpha)\langle \boldsymbol{p},
  \boldsymbol{\sigma}\rangle \not = 0\quad\forall S\in\mathcal{S}\;,
\end{equation}
which is derived from the condition
$\lim_{N\to\infty}\frac{N}{N_S} < \infty$, 
underlying the central limit theorem~\eqref{eq:strat:basic:clt} for
general stratified sampling estimators.  Notice that this condition
rules out the degenerate extreme case of a ``perfect''
stratification $\mathcal{S}$ with $\boldsymbol{\sigma} = 0$, but also
the case of optimal sample allocation ($\alpha = 1$) whenever
$\sigma_S = 0$ for some $S\in\mathcal{S}$.

  \subsubsection{Asymptotic distribution of the empirical variance
    constant}
  As before, to make the central limit theorem above practical, e.g.,
  to provide confidence estimates, the variance constant
  $C_\alpha(\boldsymbol{\sigma})$ needs to be estimated.  The natural
  estimator based on the empirical standard deviations
  $\hat{\boldsymbol{\sigma}}:= {(\hat{\sigma}_S)}_{S\in\mathcal{S}}$
  thus is $C_\alpha(\hat{\boldsymbol{\sigma}})$. For a finite sample
  size $N$, the hybrid stratified sampling estimator will, of course,
  satisfy a similar concentration inequality to the ones discussed in
  Sect.~\ref{sec:basic-strat-samp:prop-opt-allocation:concentration}
  for $\alpha\in\{0,1\}$. Indeed, for $\alpha>0$ there will be a bias
  due to $\mathbb{E}(\hat\sigma_S) - \sigma_S \not= 0$ for
  $S\in\mathcal{S}$, which will vanish asymptotically as $N\to\infty$.
  Complementary to the finite sample size concentration
  inequalities presented in Lemma~\ref{lem:basic:strat:concentration},
  here we discuss the asymptotic distribution of the empirical
  variance constant $C_\alpha(\hat{\boldsymbol{\sigma}})$ as
  $N\to\infty$. Specifically, the result below uses the delta method,
  see, e.g., \cite{Serfling_80,Asmussen_Glynn_07}, which relies on the
  gradient $\nabla C_{\alpha}(\boldsymbol{\sigma})$ of the variance
  constant $C_{\alpha}(\boldsymbol{\sigma})$ with respect to
  $\boldsymbol{\sigma}$. We report an explicit expression of the
  gradient $\nabla C_{\alpha}(\boldsymbol{\sigma})$ in
  Appendix~\ref{sec:gradient:varconst} for the reader's convenience.

  \begin{lemma}\label{lemma:hybrid:strat:clt}
    Suppose that $\mathcal{S}$ is such that
    $Q_S\in L^4(\Omega,\mathcal{F},\mathbb{P})$ for all
    $S\in\mathcal{S}$. Suppose further that $\alpha\in[0,1]$ is such
    that condition~\eqref{eq:strat:basic:clt:condition} holds. Then
  \begin{equation*}
     \sqrt{N}\left(C_\alpha(\hat{\boldsymbol{\sigma}}) - C_\alpha(\boldsymbol{\sigma})\right) \underset{N\to\infty}{\Rightarrow} \mathcal{N}\left(0, \varsigma_{\alpha}^2\right)\;,\quad \varsigma_{\alpha}^2 := \langle \nabla C_{\alpha}(\boldsymbol{\sigma}), \Sigma_\alpha \nabla C_{\alpha}(\boldsymbol{\sigma})\rangle\;,
   \end{equation*}
   where
   $\Sigma_\alpha := \mathrm{diag}\left(\varsigma_{S,\alpha}^2\colon
     S\in\mathcal{S} \right)$ with
   \begin{equation*}
     \varsigma_{S,\alpha}^2 := \begin{cases} \frac{\sigma_S^2\left(\kappa_S-1\right)\langle \boldsymbol{p}, \boldsymbol{\sigma}\rangle}{4 p_S\left((1-\alpha)\langle \boldsymbol{p}, \boldsymbol{\sigma}\rangle + \alpha \sigma_S\right)}
     \;,& \text{if}\quad \sigma_S \not= 0\;,\\
     0\;,& \text{if}\quad \sigma_S = 0\;,
     \end{cases} 
  \end{equation*}
  and $\kappa_S:= \E{ {\lvert Q_S - \E{Q_S}\rvert}^4} / \sigma_{S}^4$
  denoting the kurtosis of $Q_S$ in stratum $S\in\mathcal{S}$.
\end{lemma}
\begin{proof}
  The claim is a consequence of the multivariate delta method, e.g.,
  \cite[Ch.~3.3]{Serfling_80} and the fact that the empirical standard
  deviation in each stratum satisfies a central
  limit theorem~\cite[Ex.~3.6]{DasGupta_08}, which is also proved using
  the delta method. Specifically, we have
  \begin{equation*}
    \sqrt{N_S^\alpha}\left(\hat{\sigma_S} - \sigma_S\right) \Rightarrow \mathcal{N}\left(0, \frac{\sigma_S^2\left(\kappa_S-1\right)}{4}\right)
  \end{equation*}
  as $N_S^\alpha\to \infty$, which is ensured under the
  hypothesis. Observing that the different estimators $\hat{\sigma_S}$
  are independent across strata completes the proof.
\end{proof}

\subsubsection{Variance reduction for fixed $N$ and uniform Cartesian
  stratification}

The previous discussion focused on the effects of the hybrid
allocation rule and the sample size $N$ for a given stratification. In
this section, we address the effects of the hybrid allocation rule
on the variance reduction obtained with respect to the regularity of
the function $f$ that gives $Q = f(\bY)$. We recall that this work is
in particular motivated by problems for which we expect $f$ to be
non-smooth. 
Below, we study
the variance reduction with respect to the size of a uniform Cartesian stratification
depending on the regularity of $f$, which is an
extension of the work in \cite[Chap.~V.7]{Asmussen_Glynn_07}.  In
particular, we will focus on two special cases of $f$ that will
guide the adaptive stratification procedure in the following.

\begin{lemma}\label{lemma:hybrid:strat:large-strat}
  Consider a uniform Cartesian stratification $\mathcal{S}$ of
  $\mathfrak{U}\subset\mathbb{R}^n$.  If $f\in C^1(\mathfrak{U})$,
  then the variance of the hybrid stratification estimator
  $\hat{Q}_{\alpha}$ satisfies
    \begin{equation*}
     V_{\alpha} \le \frac{n C}{3\, N}\, {\lvert \mathcal{S} \rvert}^{-2/n}
         \min\left(\frac{1}{\alpha}, \frac{1}{1-\alpha}\right)\le \frac{2 n C}{3\, N}\, {\lvert \mathcal{S} \rvert}^{-2/n}\;,
     \end{equation*}
     for any $\alpha\in[0,1]$, where
     $C = \sup_{\bu\in\mathfrak{U}}\Vert \nabla f(\bu)\Vert_2^2 \le
     \Vert f\Vert_{C^1(\mathfrak{U})}^2$.

     If $f\colon \mathfrak{U}\to\mathbb{R}$ is a piecewise
     constant function with a jump discontinuity of size $\delta>0$
     across a 
     curve $\Gamma$ in $\mathfrak{U}$,
     then the stratification estimator's variance is bounded by
     \begin{equation*}
       V_{\alpha} \le \frac{\delta^2}{4\,N}
       \begin{cases}
         \frac{\lvert \mathcal{T} \rvert}{\lvert \mathcal{S} \rvert}\;,& \alpha = 0\;,\\
         \frac{\lvert \mathcal{T}\rvert}{\lvert \mathcal{S}\rvert}\, \min\left\{\frac{1}{1-\alpha}, \frac{1}{\alpha}\, \frac{\lvert \mathcal{T}\rvert}{\lvert \mathcal{S}\rvert}\right\}\;,& 0<\alpha < 1 \;,\\
         {\left(\frac{\lvert \mathcal{T} \rvert}{\lvert \mathcal{S} \rvert}\right)}^2\;,& \alpha=1 \;,
       \end{cases}
     \end{equation*}
     where $\mathcal{T}\subset\mathcal{S}$ denotes the set of all strata
     that contain $\Gamma$.      
\end{lemma}
\begin{proof}
  For a uniform Cartesian stratification, we have
  $p_S = 1/\lvert \mathcal{S}\rvert$ for all $S\in\mathcal{S}$, so
  that the stratification estimator's variance can be written as 
  \begin{equation*}
    V_{\alpha}
    = \frac{ \Vert \boldsymbol{\sigma} \Vert_1}{N\, \lvert \mathcal{S}\rvert} \sum_{S\in\mathcal{S}}\frac{\sigma_S^2}{ (1-\alpha) \Vert \boldsymbol{\sigma} \Vert_1 + \alpha\sigma_S \lvert \mathcal{S}\rvert }\;,
  \end{equation*}
  where $\Vert \boldsymbol{\sigma} \Vert_1 = \sum_{S\in\mathcal{S}}\sigma_S$.
  We proceed by bounding the local variances $\sigma_S^2$ in each
  stratum $S\in\mathcal{S}$. To do so, we distinguish the two
  regularity cases of the function $f\colon\mathfrak{U}\to\mathbb{R}$.

We begin with the case $f\in C^1(\mathfrak{U})$. It follows
from a Taylor expansion of $f$ and the Cauchy--Schwarz inequality that
\begin{equation*}
\sigma_S^2 = \var{Q_S} = \var{f(\bY)| \bY\in S} \le \sup_{\bu\in\mathfrak{U}} \Vert \nabla f(\bu) \Vert_{2}^2 \, \E{\Vert \bY \Vert_{2}^2} {\lvert \mathcal{S} \rvert}^{-2/n}    \;.
\end{equation*}
Noting that $\E{\Vert \bY \Vert_{2}^2} = n/3$, we obtain the bound
$\sigma_S^2 \le n\,C\, 3^{-1} {\lvert \mathcal{S} \rvert}^{-2/n}$ for
all $S\in\mathcal{S}$, where
$C = \sup_{\bu\in\mathfrak{U}} \Vert \nabla f(\bu) \Vert_{2}^2$.  We
thus have
$\Vert \boldsymbol{\sigma} \Vert_1^2 \le \frac{n\, C}{3}\, {\lvert
  \mathcal{S} \rvert}^{2-2/n}$.  Consequently, for $\alpha\in\{0,1\}$ the
estimator's variances satisfy
\begin{equation*}
    V_{0}  
    = \frac{ 1}{N\, \lvert \mathcal{S}\rvert} \sum_{S\in\mathcal{S}}\sigma_S^2 \le \frac{ n\, C}{\,3\,N} {\lvert \mathcal{S} \rvert}^{-2/n}\quad\text{and}\quad
    V_{1}
    = \frac{ \Vert \boldsymbol{\sigma} \Vert_1^2}{N\, {\lvert \mathcal{S}\rvert}^2} \le \frac{n\, C}{3\, N}\,{\lvert \mathcal{S} \rvert}^{-2/n}\;,
  \end{equation*}
  respectively.  For the intermediate values, $0< \alpha<1$ we use
  inequality \eqref{eq:hybrid:var:bound:simple} together with the
  bounds above to complete the first part of the proof.
    
  Next, we consider the case of $f\colon \mathfrak{U}\to\mathbb{R}$
  being a piecewise constant function. In that case, we have
  $\sigma_S^2 = 0$ for all
  $S\in\mathcal{S}\setminus\mathcal{T}$. Conversely, for
  $S\in\mathcal{T}$ the quantity $Q_S$ is a random variable taking two distinct
  values with probabilities proportional to the sizes of the stratum
  subdivision by $\Gamma$.  Consequently,
  $\sigma_S^2 \le \delta^2/4$, where $\delta>0$ is the height of the
  jump discontinuity. For $\alpha\in\{0,1\}$ the estimator's variances
  therefore satisfy 
\begin{equation*}
    V_{0}  
    = \frac{ 1}{N\, \lvert \mathcal{S}\rvert} \sum_{S\in\mathcal{S}}\sigma_S^2 \le \frac{\delta^2 {\lvert \mathcal{T}\rvert} }{4\,N\, {\lvert \mathcal{S}\rvert} }\quad\text{and}\quad
    V_{1} 
    = \frac{ \Vert \boldsymbol{\sigma} \Vert_1^2}{N\, {\lvert \mathcal{S}\rvert}^2} \le \frac{\delta^2 {\lvert \mathcal{T}\rvert}^2 }{4\,N\, {\lvert \mathcal{S}\rvert}^2} \;,
  \end{equation*}
  respectively, and from \eqref{eq:hybrid:var:bound:simple} we
  conclude that
  \begin{equation*}
    V_{\alpha}  \le \frac{\delta^2 }{4\,N} \,\frac{{\lvert \mathcal{T}\rvert}} {\lvert \mathcal{S}\rvert}
    \min\left\{\frac{1}{1-\alpha}, \frac{1}{\alpha}\, \frac{\lvert \mathcal{T}\rvert}{\lvert \mathcal{S}\rvert}\right\}
  \end{equation*}
  for any $0< \alpha<1$, which completes the proof.
\end{proof}  

The preceding lemma indicates that for a differentiable function $f$,
optimal allocation is not superior to proportional allocation, nor to
any other hybrid allocation, asymptotically as the stratification is
refined. Conversely, for piecewise constant functions, optimal
allocation is superior to other allocation rules. However, hybrid
allocation with $\alpha\in(0,1]$ sufficiently large offers
\emph{quasi-optimal} asymptotic variance reduction as the
stratification is refined, that is, as
$\gamma = \frac{\lvert \mathcal{T}\rvert}{\lvert \mathcal{S}\rvert}<1$
decreases. Indeed, if $\alpha\in(0,1]$ satisfies
$\alpha \ge \frac{\gamma_0}{1+\gamma_0}$, for example
$\alpha=\gamma_0$, then $V_{\alpha}\lesssim \gamma^2$ for all
$\gamma\le \gamma_0$.
Finally, it is also noteworthy that
Lemma~\ref{lemma:hybrid:strat:large-strat} shows that stratified
sampling estimators face the ``curse of dimensionality'', in the sense
that the variance reduction effects deteriorate drastically as the
dimension $n$ increases. To mitigate this effect to some extent, we
will use an adaptive procedure for dynamically creating the
stratification tailored to function $f$ that exhiit localized
variability, e.g., discontinuities. This is described next.


\section{Adaptive Stratification}
\label{sec:adapt-strat-samp}

The optimal allocation of samples in each stratum is not known a
priori, since the stratum standard deviations are not known. Nor is
the optimal stratification of the stochastic domain known in advance,
which would minimize the stratified sampling estimator's variance;
cf.~Lemmas~\ref{lem:basic:strat:concentration} and
\ref{lemma:hybrid:strat:large-strat}. The first problem will be
alleviated similarly to the method proposed
in~\cite{Etore_Jourdain_10}, where the samples are allocated
iteratively to fixed strata where the local standard deviations are
computed to satisfy the rules~\eqref{eq:allo_rules:hybrid} in the
large sample size limit. To remedy the second problem, we propose a
method that uses local variance estimates and a greedy variance
reduction approach to determine how to adaptively split the domain
into new strata.

\subsection{Stratification algorithm}

For clarity, we provide a relatively high-level summary of the
proposed adaptive stratification procedure in
Algorithm~\ref{algorithm_ass_brief}, followed by more technical
descriptions of the individual parts in subsequent subsections. We
start from a stratification, which may be based on a priori knowledge
of the dependence of the quantity of interest $Q = f(\bY)$. If such
information does not exist, the initial stratification may consist of
a few pre-determined strata or just be a single stratum, i.e., the
entire stochastic domain.  We consecutively distribute
$N_{\text{new}}$ samples to the current stratification consisting of
$N_{\text{strata}}$ strata, chosen so that the resulting sample
distribution will be as close as possible to the hybrid allocation
rule~\eqref{eq:allo_rules:hybrid} for a given value of $\alpha$. It is
worth noting that depending on previous sample allocation and due to
integer rounding, it may not be possible to exactly satisfy the
theoretical sampling rates. Furthermore, the current stratification
may be quite different from the final stratification, and the
allocation of samples is at best optimal with respect to the current
situation. For each stratum, local properties are updated (e.g., local
means, standard deviations, and number of samples). Next, we use a
greedy approach to refine the stratification by splitting the stratum
(or strata) along some hyperplane that results in the largest variance
reduction within a finite candidate set of refined stratifications. As
an additional splitting criterion, we only consider strata that
contain a minimum number of samples for splitting to make sure that
the splitting decision is an informed one. These two steps are
repeated until an upper limit $N_{\text{max}}$ for the total number of
samples has been reached.

One may note that refinement cannot increase the variance of the
estimator, cf.~Appendix~\ref{sec:further:motivation:splitting}, but a
poor choice of splitting may require many subsequent splittings
to better adapt to the variability of the quantity of interest
$Q=f(\bY)$, and each refinement is only performed after new samples
have been distributed. Hence, for practical reasons of the proposed
method, splitting should be performed with some care.  The samples of
the split stratum are distributed to the new strata via (linear)
sorting. If the hybrid parameter $\alpha\in [0,1[$ is set to be
dynamic (see Sect.~\ref{sec:setting-alpha}), then it is updated before the
process of allocating new samples and splitting is repeated until the
limit of the computational budget is reached.
Algorithm~\ref{algorithm_ass_brief} describes the workflow in
pseudocode and contains references to the subsections where the
method's components are described in more detail.

\begin{algorithm}[H]
  \caption{Adaptive Stratified Sampling.}\label{algorithm_ass_brief}
  \begin{algorithmic}[1]
    \State Evaluate initial sample set $\{ Q^{(j)}\}_{j=1}^{N_{\textup{new}}}$ of size $N_{\textup{new}}$ using prop. allocation ($\alpha=0$). 
    \State $N\gets N_{\textup{new}}$. 
    \State Initialize coarse stratification into $N_{\textup{strata}}$ strata. (\emph{Sect.~\ref{sec:strat-geom}}) 
    \State Determine variance reduction effect from tentative splits.
    \While {$N \leq N_{\textup{max}}$}
    \State \% \emph{Refine stratification by splitting strata.}
    \State Find candidate stratum $S_{\textup{split}}$ and hyperplane to split. (\emph{Sect.~\ref{sec:split-planes}})
    \If{Splitting criteria satisfied}
    \State Split stratum $S_{\textup{split}}$ into two (equal size) strata:
    \State Redistribute samples in $S_{\textup{split}}$ to the new strata.
    \State Compute stratum parameters for the new strata, remove the parent stratum.
    \State $N_{\textup{strata}}\gets N_{\textup{strata}} + 1$.
    \EndIf
    \State \textbf{end}
    \State \% \emph{Add new samples according to current stratification.}
    \State Determine number of new samples $N_{\textup{new}}$ to allocate so that $N+N_{\text{new}}\le N_{\text{max}}$.
    \For{$S \in \mathcal{S}$} 
    \State Add $N_{\textup{new},S}$ new random samples to stratum $S$. (\emph{Sect.~\ref{sec:dist_samples}})
    \State Update stratum parameters. (\emph{Sect.~\ref{sec:upd-strat-stat}})
    \State Determine variance reduction effect from tentative splits. (\emph{Sect.~\ref{sec:greedy-var-red}})
    \EndFor
    \State \textbf{end}
    \State Update hybrid allocation parameter $\alpha$. (\emph{Sect.~\ref{sec:setting-alpha}})
    \State $N \gets N+N_{\textup{new}}$ 
    \EndWhile
    \State \textbf{end}
    \State Output: estimator $\hat{Q}_{\alpha}$ (Eq.~\eqref{eq:ss_estimator:hybrid}).
  \end{algorithmic} 
\end{algorithm}

\subsection{Variance minimizing splitting}
\label{sec:greedy-var-red}
Inaccurate estimates of the stratum standard deviations may lead to
lack of variance reduction of the standard stratification
estimator~\cite{Cochran_77}. Likewise, by symmetry of $\sigma_{S}$ and
$p_{S}$ in the case of optimal allocation, c.f.~\eqref{eq:allo_rules}
and \eqref{eq:allo_vars}, an inaccurate estimate of an unknown
measures $p_s$ may also result in poor performance. As a remedy, it is
natural to choose stratifications $\mathcal{S}$ with well-defined and
easy to compute measures $p_S$ for all $S\in\mathcal{S}$. An
additional attractive and thus desirable feature of an adaptive
stratification is that each new stratum that is created is a subset of
only one existing stratum. Otherwise, a new stratum might have regions
that have been sampled with different rates as determined by different
estimates $\hat{\sigma}_{S}$ which would require special treatment in
subsequent allocation of samples as a compensation for non-uniform
sampling.  This would be a problem for all hybrid allocation rules
with $\alpha>0$. In this work, we employ stratifications where the
stratum measures are easy to compute exactly, and new strata are
created by splitting an existing stratum into one or more subsets,
each of the same geometrical shape as the parent stratum. In
particular, in Sect.~\ref{sec:practical-stratification} we will
describe in detail stratifications based on hyperrectangles and
simplicies.


We use a greedy strategy for stratum splitting. The rationale for
greedy strategies is that while there is no way of predicting the
final stratification, the best we can do is to efficiently reduce the
variance based on the information given by the current stratification.
The basis for implementing this strategy is the
expression~\eqref{eq:ss_estimator:var:hybrid} for the variance of the
estimator for a fixed stratification with hybrid sample
allocation. 
We seek the splitting that leads to the maximum variance reduction
when one (or more) existing stratum is bisected, resulting in
a refined stratification.  Recalling the notation introduced in
Sect.~\ref{sec:basic-strat-samp:hyrbid-allocation}, the stratification
estimator's variance for a given stratification $\mathcal{S}$ and a
fixed total number of samples $N$ can be written as
  \begin{equation*}
    \frac{1}{N} \sum_{S\in\mathcal{S}}\frac{p_S \sigma_S^2}{1+\alpha\left(\bar{\sigma}_S(\mathcal{S}) - 1 \right)} =: V(\mathcal{S})\;,
  \end{equation*}
  where we have suppressed the dependence on the hybrid parameter
  $\alpha$, which is assumed to be fixed in this subsection.
  Suppose now that stratum $T\in\mathcal{S}$ is split into two
  equi-probable, disjoint strata $T^{+}$ and $T^{-}$ so that
  $T = {T^{-}}\cup {T^{+}}$ and $p_T = 2p_{T^{\pm}}$. Denote by
  $\mathcal{S}_{[T]}$ the refined stratification that is obtained by
  this splitting, that is
  $\mathcal{S}_{[T]} := (\mathcal{S}\setminus T) \cup T^{+} \cup
  T^{-}$. The variance of the corresponding stratification estimator
  based on the same total number of samples
  $ \sum_{S\in\mathcal{S}} N_S = N = \sum_{S\in\mathcal{S_{[T]}}} N_S$
  can thus be written as 
  \begin{equation*}
    \begin{aligned}
      V(\mathcal{S}_{[T]}) &= \frac{1}{N} \sum_{S\in\mathcal{S}_{[T]}}\frac{p_S \sigma_S^2}{1+\alpha\left(\bar{\sigma}_S(\mathcal{S}_{[T]}) - 1 \right)}\\
      &= \frac{1}{N} \left(\sum_{S\in\mathcal{S}\setminus T}\frac{p_S \sigma_S^2}{1+\alpha\left(\bar{\sigma}_S(\mathcal{S}_{[T]}) - 1 \right)} + \frac{p_T}{2}\sum_{S\in \left\{T^{+}, T^{-}\right\}}\frac{\sigma_S^2}{1+\alpha\left(\bar{\sigma}_S(\mathcal{S}_{[T]}) - 1 \right)} \right)\;.
    \end{aligned}
  \end{equation*}
  Consequently, the variance reduction obtained by the splitting
  $T = {T^{-}}\cup {T^{+}}$ compared to not splitting is
  \begin{equation*}
    \begin{aligned}
      N\left(V(\mathcal{S}) - V(\mathcal{S}_{[T]})\right) &= p_T\left(\frac{\sigma_T^2}{1+\alpha\left(\bar{\sigma}_T(\mathcal{S}) - 1 \right)} - \frac{1}{2}\sum_{S\in \left\{T^{+}, T^{-}\right\}}\frac{\sigma_S^2}{1+\alpha\left(\bar{\sigma}_S(\mathcal{S}_{[T]}) - 1 \right)}\right)\\
      &\qquad + \alpha \sum_{S\in\mathcal{S}\setminus T}  \frac{p_S \sigma_S^2\left(\bar{\sigma}_S(\mathcal{S}_{[T]}) - \bar{\sigma}_S(\mathcal{S}) \right)}{\left( 1+\alpha\left(\bar{\sigma}_S(\mathcal{S}) - 1 \right)\right)\left( 1+\alpha\left(\bar{\sigma}_S(\mathcal{S}_{[T]}) - 1 \right)\right)}\;.
    \end{aligned}
  \end{equation*}
  In the special case of proportional allocation ($\alpha = 0$), the expression simplifies to
  \begin{equation*}
      N\left(V(\mathcal{S}) - V(\mathcal{S}_{[T]})\right) = p_T\left(\sigma_T^2 - \frac{\sigma_{T^{-}}^2+\sigma_{T^{+}}^2}{2}\right)\;.
    \end{equation*}  
 and for optimal allocation,
     \begin{equation*}
      N\left(V(\mathcal{S}) - V(\mathcal{S}_{[T]})\right)
      = -p_T\Delta_T  \left(p_T\sigma_T + 2\sum_{S\in\mathcal{S}\setminus T}  p_S \sigma_S\right)  - p_T^2\Delta_T\frac{\sigma_{{T}^{-}} + \sigma_{{T}^{+}}}{2}\;,
  \end{equation*}
  respectively, where the notation
  $\Delta_T := \frac{\sigma_{T^{-}}+\sigma_{T^{+}}}{2} - \sigma_T$ has
  been used.
 
  The \emph{greedy} splitting strategy then entails finding the
  stratum $T\in\mathcal{S}$ that when split across a hyperplane with
  index $j$ (that corresponds to some ordering of a set of permissible splitting planes) provides the largest variance reduction, that is
\begin{equation}
\label{eq:greedy_var_split_mod}
(S,j)_{\text{split}} \in \argmax_{T,j} N\left(V(\mathcal{S}) - V(\mathcal{S}_{{[T]}_j})\right)\;,
\end{equation}
where $\mathcal{S}_{{[T]}_j}$ denotes the refined stratification when
the stratum $T$ is split across the hyperplane $j$ as determined by the shape of strata. For instance, if the stratification is defined by a Cartesian grid, the hyperplane will be perpendicular to the coordinate axis $j$.

\subsection{Probability of failure to identify a stratum that should be split}
\label{sec:adapt-strat-samp:failure-split}
For any stratum
$S\in\mathcal{S}$, recall that $\hat{\sigma}_S^2$ denotes the usual
empirical variance of $Q_S$ based on the $N_S\equiv N_S^\alpha$
available samples, which is an unbiased estimator of
$\sigma_S^2\equiv \var{Q_S}$. Notice that $\sigma_S^2 = 0$ implies
$\hat{\sigma}_S^2=0$ and, conversely, that $\hat{\sigma}_S^2>0$
implies $\sigma_S^2>0$.
The proposed adaptive splitting procedure's efficiency may deteriorate if the empirical variance
of a stratum drastically underestimates the true variance of
that stratum, so that identification and refinement of a high-variability stratum fails. 
This can happen, for example, if a stratum
does not contain enough samples, so that a discontinuity is not
observed given these samples. 
%
To quantify this ``failure'' of drastically underestimating
the local variance in a stratum $S$, suppose therefore that
$\sigma_S^2>0$. As this implies that $\hat{\sigma}_S^2\ge 0$, it
follows from the Paley--Zygmund inequality combined with the
Cauchy--Schwarz inequality that the probability of underestimating the
local variance of $Q_S$ is bounded by
\begin{equation}
  \mathbb{P}(\hat{\sigma}_S^2 \le \theta \sigma_S^2) \le 1-  \frac{(1-\theta)^2}{(1-\theta)^2 + N_{S}^{-1}\left( \kappa_S - \frac{N_{S}-3}{N_{S}-1}\right)} =
  \frac{\kappa_S - \frac{N_{S}-3}{N_{S}-1}}{N_{S} (1-\theta)^2 + \left( \kappa_S - \frac{N_{S}-3}{N_{S}-1}\right)}\;,\label{eq:failure:prob}
\end{equation}
for any $\theta\in [0,1]$, where $\kappa_S$ denotes the kurtosis of
$Q_S$. This inequality shows that the probability of underestimating
the variance of $Q_{S}$ by a factor $\theta$ is of order
$\mathcal{O}(1/N_{S})$. In fact, if the kurtosis $\kappa_S$ of $Q_S$
was known, one could use the upper bound to determine the required
sample size $N_{S}$ to guarantee that
$ \mathbb{P}(\hat{\sigma}_S^2 \le \theta \sigma_S^2)\le
p_{\text{crit}}$ for some prescribed tolerance
$p_{\text{crit}} \in(0,1)$ and $\theta\in [0,1]$.  However, the upper
bound also reveals that it is affected by the kurtosis $\kappa_S$ in
stratum $S$, in the sense that a higher kurtosis leads to a higher
probability of underestimating the variance.

An exemplary borderline case that is insightful in the context of the
proposed stratified sampling procedure for heterogeneous functions $f$
constitutes the situation where the stratum $S$ contains a region
$R\subset S$, $S\setminus R \not=\emptyset$, on which $f$ is constant,
and all available samples are contained in $R$. That is,
$\bY^{(1)},\dots, \bY^{(N)} \in R\subset S$ so that
$Q^{(i)}_S = f(\bY^{(i)}) \equiv f_R$ for all $i=1,\dots, N_{S}$. It
follows that $\hat{\sigma}_S^2=0$, even though $\sigma_S^2>0$, since
not the entire stratum has been sampled from and the region
$S\setminus R$ has been missed. Consequently, this can lead to a poor
performance of the adaptive procedure due to missing to identify a
high variance stratum which should have been subdivided. As a matter
of fact, minimizing the risk of falsely estimating a stratum variance
to be zero is particularly important when using
Algorithm~\ref{algorithm_ass_brief} with optimal sample allocation
($\alpha=1$), since such a drastic underestimation would not be
recovered and thus remain a performance bottleneck.

Although the kurtosis is known for some distributions, the upper bound
in \eqref{eq:failure:prob} is not of immediate use in practice where
the kurtosis is unknown.
%
In fact, even estimating the kurtosis empirically based on the
available samples will be undefined in that case.  A possible remedy
is to use a rule-of-thumb, such as assuming that $Q_S$ follows a
normal distribution ($\kappa=3$), a uniform distribution on some
interval ($\kappa = 9/5$), or some other distribution with known
kurtosis in stratum $S$; see Fig.~\ref{fig:failure:prob:examples} for
examples of the upper bound on the failure probability in
\eqref{eq:failure:prob} for two different distributions.
\begin{figure}[tb!]
  \centering
  \includegraphics[width=0.4\textwidth]{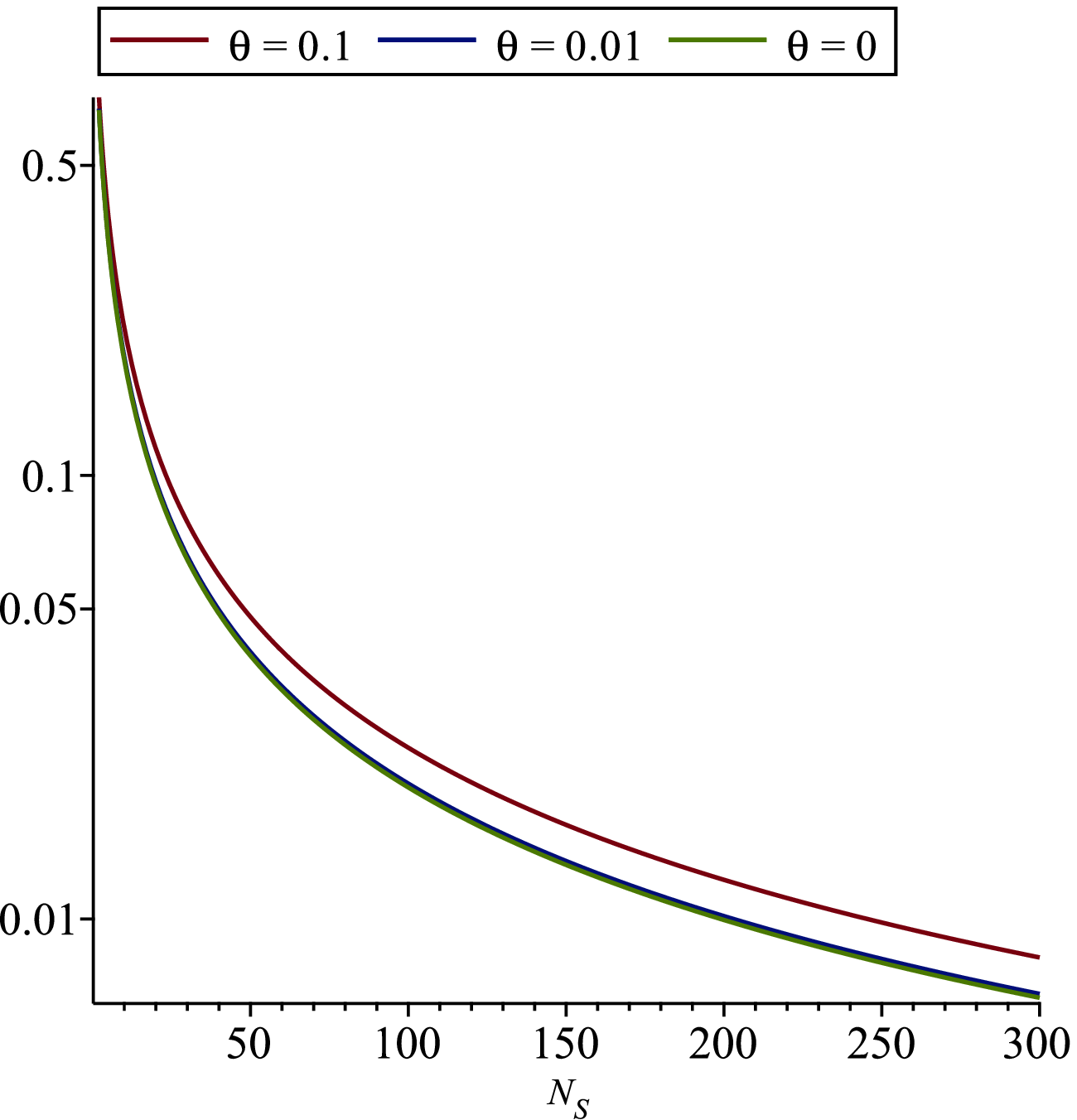}\quad
  \includegraphics[width=0.4\textwidth]{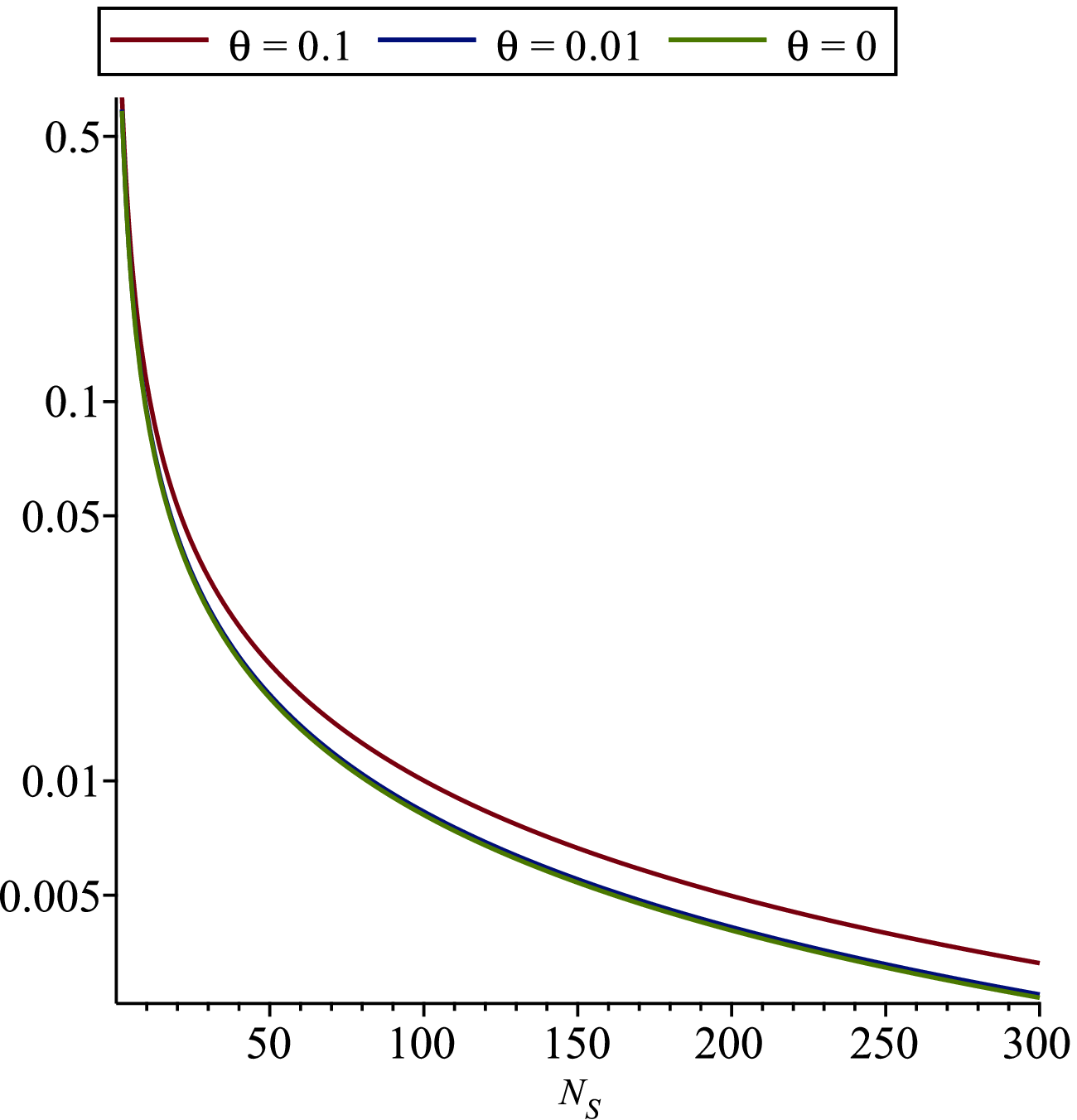}
  \caption{Upper bound of the failure probability for a Gaussian
    distribution (left) and a uniform distribution (right) based on
    $N_S$ samples.}
  \label{fig:failure:prob:examples}
\end{figure}

\subsubsection{Kernel-based estimation}
Alternatively, a sample-based approach with
moment approximations by means of a
kernel density estimator (KDE) could be useful. In fact, even in the
special case when obtaining only samples from the region $R$ where
$f$ is constant, this KDE approach is possible and yields,
interestingly, the Gaussian kurtosis rule-of-thumb approximation as we
will see below.

In general, the KDE-based moments are given by
\[
m_{k,\scriptscriptstyle{\textup{KDE}}}(Q_{S}) = \frac{1}{\delta N_S} \sum_{i=1}^{N_S}\int q^k K\left(\frac{q-Q^{(i)}_S}{\delta}\right)\,\textup{d}q\;,
\]
where $K$ is a symmetric probability density function on $\mathbb{R}$,
e.g., that of a standard Normal distribution, and $\delta>0$ denotes
the bandwidth, which controls the smoothing
\cite[Chap.~8.5]{Kroese_etal_11}. For example, this KDE-moment
approximation with a Gaussian kernel yields
\begin{equation}
\begin{aligned}
m_{1,\scriptscriptstyle{\textup{KDE}}}(Q_{S}) &= \frac{1}{N_{S}} \sum_{i=1}^{N_{S}}Q_{S}^{(i)}\;, \\ 
m_{2,\scriptscriptstyle{\textup{KDE}}}(Q_{S}) &= \frac{1}{N_{S}} \sum_{i=1}^{N_{S}}(Q_{S}^{(i)})^2 + \delta^2, \\ 
m_{3,\scriptscriptstyle{\textup{KDE}}}(Q_{S}) &= \frac{1}{N_{S}}\sum_{i=1}^{N_{S}}(Q_{S}^{(i)})^{3} + 3\delta^2\frac{1}{N_{S}}\sum_{i=1}^{N_{S}}Q_{S}^{(i)}, \\
m_{4,\scriptscriptstyle{\textup{KDE}}}(Q_{S}) &= \frac{1}{N_{S}}\sum_{i=1}^{N_{S}}(Q_{S}^{(i)})^{4} + 6\delta^2 \frac{1}{N_{S}}\sum_{i=1}^{N_{S}}(Q_{S}^{(i)})^2  + 3\delta^4\;,
\end{aligned}
\label{eq:KDE_moments}
\end{equation}
for any $\delta>0$. That is, except for the first moment, the KDE
moments differ from the empirical moments computed directly from
samples by a term $\mathcal{O}(\delta^2)$, which indicates the
additional smoothing. Consequently, the KDE-based kurtosis
approximation $\kappa_{\scriptscriptstyle{\textup{KDE}}}(Q_{S})$ of
$\kappa_S$, that is
\begin{equation*}
  \kappa_{\scriptscriptstyle{\textup{KDE}}}(Q_{S}) := \frac{m_{4,\scriptscriptstyle{\textup{KDE}}}(Q_S) - 4m_{1,\scriptscriptstyle{\textup{KDE}}}(Q_S) m_{3,\scriptscriptstyle{\textup{KDE}}}(Q_S) + 6 m_{1,\scriptscriptstyle{\textup{KDE}}}(Q_S)^2m_{2,\scriptscriptstyle{\textup{KDE}}}(Q_S) - 3 m_{1,\scriptscriptstyle{\textup{KDE}}}(Q_S)^4}{{\left(m_{2,\scriptscriptstyle{\textup{KDE}}}(Q_S) - m_{1,\scriptscriptstyle{\textup{KDE}}}(Q_S)^2\right)}^2}\;,
\end{equation*}
is always well-defined. In fact, even in the borderline case with
$Q_{S}^{(1)} = \dots = Q_{S}^{(N_{S})} = f_R$. Indeed,
then we find
\begin{equation*}
  \begin{aligned}
  m_{1,\scriptscriptstyle{\textup{KDE}}}(Q_{S}) &= f_R\;,\quad m_{2,\scriptscriptstyle{\textup{KDE}}}(Q_S) = {f_R}^2 + \delta^2\;,\\
  m_{3,\scriptscriptstyle{\textup{KDE}}}(Q_{S}) &= {f_R}^3 + 3\delta^2 f_R\;, \quad m_{4,\scriptscriptstyle{\textup{KDE}}}(Q_S) = {f_R}^4 + 6\delta^2 {f_R}^2 + 3\delta^4\;,
\end{aligned}
\end{equation*}
which yields $\kappa_{\scriptscriptstyle{\textup{KDE}}}(Q_{S}) = 3$
for any bandwidth $\delta>0$.


\section{Computational aspects}
\label{sec:practical-stratification}

In the previous section~\ref{sec:adapt-strat-samp}, we have described
the generic, conceptual aspects of the adaptive stratified sampling
procure introduced in this work. Here, we will complement the abstract
algorithmic component with implementation-specific details, including
the discussion of two concrete classes of stratifications.

\subsection{Stratification geometry}
\label{sec:strat-geom}
For a practical implementation of Algorithm~\ref{algorithm_ass_brief},
it is desirable that all strata have the same geometric shape so that
their volumes, i.e., probability measures $p_S$ for $S\in\mathcal{S}$,
can be easily and exactly computed. When a stratification is refined,
the new strata should therefore also maintain the same geometric shape
as the parent stratum. Since the sampling rates vary between the
strata, any new strata that inherits existing samples should be
properly contained within an existing stratum, i.e., it should be a
member of a partition of a larger stratum.

\subsubsection{Hyperrectangular stratification}
Adaptively splitting the stochastic domain into hyperrectangles
provides a geometry of strata that is easy to visualize (in lower
dimensions) and with properties that have clear generalizations in
multiple dimensions. In fact, both computing the volumes of
hyperrectangles and uniform sampling in hyperrectangles is
straightforward.
To maintain the number of strata at feasible levels, we
bisect a stratum into two substrata, i.e., split across a single dimension
rather than across all dimensions simultaneously. In addition, we split only one
stratum at the time. We emphasize that other splitting strategies may
be desirable and are possible within the proposed framework with very
minor modifications. Hyperrectangles are expected to be very efficient
for stratifications of problems where sharp features are nearly
parallel with the coordinate axes in the random domain.  The numerical
cost of performing the stratification refinement by greedy maximum
variance reduction described in Sect.~\ref{sec:greedy-var-red} is
dominated by sorting the existing samples according to them being
larger or smaller than a tentative stratum boundary, performed stratum
by stratum and dimension by dimension. Note that the samples do not
need to be fully sorted. Hence, the numerical cost assuming splitting
after distribution of
$N_{\text{new}} = \sum_{S \in \mathcal{S}} N_{\text{new},S}$ new
samples is proportional to
$\sum_{S \in \mathcal{S}} n N_{\text{new},S} = n N_{\text{new}}$.

\subsubsection{Simplex stratification}
A stratification based on simplices offers more flexibility compared
to a hyperrectangular stratification where all strata boundaries must
be aligned with the coordinate axes in stochastic space. Drawing
uniform samples in an arbitrary simplex can also be done
straightforwardly using standard methodologies, see, e.g.,
\cite[Ch.~3.3.2]{Kroese_etal_11}. Since the full stochastic domain of
interest $\mathfrak{U}$ is the unit hypercube in $n$ dimensions, it is
clearly possible to form a (trivial) hyperrectangular stratification
with just a single stratum. The smallest number of simplices required
to tessellate the same hypercube is $n!$. Hence, a full
$n$-dimensional simplex stratification is limited to small $n$, but we
remark that an approach where simplex stratification is only applied
to a few random dimensions may remedy this issue if the remaining
dimensions are stratified using hyperrectangles, or not at all.

To perform the greedy variance reduction stratification refinement,
each sample in the input space is transformed to barycentric
coordinates, which involves solving a linear system of size
$n\times n$. Alternatively, for $n$ relatively small compared to the
number of samples, the transformation can be carried out with
matrix-vector multiplication for each sample. The total cost is then
the inversion of two matrices of size $n \times n$ after splitting,
and $N_{\text{new}}$ matrix-vector multiplications for transformation
of each new sample to barycentric coordinates. Then, for each of the
$n(n+1)/2$ tentative splitting planes of each simplex, the
$N_{\text{new}}$ samples are labeled as belonging to one of two
possible substrata, which amounts to comparing the two barycentric
coordinate entries corresponding to the two vertices that would be
assigned to different strata in case of partitioning.

\subsubsection{Initialization of the simplex stratification}
The coarsest simplex stratification of the hypercube
$\mathfrak{U}\subset\mathbb{R}^n$, consisting of $n!$ elements, is not
unique. A suboptimal initial stratification may lead to poor variance
reduction even after subsequent dynamic splitting based on refined
estimates of the stratum variances. In particular, if $n$ is large, one
may need to perform numerous simplex splits to compensate for
a suboptimal initial stratification. As a remedy, the stratification is
initialized by choosing the stratification that minimizes the
estimator's sample variance, given an initial set of Monte Carlo
samples in the hypercube. Whereas the task of stratified sampling is
otherwise to add new samples given an existing stratification, at this
initial stage we instead seek an optimal stratification given a
small set of solution samples. To generate the set of possible simplex
tessellations, we proceed as follows. Kuhn's decomposition is used to
divide the $n$-hypercube into $n$-simplices~\cite{Kuhn_60}. An
implementation of Kuhn's decomposition, which has also been employed
in this work, is described in~\cite{Cuvelier_Scarella_18}. Assuming, as
before, that the domain is given by the unit hypercube in
$\mathbb{R}^{n}$, Kuhn's decomposition has the property that
$(0,\dots, 0)$ and $(1,\dots,1)$ are common vertices of all simplex
elements. The stratification is thus characterized by the unique edge
crossing all $n$ dimensions. We generate the set of possible simplex
stratifications by rotating the Kuhn decomposition so that we get all
combinations of diagonal edges, i.e., the $2^{n-1}$ ways to choose two
vertices $\bv_{\textup{start}}, \bv_{\textup{end}}$ with the
properties that $\bv_{\textup{start},i} = 1 - \bv_{\textup{end},i}$
for $i=1,\dots, n$. Figure~\ref{fig:init_simplex_tess} shows the
different initial simplex stratifications for $n=3$.

\begin{figure}[h]
\centering
\subfigure 
{\includegraphics[width=0.225\textwidth]{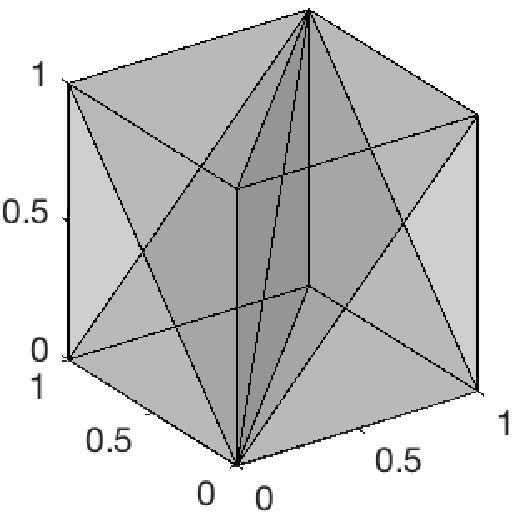}
\label{fig:init_tess_1}
}
~
\subfigure
{\includegraphics[width=0.225\textwidth]{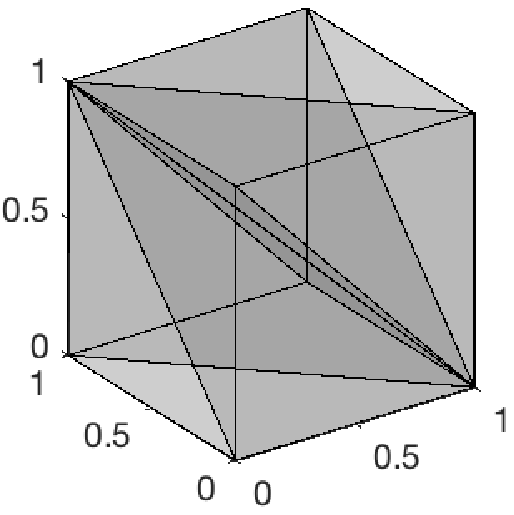}
\label{fig:init_tess_2}
}
~
\subfigure
{\includegraphics[width=0.225\textwidth]{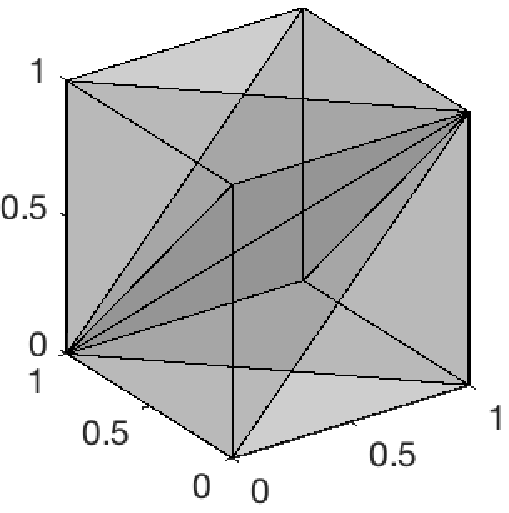}
\label{fig:init_tess_3}
}
~
\subfigure
{\includegraphics[width=0.225\textwidth]{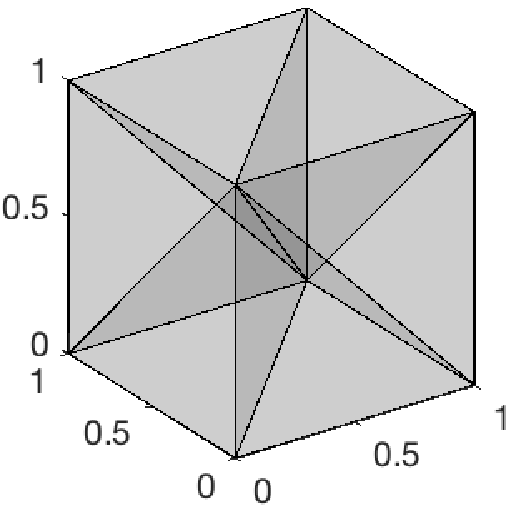}
\label{fig:init_tess_4}
}
\caption{Initial simplex stratifications of the hypercube in $n=3$ based on Kuhn's decomposition.}
\label{fig:init_simplex_tess}
\end{figure}

For each candidate stratification, the existing initial Monte Carlo
samples are distributed to the simplices through a transformation
to barycentric coordinates, and local statistics are computed. The
minimum estimator variance stratification among the $2^n$ candidates
is chosen as the initial stratification of the hypercube.

\subsection{Identification of bisection planes for splitting}
\label{sec:split-planes}

A stratum can theoretically be split along any one of an uncountable
number of planes, but finding a nearly optimal plane may require an
unnecessary amount of work, especially since we expect more
information through new samples to change the shape of the optimal
stratification. During stratification refinement, strata are instead
split to obtain the greatest variance reduction among a finite number
of possibilities. To limit the search space of possible new strata, we
consider only all possible bisections into two new members of the
stratum class (here: hyperrectangles or simplices) of all strata
and choose the minimizer as in \eqref{eq:greedy_var_split_mod}. Next
we describe how this is done in the case of hyperrectangular and
simplex stratifications, but note that other options such as splitting
more than one stratum at once could also be implemented without major
changes to the algorithm.

With a hyper-rectangular stratification, splitting must remain aligned
with the coordinate system.  Each time we split a stratum across
$n_{\textup{split}}$ dimensions, the number of strata is increased by
$2^{n_{\textup{split}}}-1$. This is not a severe limitation on the
choice of $n_{\textup{split}}$, but we do require at least one sample
in each stratum, so $n_{\textup{split}}$ should not be too large. With
$n_{\textup{split}}=1$, which was also used for the numerical
experiments discussed in Sect.~\ref{sec:num-res}, there are $n$
possible ways to split a hyperrectangle by bisection into two new
hyperrectangles, as illustrated for $n=3$ in
Fig.~\ref{fig:hyperrec_bisect}.

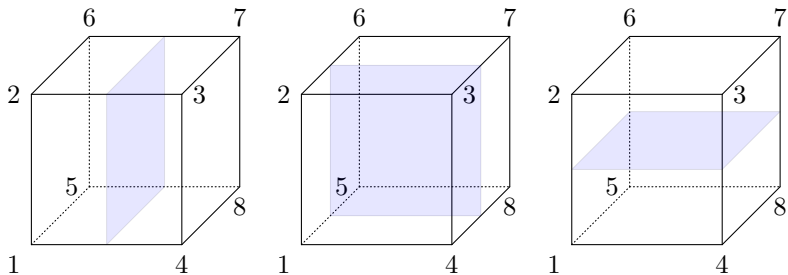
\begin{figure}[H]  
\centering  

\subfigure  
{%

\begin{tikzpicture}[line join = round, line cap = round]

\coordinate [label=left:\footnotesize{5}] (V1) at (0,0,0);
\coordinate [label=above:\footnotesize{6}] (V2) at (0,2,0);
\coordinate [label=left:\footnotesize{2}] (V3) at (0,2,2);
\coordinate [label=below left:\footnotesize{1}] (V4) at (0,0,2);

\coordinate [label=below:\footnotesize{8}] (V5) at (2,0,0);
\coordinate [label=above:\footnotesize{7}] (V6) at (2,2,0);
\coordinate [label=right:\footnotesize{3}] (V7) at (2,2,2);
\coordinate [label=below:\footnotesize{4}] (V8) at (2,0,2);

\draw[densely dotted] (V1)--(V2);
\draw (V2)--(V3);
\draw (V3)--(V4);
\draw[densely dotted] (V4)--(V1);

\draw (V5)--(V6);
\draw (V6)--(V7);
\draw (V7)--(V8);
\draw (V8)--(V5);

\draw (V4)--(V8);
\draw (V3)--(V7);
\draw[densely dotted] (V1)--(V5);
\draw (V2)--(V6);

\draw[fill=blue,opacity=0.1] ($0.5*(V4)+0.5*(V8)$) -- ($0.5*(V3)+0.5*(V7)$) -- ($0.5*(V2)+0.5*(V6)$)-- ($0.5*(V1)+0.5*(V5)$) -- cycle;

\end{tikzpicture}

}%
\subfigure  
{%

\begin{tikzpicture}[line join = round, line cap = round]

\coordinate [label=left:\footnotesize{5}] (V1) at (0,0,0);
\coordinate [label=above:\footnotesize{6}] (V2) at (0,2,0);
\coordinate [label=left:\footnotesize{2}] (V3) at (0,2,2);
\coordinate [label=below left:\footnotesize{1}] (V4) at (0,0,2);

\coordinate [label=below:\footnotesize{8}] (V5) at (2,0,0);
\coordinate [label=above:\footnotesize{7}] (V6) at (2,2,0);
\coordinate [label=right:\footnotesize{3}] (V7) at (2,2,2);
\coordinate [label=below:\footnotesize{4}] (V8) at (2,0,2);

\draw[densely dotted] (V1)--(V2);
\draw (V2)--(V3);
\draw (V3)--(V4);
\draw[densely dotted] (V4)--(V1);

\draw (V5)--(V6);
\draw (V6)--(V7);
\draw (V7)--(V8);
\draw (V8)--(V5);

\draw (V4)--(V8);
\draw (V3)--(V7);
\draw[densely dotted] (V1)--(V5);
\draw (V2)--(V6);

\draw[fill=blue,opacity=0.1] ($0.5*(V2)+0.5*(V3)$) -- ($0.5*(V6)+0.5*(V7)$) -- ($0.5*(V5)+0.5*(V8)$)-- ($0.5*(V1)+0.5*(V4)$) -- cycle;

\end{tikzpicture}

}%
\subfigure  
{%

\begin{tikzpicture}[line join = round, line cap = round]

\coordinate [label=left:\footnotesize{5}] (V1) at (0,0,0);
\coordinate [label=above:\footnotesize{6}] (V2) at (0,2,0);
\coordinate [label=left:\footnotesize{2}] (V3) at (0,2,2);
\coordinate [label=below left:\footnotesize{1}] (V4) at (0,0,2);

\coordinate [label=below:\footnotesize{8}] (V5) at (2,0,0);
\coordinate [label=above:\footnotesize{7}] (V6) at (2,2,0);
\coordinate [label=right:\footnotesize{3}] (V7) at (2,2,2);
\coordinate [label=below:\footnotesize{4}] (V8) at (2,0,2);

\draw[densely dotted] (V1)--(V2);
\draw (V2)--(V3);
\draw (V3)--(V4);
\draw[densely dotted] (V4)--(V1);

\draw (V5)--(V6);
\draw (V6)--(V7);
\draw (V7)--(V8);
\draw (V8)--(V5);

\draw (V4)--(V8);
\draw (V3)--(V7);
\draw[densely dotted] (V1)--(V5);
\draw (V2)--(V6);

\draw[fill=blue,opacity=0.1] ($0.5*(V3)+0.5*(V4)$) -- ($0.5*(V1)+0.5*(V2)$) -- ($0.5*(V5)+0.5*(V6)$)-- ($0.5*(V7)+0.5*(V8)$) -- cycle;

\end{tikzpicture}

}%

\caption{\label{fig:hyperrec_bisect} Splitting of 3D hyperrectangle by bisection parallel to the coordinate axes.}
\end{figure}

Simplices can be split in multiple ways to result in a refined simplex
partitioning. Perhaps the simplest way to avoid excessive growth in
the number of strata is to split a candidate simplex in two
equal-sized simplices defined by adding the hyperplane going through
the mid-point of one edge and the remaining vertices that are not the
end points of the edge whose midpoint is used. As a simplex has $n+1$
edges, and each one corresponds to a possible bisection, there are
$(n+1)n/2$ ways to split a simplex using this method. This is
illustrated for $n=3$ in Fig.~\ref{fig:simplex_bisect}, with six
different ways to bisect the simplex, each producing two new
simplices. Clearly, a larger number of conditional variances for each
tentative bisection must be computed, compared to the case of
hyperrectangular dynamic stratification. Bisection leads to very
simple computation of the probability measure of new strata: they are
just 1/2 of the measure of their parent stratum. A straightforward
generalization of the above described bisection method is to split
along a plane where the split point of the edge is not a midpoint. In
that case, the volume can be computed by
$1/n! \left| \det(\bv_2-\bv_1, \bv_3-\bv_1, \dots, \bv_{n+1}-\bv_{1})
\right|$, where $\{ \bv_{i}\}_{i=1}^{n+1}$ is the set of vertices.

\begin{figure}[H]  
\centering  

\subfigure
{%
\begin{tikzpicture}[line join = round, line cap = round]

\coordinate [label=above:\footnotesize{3}] (V3) at (0,{sqrt(2)},0);
\coordinate [label=left:\footnotesize{2}] (V2) at ({-.5*sqrt(3)},0,-.5);
\coordinate [label=below:\footnotesize{1}] (V1) at (0,0,1);
\coordinate [label=right:\footnotesize{0}] (V0) at ({.5*sqrt(3)},0,-.5);

\draw (V1)--(V0);
\draw (V1)--(V2);
\draw (V2)--(V3);
\draw (V1)--(V3);
\draw (V0)--(V3);
\draw[densely dotted] (V0)--(V2);

\draw[fill=blue,opacity=0.1] ($0.5*(V2)+0.5*(V3)$) -- (V0) -- (V1)-- ($0.5*(V2)+0.5*(V3)$) -- cycle;
\end{tikzpicture}
}%
\subfigure
{%
\begin{tikzpicture}[line join = round, line cap = round]

\coordinate [label=above:\footnotesize{3}] (V3) at (0,{sqrt(2)},0);
\coordinate [label=left:\footnotesize{2}] (V2) at ({-.5*sqrt(3)},0,-.5);
\coordinate [label=below:\footnotesize{1}] (V1) at (0,0,1);
\coordinate [label=right:\footnotesize{0}] (V0) at ({.5*sqrt(3)},0,-.5);

\draw (V1)--(V0);
\draw (V1)--(V2);
\draw (V2)--(V3);
\draw (V1)--(V3);
\draw (V0)--(V3);
\draw[densely dotted] (V0)--(V2);

\draw[fill=blue,opacity=0.1] ($0.5*(V1)+0.5*(V3)$) -- (V0) -- (V2)-- ($0.5*(V1)+0.5*(V3)$) -- cycle;
\end{tikzpicture}
}%
\subfigure
{%
\begin{tikzpicture}[line join = round, line cap = round]

\coordinate [label=above:\footnotesize{3}] (V3) at (0,{sqrt(2)},0);
\coordinate [label=left:\footnotesize{2}] (V2) at ({-.5*sqrt(3)},0,-.5);
\coordinate [label=below:\footnotesize{1}] (V1) at (0,0,1);
\coordinate [label=right:\footnotesize{0}] (V0) at ({.5*sqrt(3)},0,-.5);

\draw (V1)--(V0);
\draw (V1)--(V2);
\draw (V2)--(V3);
\draw (V1)--(V3);
\draw (V0)--(V3);
\draw[densely dotted] (V0)--(V2);

\draw[fill=blue,opacity=0.1] ($0.5*(V0)+0.5*(V3)$) -- (V1) -- (V2)-- ($0.5*(V0)+0.5*(V3)$) -- cycle;
\end{tikzpicture}
}%
\subfigure
{%
\begin{tikzpicture}[line join = round, line cap = round]

\coordinate [label=above:\footnotesize{3}] (V3) at (0,{sqrt(2)},0);
\coordinate [label=left:\footnotesize{2}] (V2) at ({-.5*sqrt(3)},0,-.5);
\coordinate [label=below:\footnotesize{1}] (V1) at (0,0,1);
\coordinate [label=right:\footnotesize{0}] (V0) at ({.5*sqrt(3)},0,-.5);

\draw (V1)--(V0);
\draw (V1)--(V2);
\draw (V2)--(V3);
\draw (V1)--(V3);
\draw (V0)--(V3);
\draw[densely dotted] (V0)--(V2);

\draw[fill=blue,opacity=0.1] ($0.5*(V1)+0.5*(V2)$) -- (V0) -- (V3)-- ($0.5*(V1)+0.5*(V2)$) -- cycle;
\end{tikzpicture}
}%
\subfigure
{%
\begin{tikzpicture}[line join = round, line cap = round]

\coordinate [label=above:\footnotesize{3}] (V3) at (0,{sqrt(2)},0);
\coordinate [label=left:\footnotesize{2}] (V2) at ({-.5*sqrt(3)},0,-.5);
\coordinate [label=below:\footnotesize{1}] (V1) at (0,0,1);
\coordinate [label=right:\footnotesize{0}] (V0) at ({.5*sqrt(3)},0,-.5);

\draw (V1)--(V0);
\draw (V1)--(V2);
\draw (V2)--(V3);
\draw (V1)--(V3);
\draw (V0)--(V3);
\draw[densely dotted] (V0)--(V2);

\draw[fill=blue,opacity=0.1] ($0.5*(V0)+0.5*(V1)$) -- (V2) -- (V3)-- ($0.5*(V0)+0.5*(V1)$) -- cycle;
\end{tikzpicture}
}%
\subfigure
{%
\begin{tikzpicture}[line join = round, line cap = round]

\coordinate [label=above:\footnotesize{3}] (V3) at (0,{sqrt(2)},0);
\coordinate [label=left:\footnotesize{2}] (V2) at ({-.5*sqrt(3)},0,-.5);
\coordinate [label=below:\footnotesize{1}] (V1) at (0,0,1);
\coordinate [label=right:\footnotesize{0}] (V0) at ({.5*sqrt(3)},0,-.5);

\draw (V1)--(V0);
\draw (V1)--(V2);
\draw (V2)--(V3);
\draw (V1)--(V3);
\draw (V0)--(V3);
\draw[densely dotted] (V0)--(V2);

\draw[fill=blue,opacity=0.1] ($0.5*(V0)+0.5*(V2)$) -- (V1) -- (V3)-- ($0.5*(V0)+0.5*(V2)$) -- cycle;
\end{tikzpicture}
}%
\caption{\label{fig:simplex_bisect} Bi-section of 3D simplex along the hyperplane that intersects the midpoint of an edge and all vertices that do not belong to that edge.}
\end{figure}
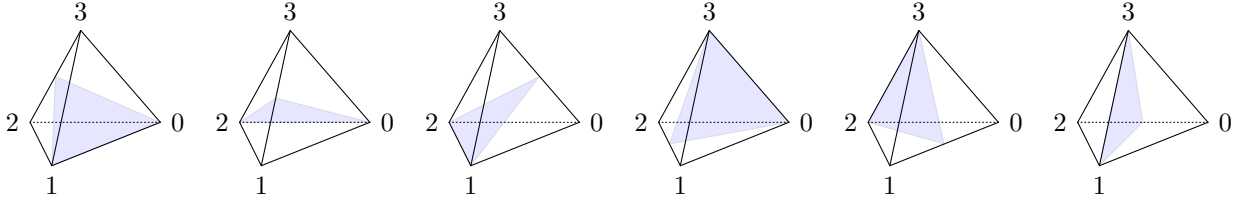


\subsection{Sequential allocation of samples}
\label{sec:dist_samples}
The stratification is dynamically updated by bisection (or other means
of splitting) of existing strata. We refer to a cycle of bisection and
subsequent sampling as an \emph{iteration} of
Algorithm~\ref{algorithm_ass_brief}.  At every iteration and for a
given value $\alpha$, $N_{\text{new}} = \sum_{S} N_{\text{new},S}$
samples are newly distributed and added to the current stratification
for an asymptotic (i.e., for $N_{\max}$ sufficiently large) sample
allocation according to the rates
\begin{equation*}
  q_{S}^{\alpha} := (1-\alpha)p_S + \alpha \bar{\sigma}_S p_S\;.
\end{equation*}
As the hybrid parameter $\alpha\in [0,1]$ is fixed, we will suppress
the explicit dependence and simply write $q_S\equiv
q_{S}^{\alpha}$. Adding these $N_{\text{new}}$ samples at the current
iteration to the current stratification then amounts to adding
$N_{\text{new}, S}$ new samples to stratum $S$ satisfying
\[
N_{\text{new},S} = \max(0, \min( \lceil  ( N_{\text{total}} + N_{\text{new}}) q_{S} - N_{S} \rceil, N_{\text{new}})),
\]
i.e., striving to fulfill
$N_{S}= q_{S} (N_{\text{total}}+ N_{\text{new}})$ as in
\eqref{eq:allo_rules:hybrid}, assuming a fixed stratification.
Here $\lceil x \rceil$ denotes the smallest integer greater than or
equal to $x$.
Even for hybrid sampling rules where a fraction of samples is
proportionally allocated (i.e., $\alpha<1$), some strata may not be
assigned new samples in every iteration. Eventually, all strata will
receive new samples, but for strata where variability has been
underestimated this can require several iterations and may not happen
before the total sampling budget of $N_{\text{max}}$ samples has been
reached. To avoid that a stratum where the standard deviation has
mistakenly been assigned to be zero does not get updated, we may
reserve at least one sample for every stratum, independent of the
sampling rates. Hence, we may alternatively use the allocation rule
\[
N_{\text{new},S} = 1+\max(0, \min(\lceil ( N_{\text{total}} + N_{\text{new}} -N_{\text{strata}}) q_{S}-N_{S} \rceil, N_{\text{new}}-N_{\text{strata}})).
\]
In case at least one stratum has been oversampled, e.g., as a result of
a stratum being created from a parent stratum with larger variance,
both sampling rules require a total number of samples exceeding
$N_{\text{new}}$. In these cases, we add samples according to the
allocation rule as long as the total sampling budget is not exceeded.

Since the stratification is dynamic and changes with (almost) every
iteration, there is a trade-off between choosing $N_{\text{new}}$ large
enough to approximately satisfy the rates $q_S$, and small enough to
save samples for further refinement of the stratification. A
reasonable compromise may be to always add a constant $c$ times the
current number of strata so that every stratum gets \emph{on average}
$c$ new samples in each iteration,
 \begin{equation}
 \label{eq:dist_samples}
N_{\text{new}} = c N_{\text{strata}}. 
 \end{equation}
 We emphasize that the distribution of samples is not restricted to
 the choice~\eqref{eq:dist_samples}, and other choices may
 significantly impact the performance of the algorithm. At the same
 time, the results in Sect.~\ref{sec:adapt-strat-samp:failure-split}
 offer a convenient probabilistic interpretation of the value for $c$
 in \eqref{eq:dist_samples} as the (stratification averaged) likelihood
 of drastically underestimating a local stratum variance.

 \subsubsection{Updating stratum statistics}
 \label{sec:upd-strat-stat}
 Updating the statistics of each stratum in every iteration may become
 quite costly for large sample sizes, in particular if only a small
 number of samples are added at every iteration so that many
 iterations are performed. Instead of recomputing the sample means and
 standard deviations from scratch using the standard formulae, we use
 the following update formulae
\begin{align}
\hat{\mu}_{S} &= \frac{N_{\text{old}, S} \hat{\mu}_{\text{old}, S} + N_{\text{new},S} \hat{\mu}_{\text{new},S}}{N_{\text{old},S}+N_{\text{new},S}}, \\            
\hat{\sigma}^2_{S} &= \frac{(N_{\text{old},S}-1) \hat{\sigma}_{\text{old},S}^2 + (N_{\text{new},S}-1) \hat{\sigma}_{\text{new},S}^2 + \frac{N_{\text{old},S} N_{\text{new},S}}{N_{\text{old},S}+N_{\text{new},S}} (\hat{\mu}_{\text{old},S}-\hat{\mu}_{\text{new},S})^2}{N_{\text{old},S}+N_{\text{new},S}-1}\;,   
\end{align}
which follows the ideas of Welford's online algorithm
\cite{Welford_1962}. Here, the subscripts 'old' and 'new' denote two
independent sample sets of the same quantity. These expressions are
not only used for the stratum means and standard deviations, but also
to update the corresponding quantities for potential splits, i.e.,
possible new strata.

\subsection{Setting the hybrid sampling parameter $\alpha$}
\label{sec:setting-alpha}

The choice of the hybrid parameter $\alpha\in[0,1]$, which controls
the sampling allocation rule, is affected by two opposing trends:
robustness and variance reduction. In
Sect.~\ref{sec:adapt-strat-samp:failure-split} we have discussed that
optimal allocation with $\alpha=1$ can lead to poor performance of the
adaptive procedure described in Algorithm~\ref{algorithm_ass_brief}
whenever the variances in high-variance strata are drastically
underestimated.  This is because the iterative procedure in
Algorithm~\ref{algorithm_ass_brief} with $\alpha=1$ may never (e.g.,
in the case of zero variance estimates) add new samples to these
strata so that the underestimation cannot be recovered, not even
asymptotically. In that sense, the choice $\alpha = 1$ lacks
robustness. In contrast, proportional allocation (i.e., $\alpha=0$) is
the most robust in that regard as it does not rely on estimates of the
strata standard deviations; cf.~also
Lemma~\ref{lem:basic:strat:concentration}. On the other hand, the
results in Sect.~\ref{sec:basic-strat-samp:hyrbid-allocation} indicate
that optimal allocation with $\alpha=1$ does not only offer minimal
variance for a fixed stratification, but it is also optimal with
respect to the size of a stratification when the function $f$ is
discontinuous. This is in particular relevant for the class of
functions of interest in this work. However,
Lemma~\ref{lemma:hybrid:strat:large-strat} also reveals that suitable
values of $\alpha$ smaller than $1$ may still offer quasi-optimal
variance reduction for discontinuous functions. The summary of these
two opposing effects is therefore that one should select
$\alpha\in [0,1]$ as small as possible yet large enough to provide
quasi-optimal variance reduction.
The numerical examples presented in Sect.~\ref{sec:num-res} contains both results for 
fixed values of $\alpha$, and dynamic choice of $\alpha$, to be described next.
%

\subsubsection{Dynamic choice of the hybrid parameter}
An alternative is to chose $\alpha\in [0,1]$ dynamically across
iterations of the adaptive stratification algorithm described in
Algorithm~\ref{algorithm_ass_brief}. The procedure that we describe in
the following is based on the intuition that at the beginning of the
iterative procedure we may have unreliable estimates of the strata
standard deviations, so that $\alpha$ small is advisable. Eventually,
as the procedure adapts the stratification and increasingly more
samples are being generated, one expects that more reliable estimates
of the standard deviations are available for values of $\alpha<1$, so
that $\alpha$ large is desired for quasi-optimal variance reduction.

To formalize this intuition, let $k\in\mathbb{N}_0$ denote the
iteration counter of Algorithm~\ref{algorithm_ass_brief} and suppose
that the hybrid parameter is initialized with $\alpha_0$ for the first
iteration (e.g., $\alpha_0 = 0$ is natural). Let $\alpha_k$ denote the
value of the hybrid parameter used during the $k$-th iteration. At the
end of that iteration (line 21 in
Algorithm~\ref{algorithm_ass_brief}), the value of the hybrid
parameter for the iteration $k+1$ should then be selected in a way
that provides a good compromise between variance reduction and
robustness.
Here we describe an update based on the stratified sampling estimators
variance $V_\alpha = C_\alpha(\boldsymbol{\sigma})/N$; cf.\
Sect.~\ref{sec:basic-strat-samp:hyrbid-allocation}. If the exact
strata standard deviations $\boldsymbol{\sigma}$ were known, then the
optimal variance reduction would be achieved for $\alpha=1$. However,
the standard deviations $\boldsymbol{\sigma}$ are estimated based on
the samples that have been generated during the iterations so
far. Hence, only the empirical variance constant
$C_\alpha(\hat{\boldsymbol{\sigma}})$ is available. Naively selecting
the hybrid parameter for iteration $k+1$ as
$\alpha_{k+1} = \argmin_{\alpha\in[0,1]}
C_\alpha(\hat{\boldsymbol{\sigma}})$, where
$\hat{\boldsymbol{\sigma}}$ contains the estimated strata standard
deviations using the samples available at iteration $k$, may result in
spurious effects due to the lack of robustness, as the empirical
variance constant introduces an error. In fact, in view of the
asymptotic distribution characterized in
Lemma~\ref{lemma:hybrid:strat:clt}, the empirical variance constant
$C_\alpha(\hat{\boldsymbol{\sigma}})$ is random, and it will fluctuate
around $C_\alpha(\boldsymbol{\sigma})$, where those fluctuations will
be asymptotically normally distributed. The size of the fluctuations
can then be approximately quantified via the limit distribution's
variance, which we write as
  $\varsigma_\alpha^2(\boldsymbol{\sigma},\boldsymbol{\kappa}) = \langle \nabla
  C_{\alpha}(\boldsymbol{\sigma}), \Sigma_\alpha(\boldsymbol{\sigma},\boldsymbol{\kappa}) \nabla
  C_{\alpha}(\boldsymbol{\sigma})\rangle$
to emphasize the dependence on both the strata standard deviations
$\boldsymbol{\sigma}$ and the kurtosis $\boldsymbol{\kappa}$ for
practical considerations. Therefore, instead of selecting
$\alpha\in [0,1]$ by minimizing
$C_{\alpha}(\hat{\boldsymbol{\sigma}})$ at iteration $k$ we
incorporate these fluctuations and consider the minimization of the
empirical function $J_{k}\colon [0,1]\to\mathbb{R}$
\begin{equation*}
  J_{k}(\alpha) 
  := C_{\alpha}(\hat{\boldsymbol{\sigma}}) + \frac{\varsigma_{\alpha}(\hat{\boldsymbol{\sigma}},\hat{\boldsymbol{\kappa}} )}{\sqrt{N}} \;,
\end{equation*}
where the index $k$ indicates that $\hat{\boldsymbol{\sigma}}$ and
$\hat{\boldsymbol{\kappa}}$ contain the estimated strata standard
deviations and kurtosis, receptively, using the $N$ samples available
at iteration $k$. That is, for every iteration $k$, the function
$J_{k}$ corresponds to the upper end of the $68.27~\%$ approximate
(asymptotic) confidence interval for the variance constant
$C_{\alpha}(\boldsymbol{\sigma})$. Moreover, it satisfies
$\lim_{k\to\infty} J_k(\alpha) = C_{\alpha}(\boldsymbol{\sigma})$
almost surely, since $k\to\infty$ implies $N\to\infty$. One therefore
expects that the update by
$\alpha_{k+1} = \argmin_{\alpha\in [0,1]} J_{k}(\alpha) $ satisfies
$\alpha_k\to 1$ asymptotically as $k\to\infty$.
To further robustify the iterative selection of $\alpha$, one could, 
for example, choose $\alpha_{k+1}$ as the smallest value in $[0,1]$
that already provides $\tau\cdot 100\,\%$, $0<\tau \le 1$ of the
optimal (i.e., minimal) upper confidence band for the variance
constant.  The complete selection procedure outlined above is
summarized in Algorithm~\ref{algo:selection:alpha}.
\begin{algorithm}[H]
  \caption{Updating the hybrid sampling parameter $\alpha$.}
  \label{algo:selection:alpha}
  \begin{algorithmic}[1]
    \State Let $\alpha_k$ be the hybrid parameter used for the sample
    allocation at the $k$-th iteration of the adaptive stratification
    procedure.

    \State Estimate $\hat{\boldsymbol{\sigma}}$ and
    $\hat{\boldsymbol{\kappa}}$ based on the available
    $N^{\alpha_k}_S$ samples in each stratum $S$ of the current
    stratification $\mathcal{S}$ at iteration $k$.
    
    \State Determine $J_k^\ast = \min_{\alpha\in [0,1]} J_k(\alpha)$

    \State Set
    $\alpha_{k+1} = \min\bigl\{\alpha\in[0,1]\colon J_k(\alpha)
    -J_k^\ast \le (1-\tau)J_k^\ast \bigr\}$, or $\alpha_{k+1} = 0$ if
    the set is empty.
  \end{algorithmic}
\end{algorithm}
Notice that the choice $\tau = 0$ coincides with the case of no
additional robustification.  Moreover, step 2 of
Algorithm~\ref{algo:selection:alpha} requires estimating both
$\hat{\boldsymbol{\sigma}}$ and $\hat{\boldsymbol{\kappa}}$ based on
the available samples at iteration $k$. To minimize the risk of
drastically misestimating the strata variances (and thus the
kurtosis), cf.\ inequality~\eqref{eq:failure:prob}, we propose to use
KDE-based techniques moment estimates discussed in
Sect.~\ref{sec:adapt-strat-samp:failure-split}, which introduce
additional smoothing for strata with small observed variability.


\section{Numerical results}
\label{sec:num-res}

As discussed in Sect.~\ref{sec:basic-strat-samp}, for a given
stratification (and if exact sample allocation rules are available)
stratified sampling yields an unbiased estimator of
$\E{Q} = \E{f(\bY)}$, whose variance is reduced compared to that of a
classic Monte Carlo estimator; cf.~\eqref{eq:allo_vars}. In fact, the
ratio of the Monte Carlo estimator's variance
$\var{\hat{Q}_{\textup{MC}}} = \var{Q}/N$ and a hybrid stratified
sampling estimator's variance $\var{\hat{Q}_{\alpha}}$ gives rise to
as a measure of the \textit{speedup} in view of the central limit
theorem \eqref{eq:strat:basic:clt}. Indeed, the speedup
\begin{equation}\label{eq:speed-up}
\textup{speedup} \equiv \frac{\var{\hat{Q}_{\textup{MC}}}}{\var{\hat{Q}_{\alpha}}} = \frac{\var{Q}}{C_{\alpha}(\boldsymbol{\sigma})}\;,
\end{equation}
can be seen as the factor of how many more simulations (i.e.,
evaluations of the model $f$) are needed to achieve a certain mean
squared error goal when using a Monte Carlo estimator compared to a
hybrid stratified sampling estimator.

The adaptive stratification procedure introduced in this work aims at
minimizing the stratification estimator's variance by consecutively
refining high variance strata. That is, the procedure is designed to
maximize the speedup even in the presence of heterogeneous features of
$f$. A consequence of this dynamic sampling-based adaptation is that
it is no longer guaranteed that the estimator at the final iteration
will be a statistically unbiased estimator of $\E{Q}$ for a finite
sample size $N_{\text{max}}$, due to the dependencies on the
previously sampled variables and strata refinements. However, taking
advantage of the hybrid sample allocation and following the directives
introduced in Sect.~\ref{sec:adapt-strat-samp}, the adaptive
stratified sampling estimator is still asymptotically unbiased for any
$\alpha\in[0,1)$. Moreover, numerical comparisons (not shown here) of
the empirical adaptive stratified sampling estimators' bias with the
empirical Monte Carlo bias for the test cases considered in the
following indicated that both terms are of similar size.

Recall that in this work, we assume that the total cost of obtaining
$N_{\text{max}}$ samples of the function of interest is vastly higher
than the cost of the dynamic stratification and allocation of
samples. In what follows, the performance of the algorithm presented
in previous sections is evaluated on a range of different test
cases. We note that we do not present trivial cases where e.g., a
discontinuity can be exactly tessellated by the kind of stratification
chosen here (hyperrectangles or simplices), but report that for such
cases the speedups compared to standard Monte Carlo become arbitrarily
large.  A discontinuity defined by $1/2^n$ of a hypersphere is thus a
nontrivial and more insightful test case where the number of
stochastic dimensions $n$ can be varied to investigate how the
performance depends on the dimensionality of the problem
(Sect.~\ref{sec:num-res:hyper-sphere}). The variance is strictly
localized to strata that contain the boundary of the hypersphere, and
zero elsewhere. That is, the challenging aspect of these test cases
for the adaptive stratification procedure is the identification (and
adaptive isolation) of a discontinuity.
Next, an idealized model for the critical pressure in a faulted
reservoir based on analytical physical expressions is investigated
(Sect.~\ref{sec:num-res:fault-stress}). The model exhibits both
discontinuities and regions of sharp but continuous
variation. Finally, we perform numerical tests on the Sod test case
from fluid dynamics, and a vertical
equilibrium porous medium model describing subsurface CO$_2$ storage
(Sects.~\ref{sec:num-res:sod} and \ref{sec:num-res:CO2}). All these
cases exhibit discontinuous dependence of the quantity of interest
(QoI) $Q=f(\bY)$ with respect to the stochastic parameters $\bY$, as
is depicted in Fig.~\ref{fig:surfaces_test_problems} and described in
more details in subsequent subsections.
\begin{figure}[h]
\centering
\subfigure[Fault stress problem, Sect.~\ref{sec:num-res:fault-stress}.]
{\includegraphics[width=0.31\textwidth]{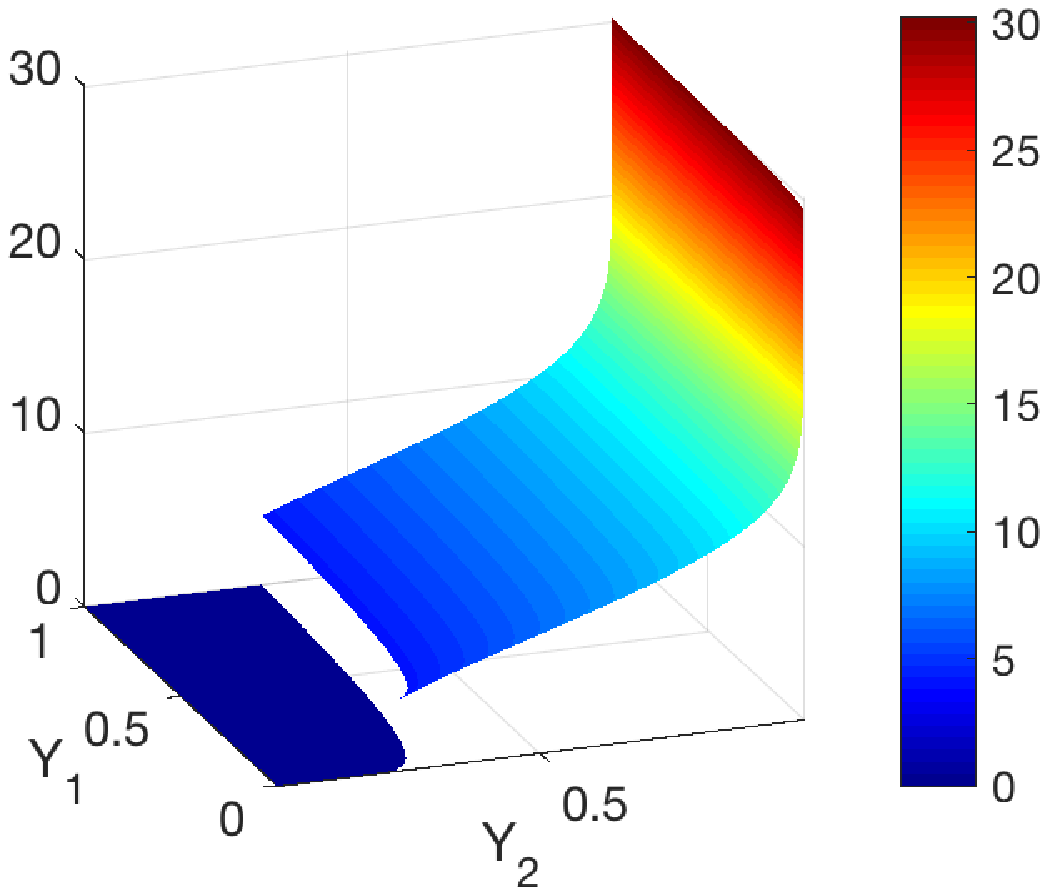}
}
~
\subfigure[Sod test case, Sect.~\ref{sec:num-res:sod}.]
{\includegraphics[width=0.31\textwidth]{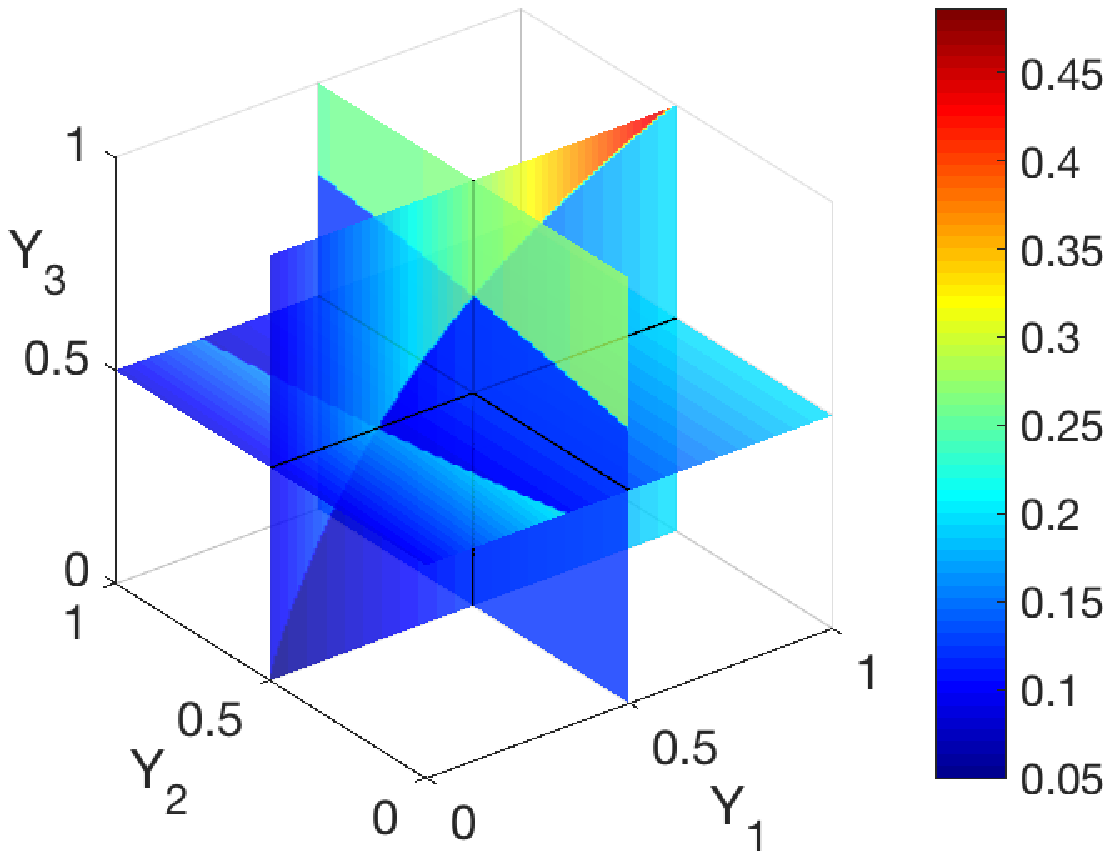}
\label{fig:surfaces:Sod}
}
~
\subfigure[{CO$_2$} storage problem, Sect.~\ref{sec:num-res:CO2}.]
{\includegraphics[width=0.31\textwidth]{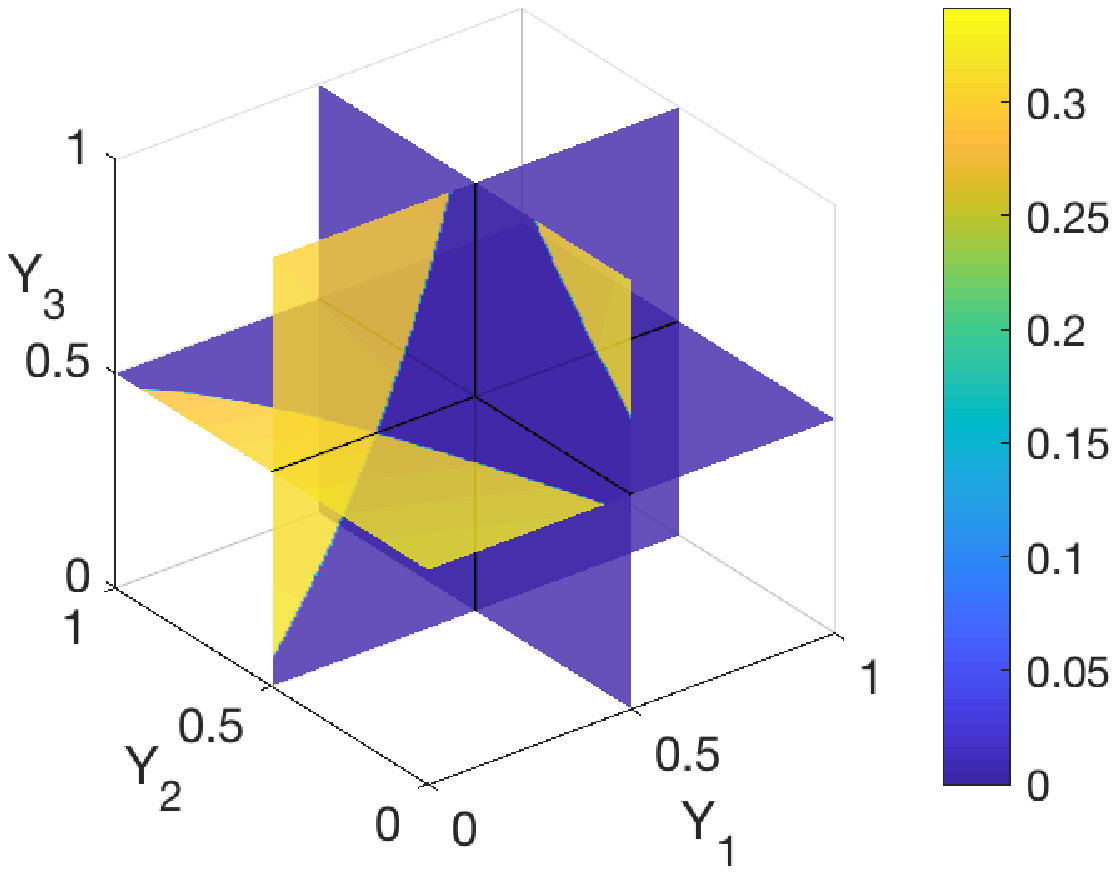}
\label{fig:surfaces:CO2}
}
\caption{Test problems $Q=f(\bY)$. Here, $Q$ is the stress threshold
  (MPa), density (kg/m$^3$), and normalized CO$_2$ plume height
  (dimensionless), respectively.}
\label{fig:surfaces_test_problems}
\end{figure}
For all test problems, we present results using hybrid sampling with
$\alpha=0$ (proportional sampling), $\alpha=0.9$, and $\alpha$
determined dynamically on $[0, \ 0.95]$ using
Algorithm~\ref{algo:selection:alpha}.

\subsection{Hyperspherical discontinuity in multiple dimensions}
\label{sec:num-res:hyper-sphere}
Consider the unit step function with all coordinates within the unit
intervals, and restricted by the hypersphere in $n$ dimensions with
center at the origin and radius $r_n$ chosen so that the hypersphere's
volume is equal to $2^{n-1}$,
\begin{equation}
\label{eq:test:hypersphere}
f(\bY) = \left\{ \begin{array}{ll} 
1 & \mbox{if } \Vert\bY \Vert_2^2\leq r_{n}^2 \mbox{ and } 0\leq Y_i \leq 1, i=1,\dots, n, \\
0 & \mbox{otherwise}
\end{array}\right..
\end{equation}
An example of the development of the adaptive stratification is shown
for the $n=2$ case in Fig.~\ref{fig:stratifications_P1}. Starting from
30 uniformly distributed samples in the unit square, a total of 10,000
samples are distributed adaptively, resulting in a final
stratification consisting of 17 strata. The samples are progressively
allocated to strata that contain the discontinuity, i.e., have
non-zero variances. In addition, 10\% of samples every iteration are
allocated proportionally to enable detecting unseen solution features.
\begin{figure}[h]
\centering
\subfigure[2 strata.]
{\includegraphics[width=0.225\textwidth]{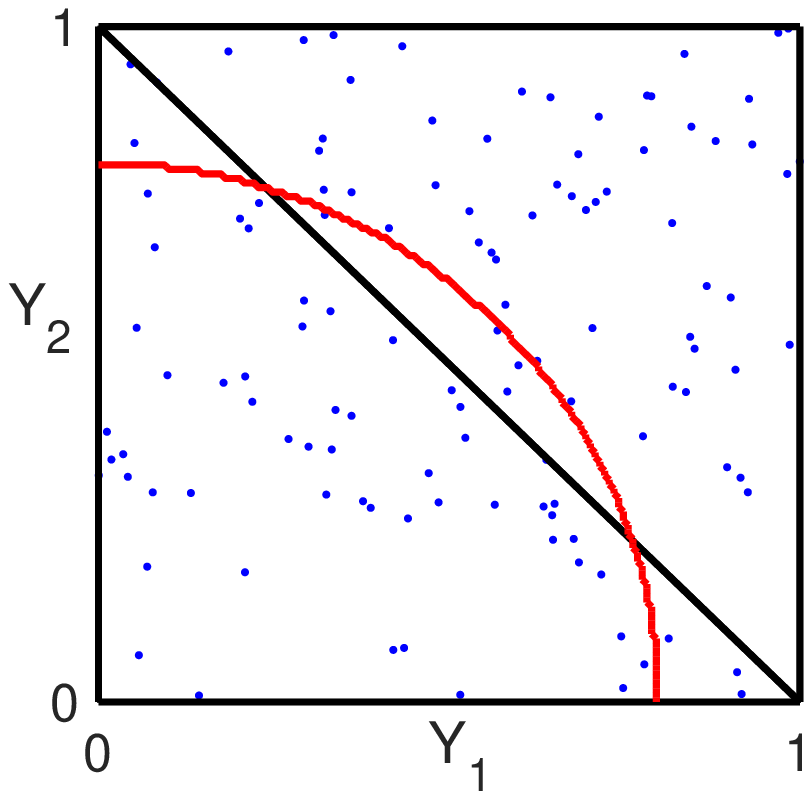}
}
~
\subfigure[3 strata.]
{\includegraphics[width=0.225\textwidth]{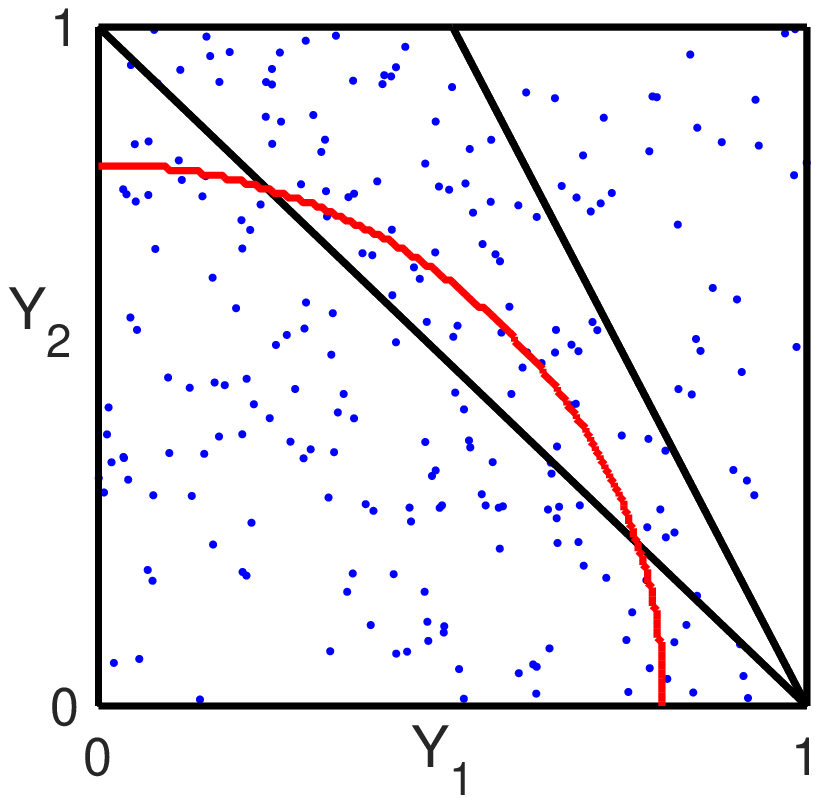}
}
~
\subfigure[10 strata.]
{\includegraphics[width=0.225\textwidth]{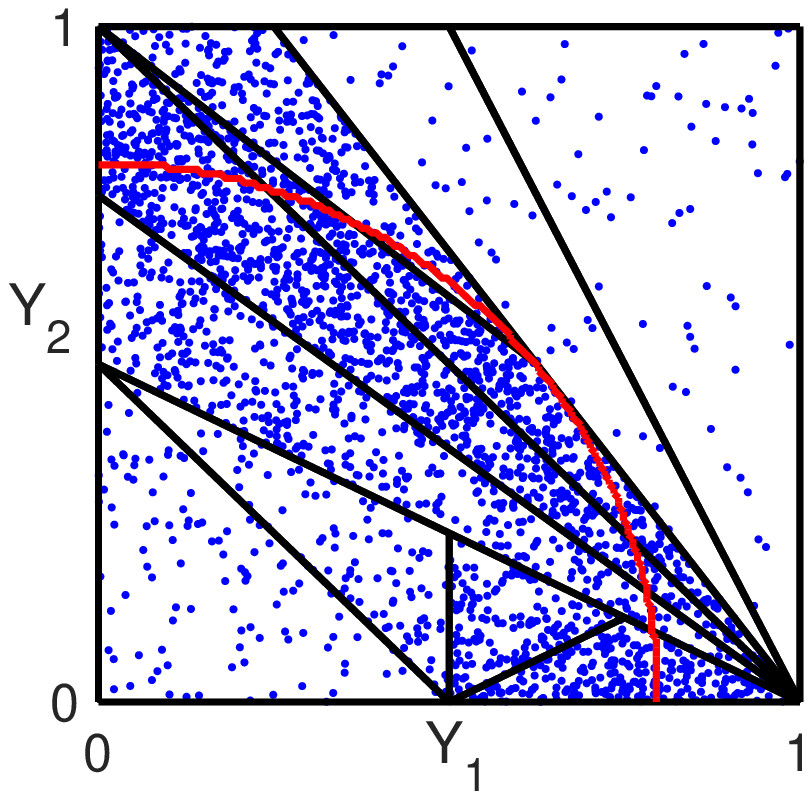}
}
~
\subfigure[17 strata.]
{\includegraphics[width=0.225\textwidth]{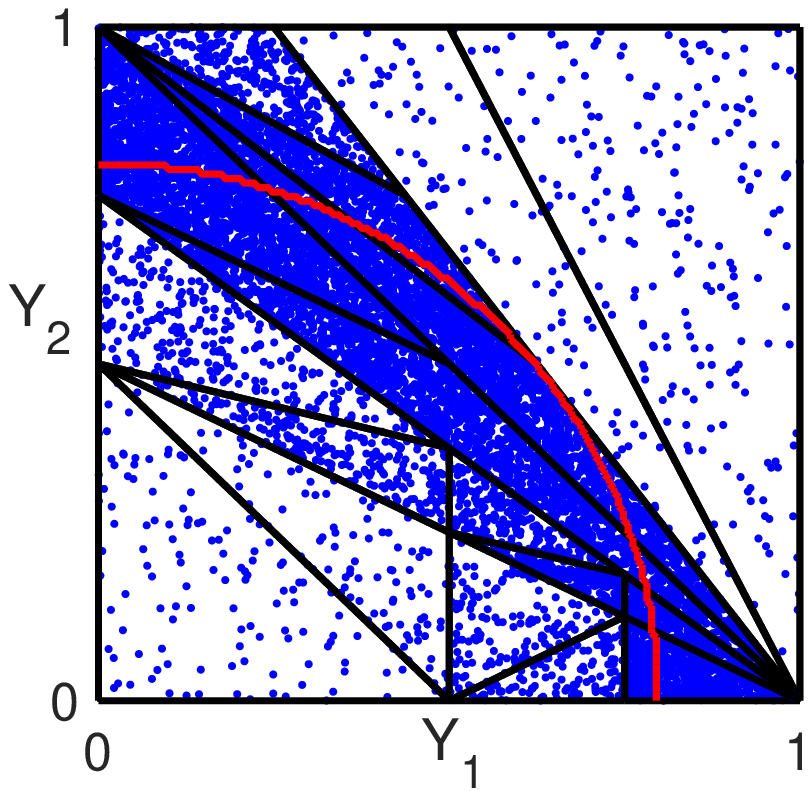}
}
\caption{Evolution of simplex tessellation for the quarter of a circle (hypersperical $n=2$) test case. Hybrid sampling with $\alpha=0.9$ and 30 samples per simplex added at every iteration with a total of 10,000 samples.}
\label{fig:stratifications_P1}
\end{figure}
Speedups for $n=2,3,4$ are shown in
Figs.~\ref{fig:speedups_P1}-\ref{fig:speedups_P3}, respectively. For
$n=2$ and the largest sample sizes, the speedups are 2-3 orders of
magnitude compared to standard Monte Carlo sampling. For $n=3$, the
corresponding speedups are 10-40, and for $n=4$ typically 5-10.  Small
values of $c$, intuitively corresponding to more greedy yet more
``risky'' stratifications in view of Sect.~\ref{sec:dist_samples},
tend to slightly increase speedup, but the pattern is not
unambiguous. Dynamically choosing $\alpha$ in most cases resembles
sampling with fixed $\alpha=0.9$ with some notable differences for
smaller sample sizes and $n=4$.


\begin{figure}[h]
\centering
\subfigure[Hyperrect., $\alpha=0$.]
{\includegraphics[width=0.3\textwidth]{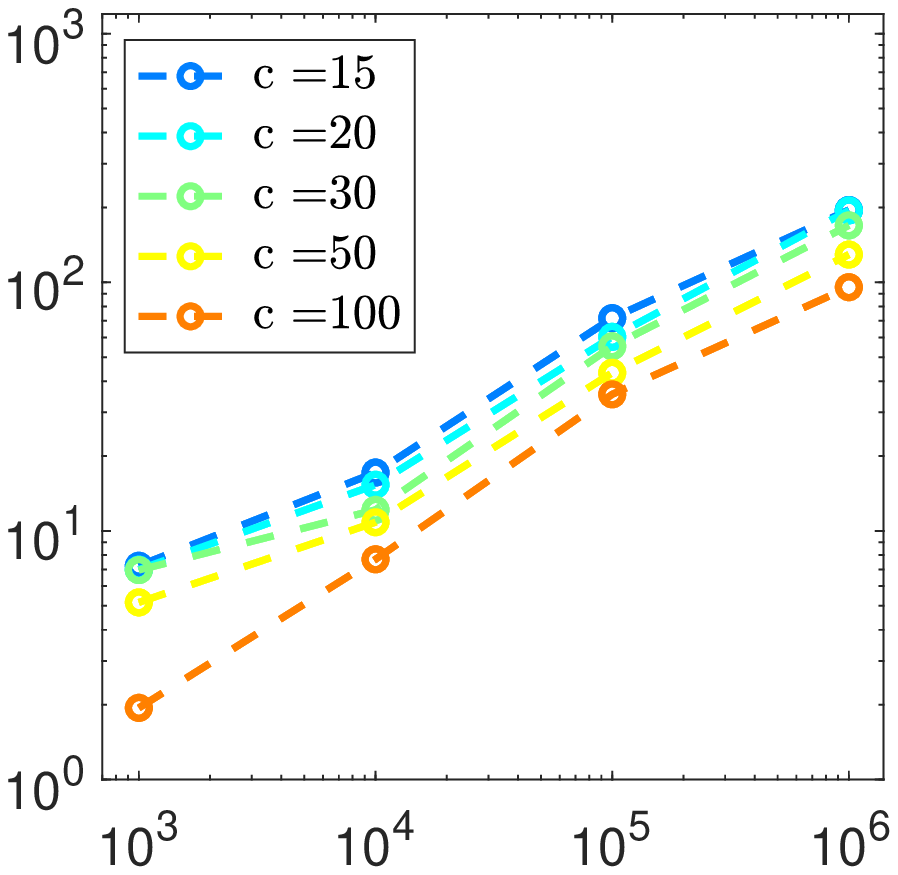}
}
~
\subfigure[Hyperrect., $\alpha=0.9$.]
{\includegraphics[width=0.3\textwidth]{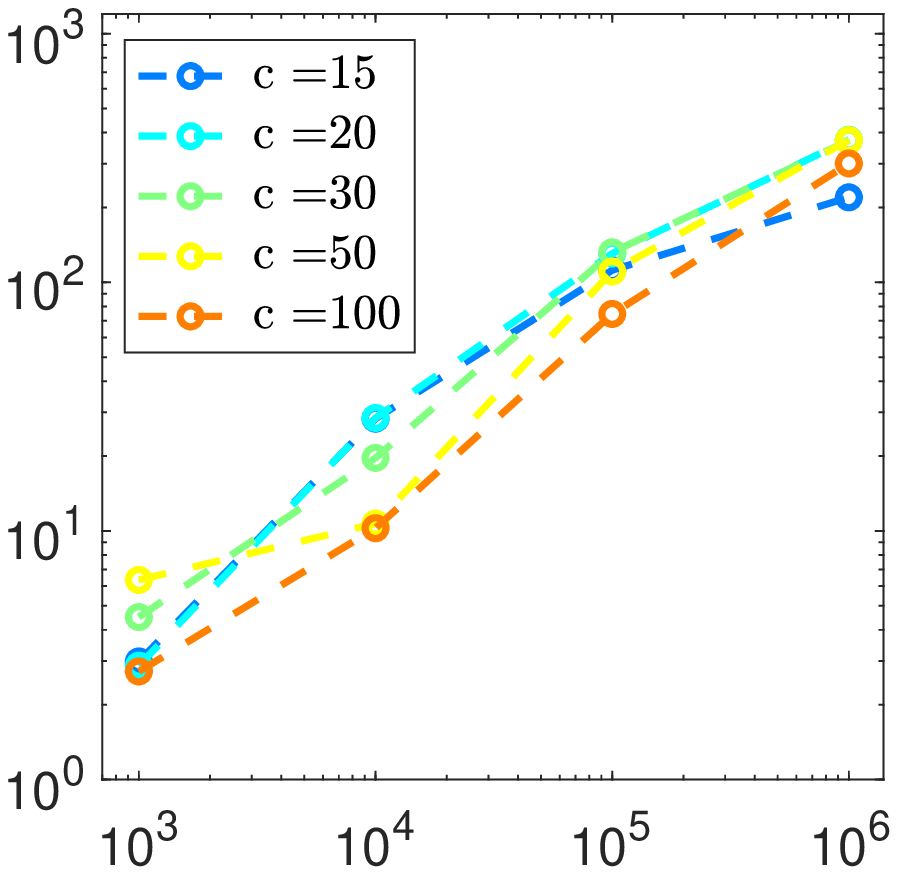}
}
~
\subfigure[Hyperrect., dynamic $\alpha$.]
{\includegraphics[width=0.3\textwidth]{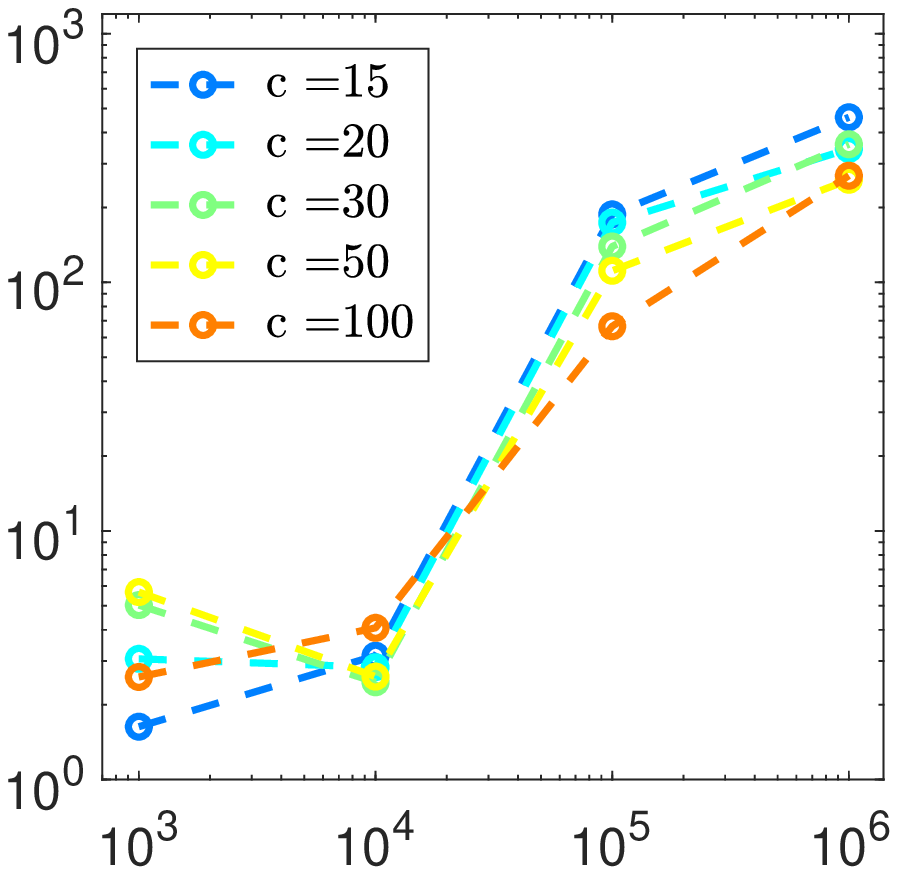}
}
~
\subfigure[Simplices, $\alpha=0$.]
{\includegraphics[width=0.3\textwidth]{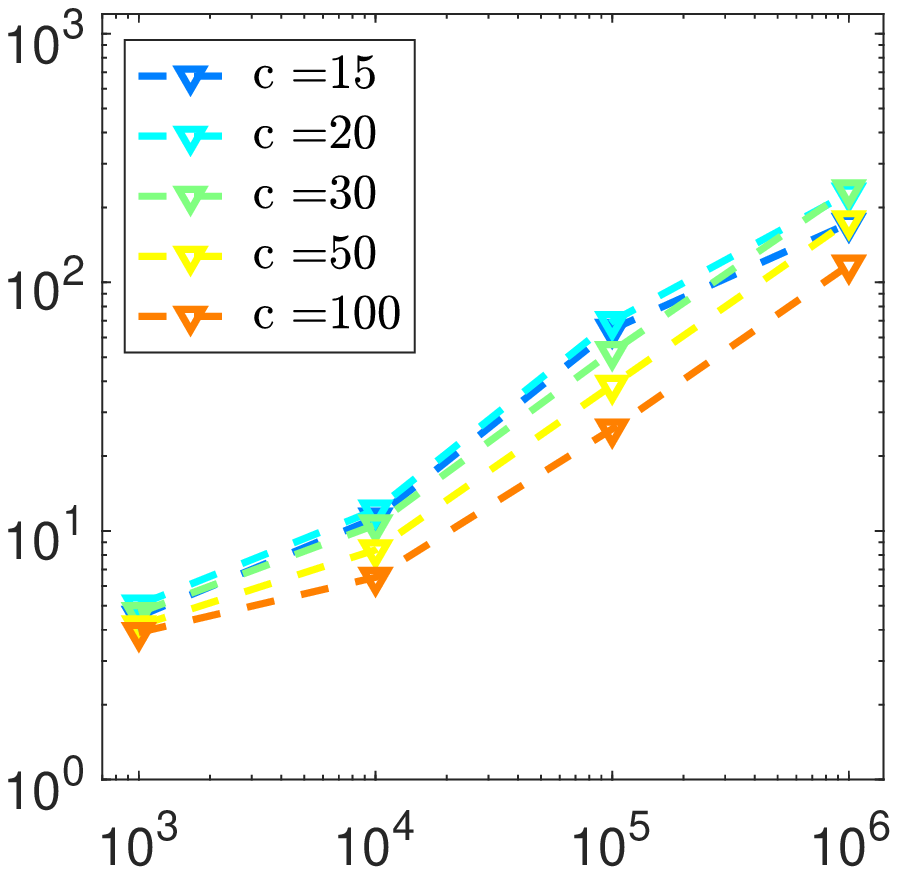}
}
~
\subfigure[Simplices, $\alpha=0.9$.]
{\includegraphics[width=0.3\textwidth]{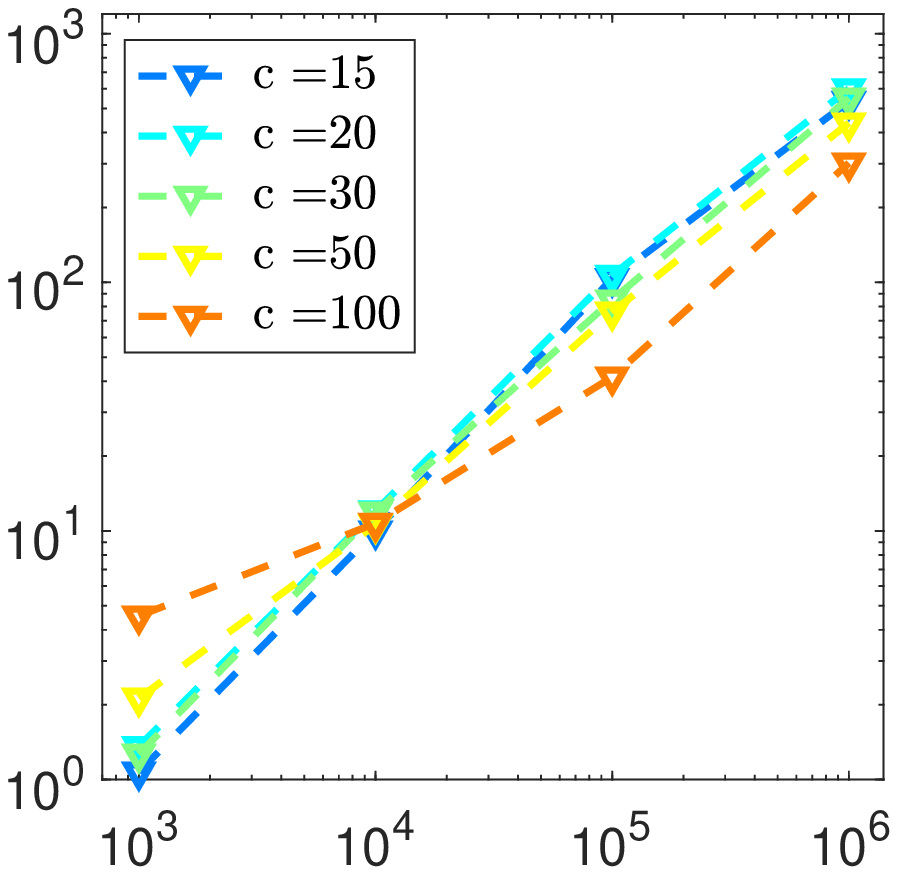}
}
~
\subfigure[Simplices, dynamic $\alpha$.]
{\includegraphics[width=0.3\textwidth]{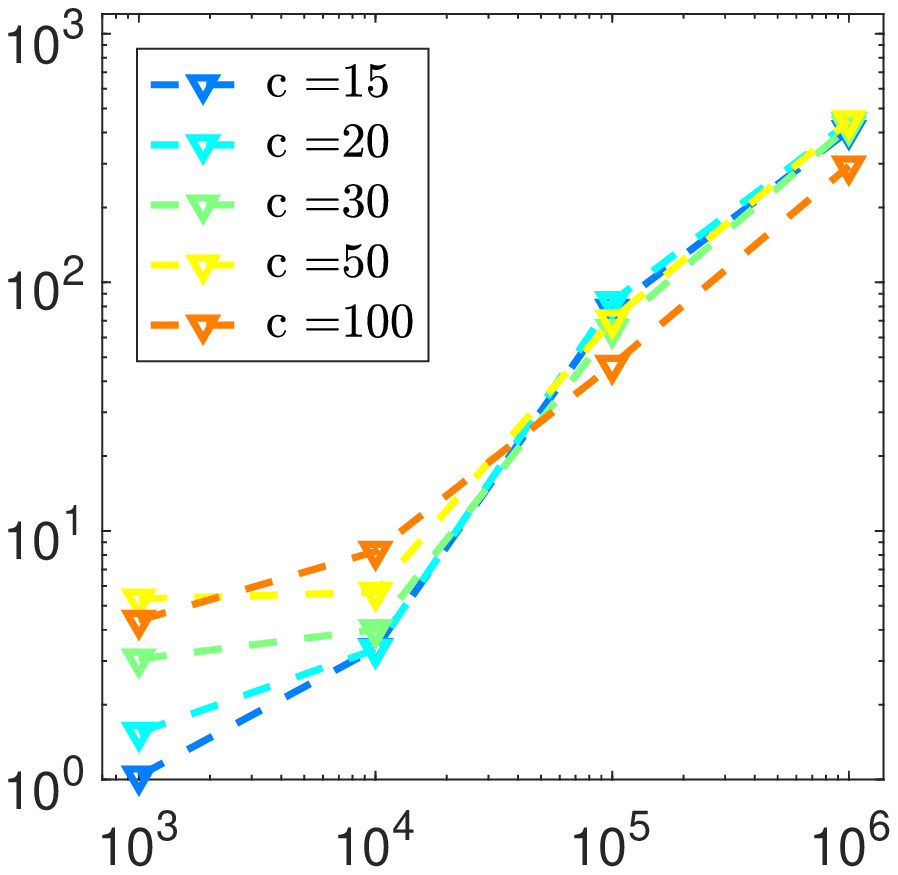}
}
\caption{Hyperspherical discontinuity problems with $n=2$ random dimensions. Hyperrectangular and simplex tessellation, sampling with $\alpha$ being 0, 0.9 or set dynamically, and 1000 repetitions, variable $N_{\textup{max}}$ and sampling constant $c$.}
\label{fig:speedups_P1}
\end{figure}


\begin{figure}[h]
\centering
\subfigure[Hyperrect., $\alpha=0$.]
{\includegraphics[width=0.3\textwidth]{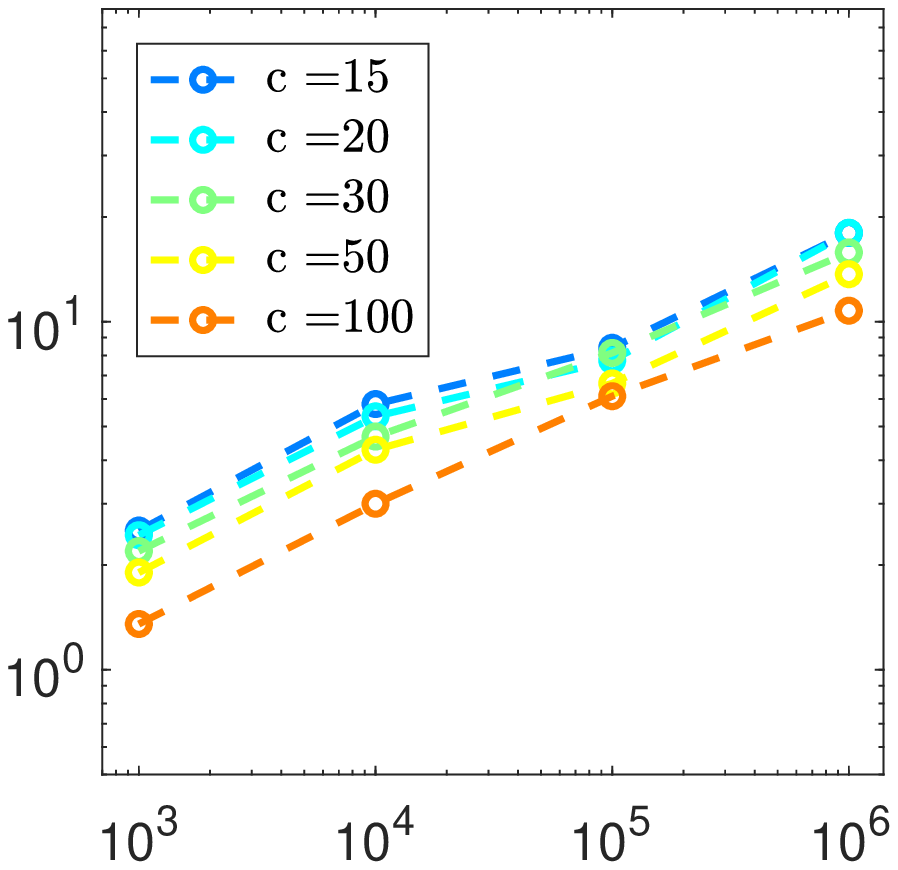}
}
~
\subfigure[Hyperrect., $\alpha=0.9$.]
{\includegraphics[width=0.3\textwidth]{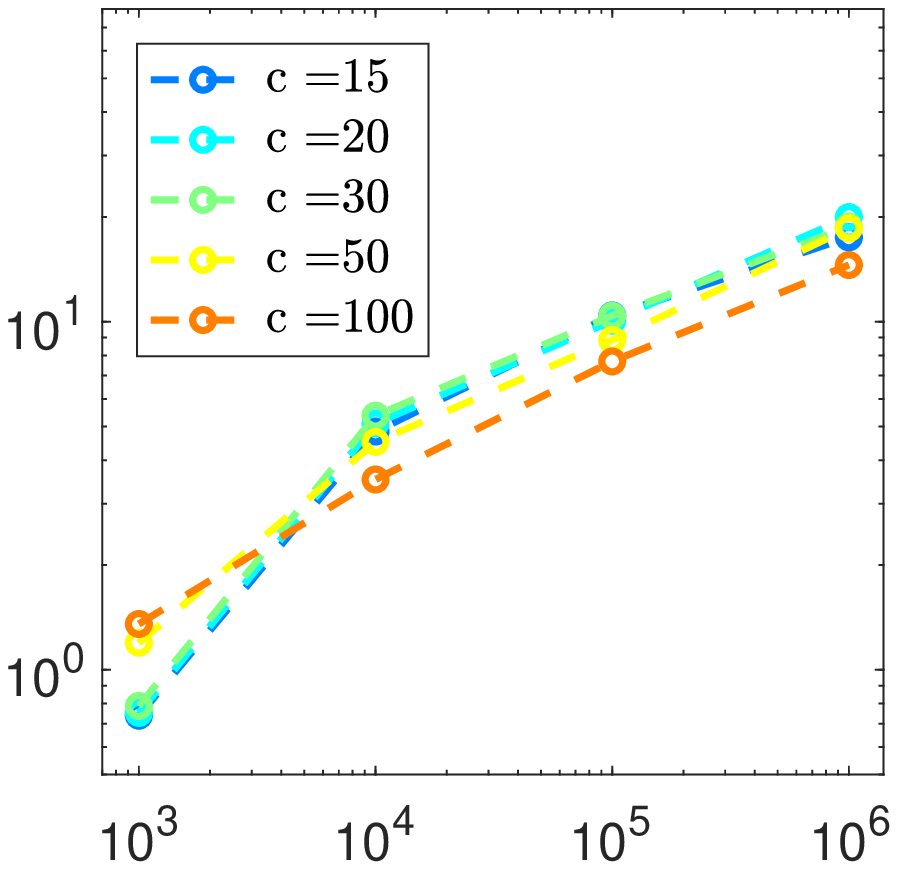}
}
~
\subfigure[Hyperrect., dynamic $\alpha$.]
{\includegraphics[width=0.3\textwidth]{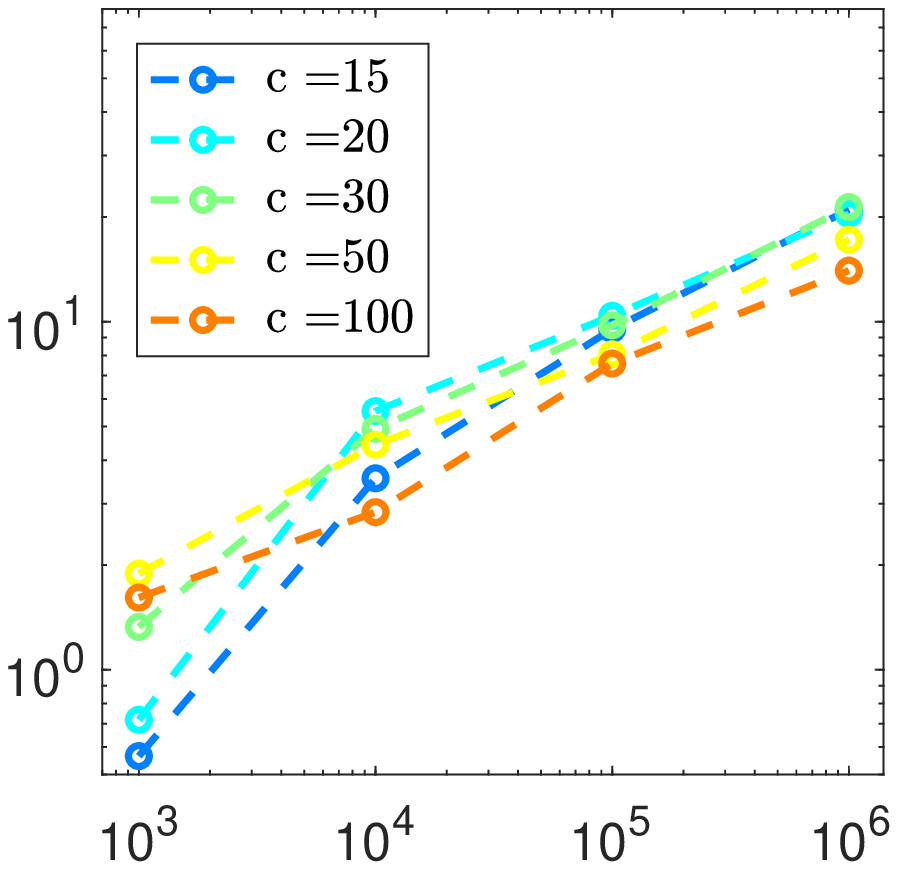}
}
~
\subfigure[Simplices, $\alpha=0$.]
{\includegraphics[width=0.3\textwidth]{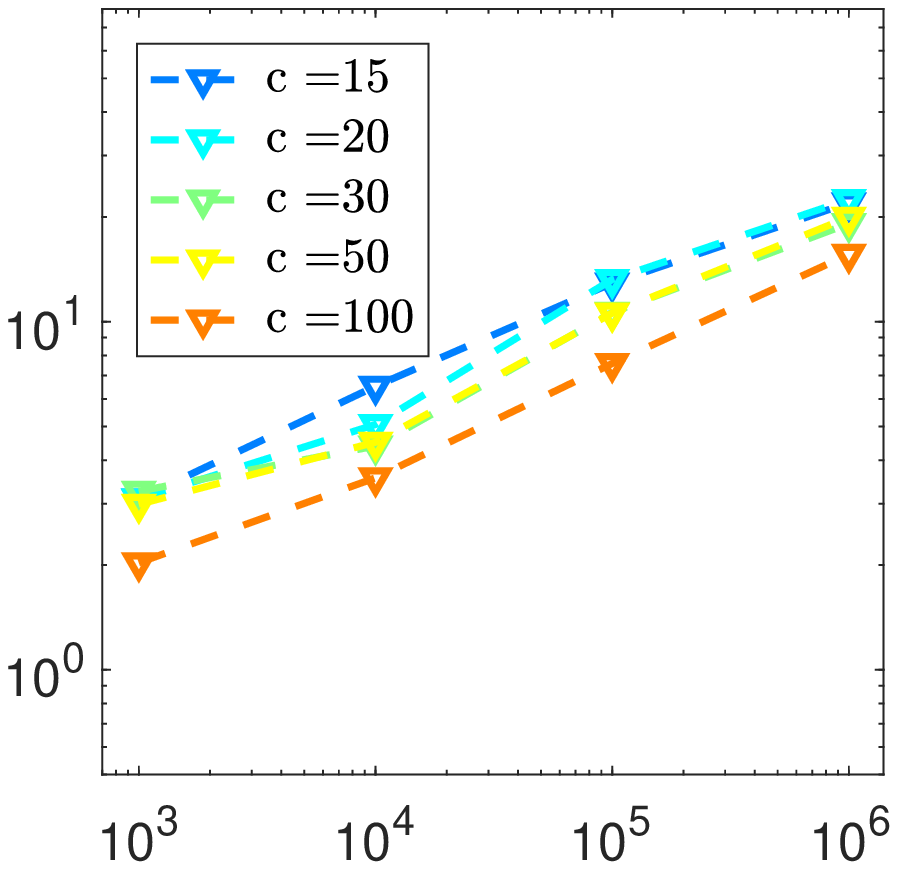}
}
~
\subfigure[Simplices, $\alpha=0.9$.]
{\includegraphics[width=0.3\textwidth]{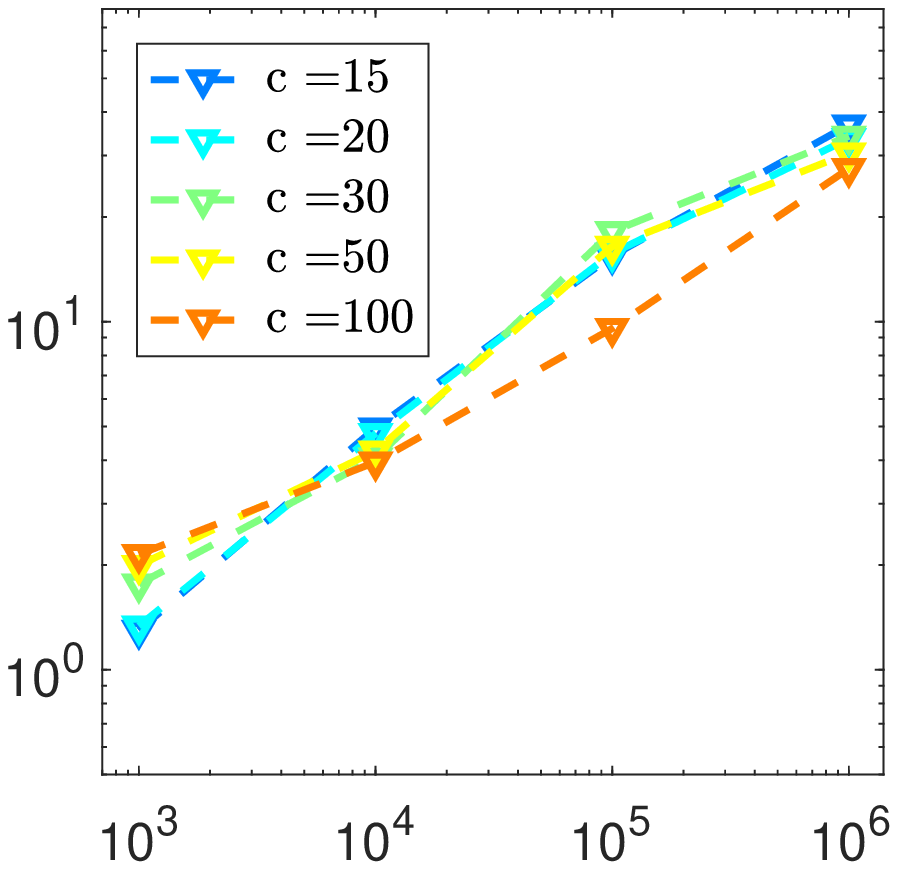}
}
~
\subfigure[Simplices, dynamic $\alpha$.]
{\includegraphics[width=0.3\textwidth]{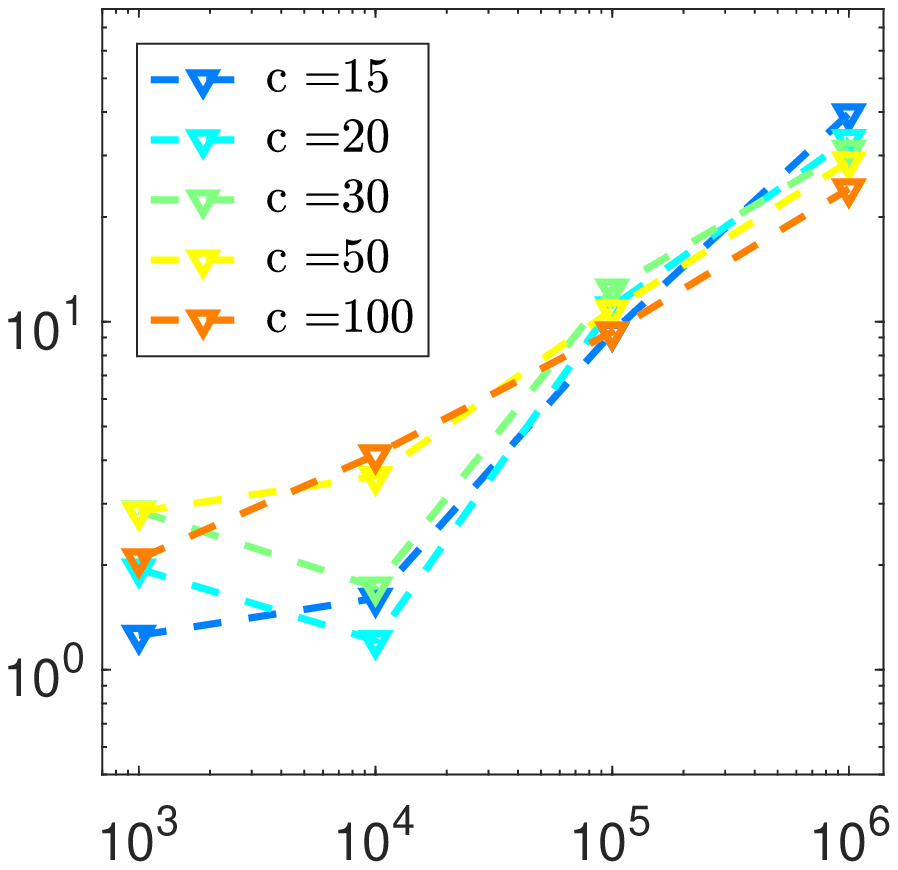}
}
\caption{Hyperspherical discontinuity problems with $n=3$ random dimensions. Hyperrectangular and simplex tessellation, sampling with $\alpha$ being 0, 0.9 or set dynamically, and 1000 repetitions, variable $N_{\textup{max}}$ and sampling constant $c$.}
\label{fig:speedups_P2}
\end{figure}


\begin{figure}[h]
\centering
\subfigure[Hyperrect., $\alpha=0$.]
{\includegraphics[width=0.3\textwidth]{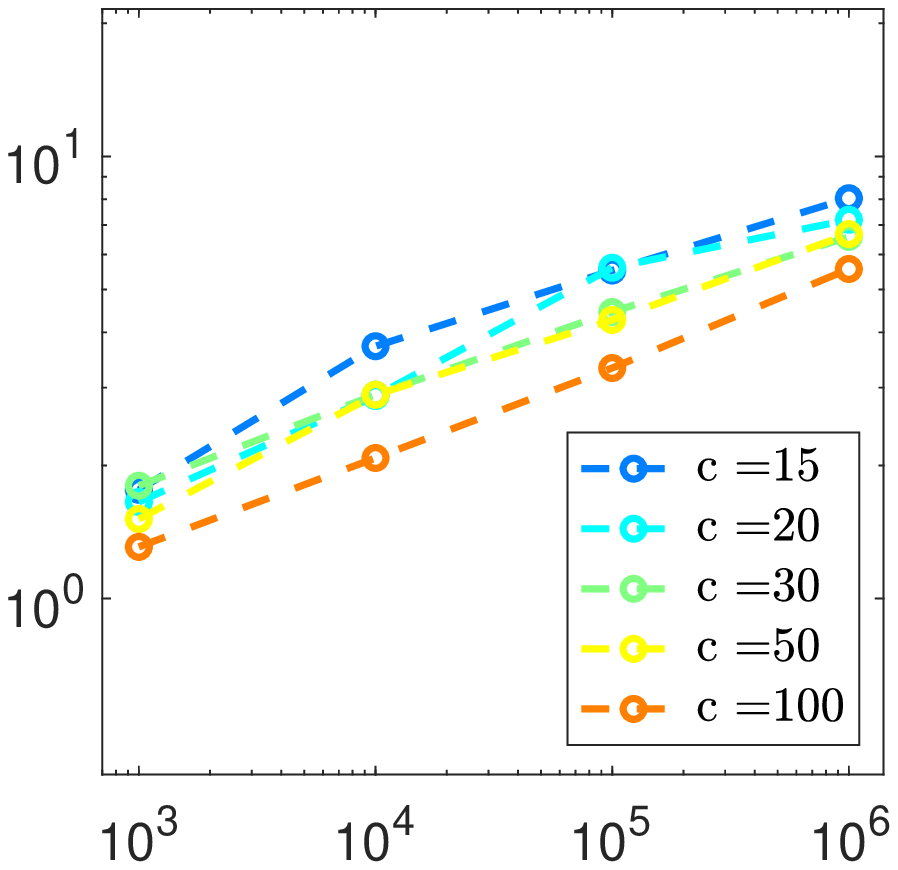}
}
~
\subfigure[Hyperrect., $\alpha=0.9$.]
{\includegraphics[width=0.3\textwidth]{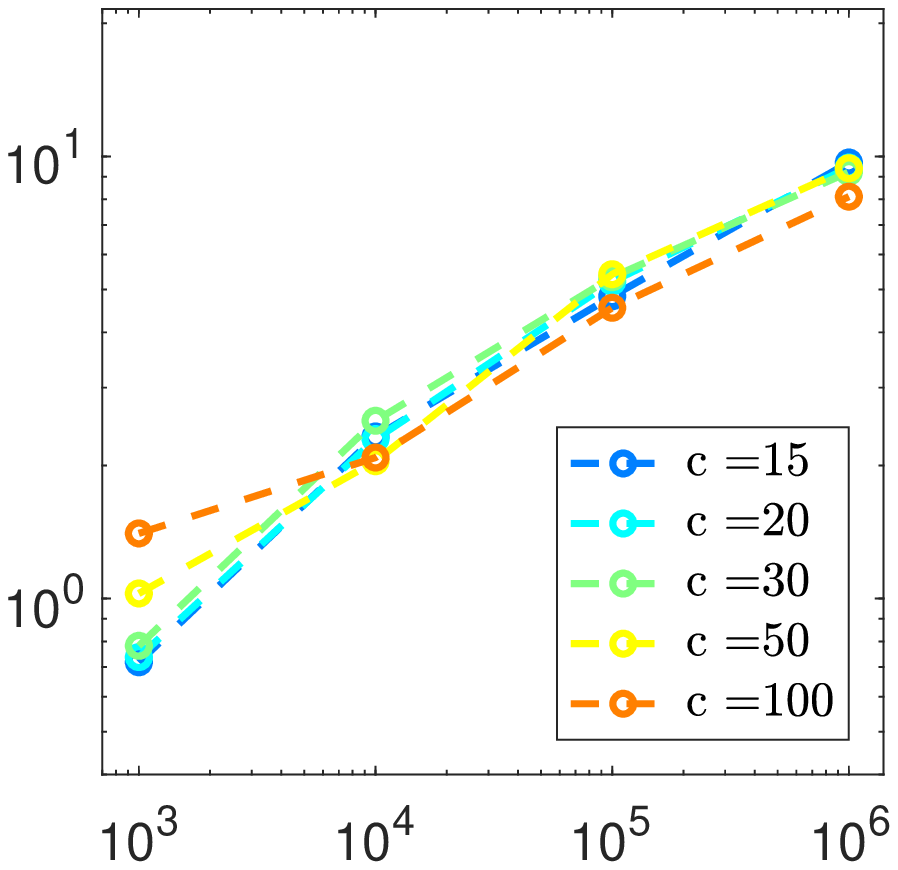}
}
~
\subfigure[Hyperrect., dynamic $\alpha$.]
{\includegraphics[width=0.3\textwidth]{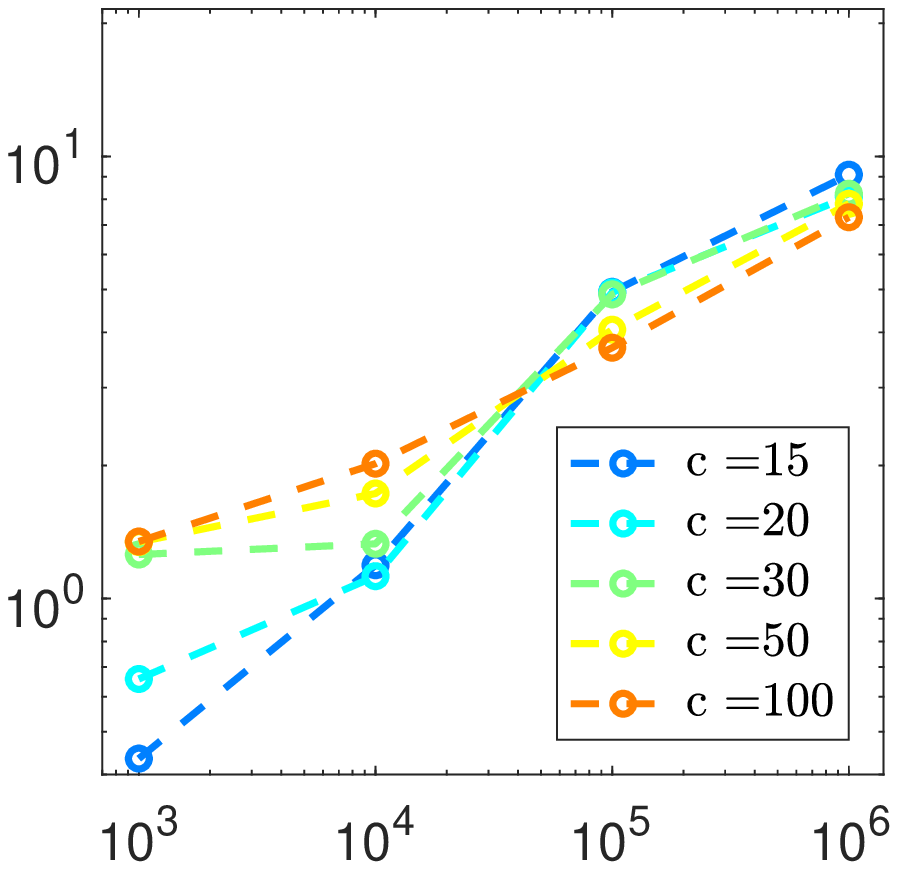}
}
~
\subfigure[Simplices, $\alpha=0$.]
{\includegraphics[width=0.3\textwidth]{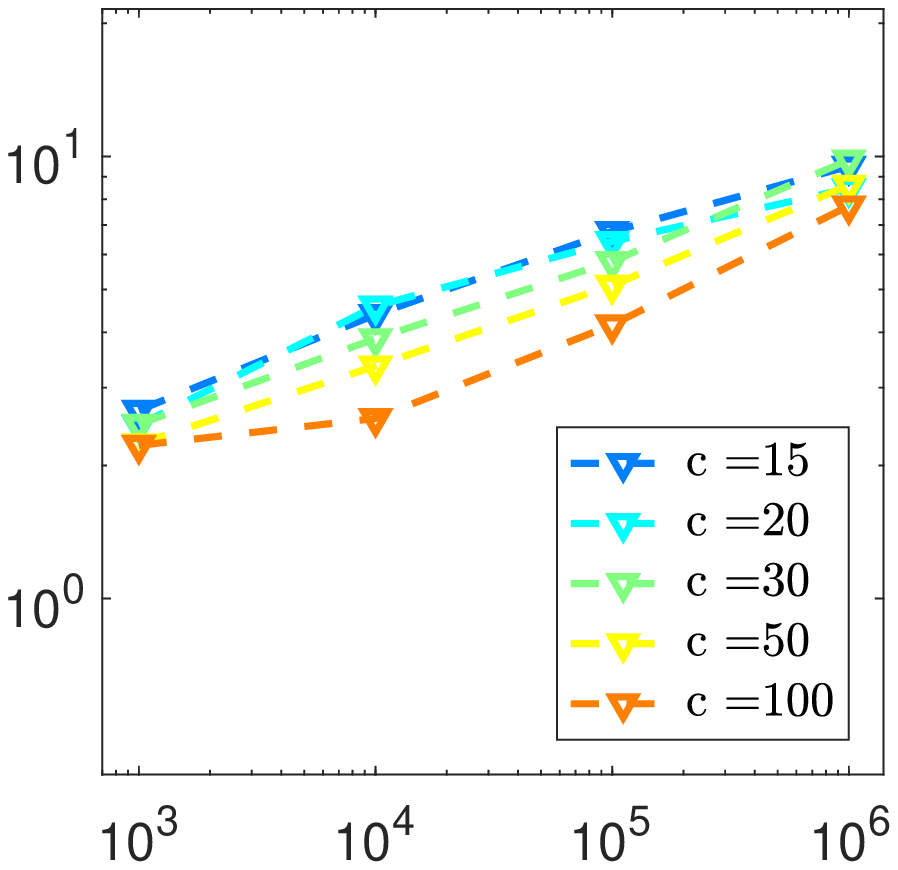}
}
~
\subfigure[Simplices, $\alpha=0.9$.]
{\includegraphics[width=0.3\textwidth]{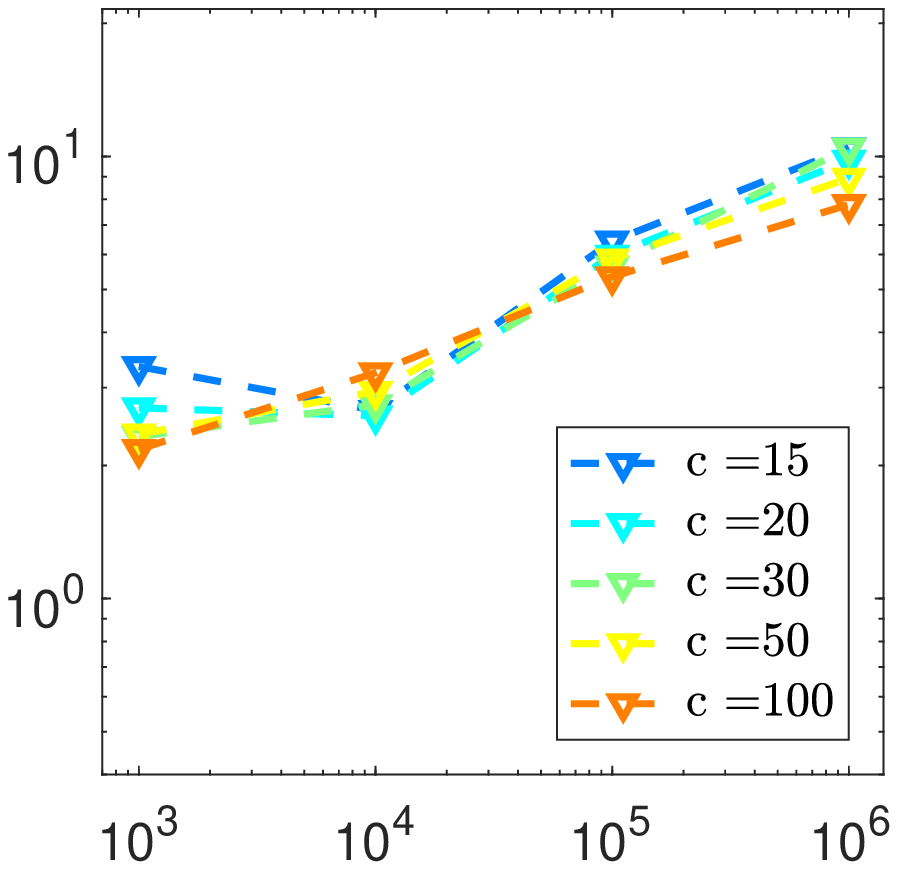}
}
~
\subfigure[Simplices, variable $\alpha$.]
{\includegraphics[width=0.3\textwidth]{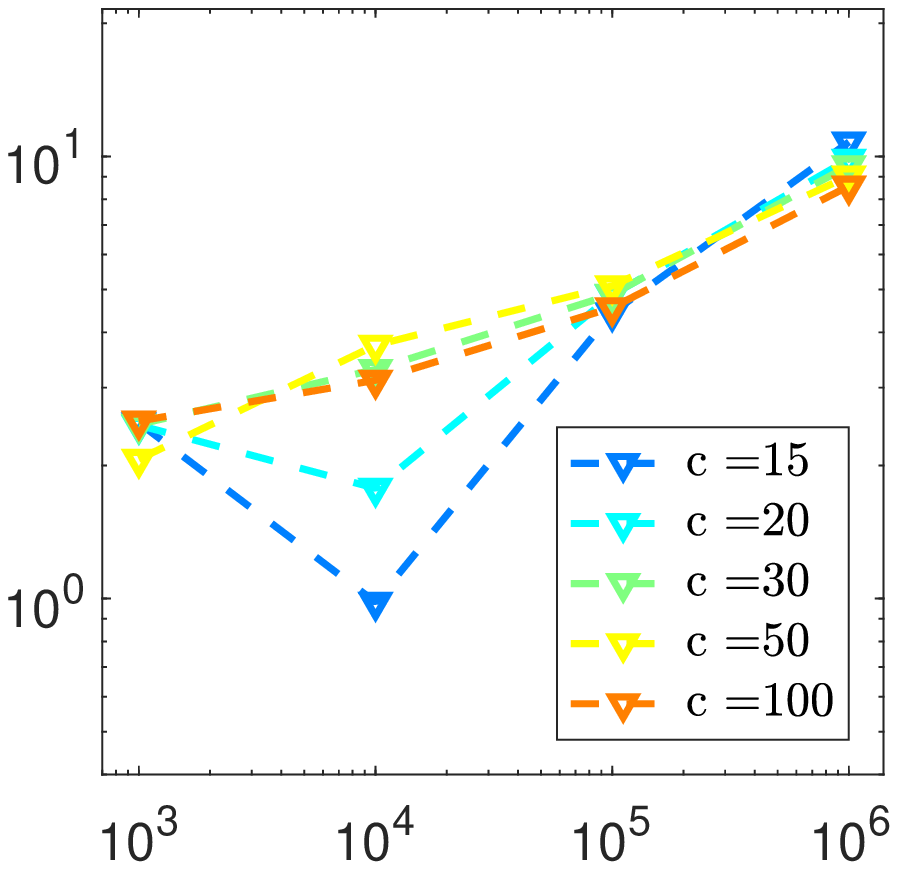}
}
\caption{Hyperspherical discontinuity problems with $n=4$ random dimensions. Hyperrectangular and simplex tessellation, sampling with $\alpha$ being 0, 0.9 or set dynamically, and 1000 repetitions, variable $N_{\textup{max}}$ and sampling constant $c$.}
\label{fig:speedups_P3}
\end{figure}

\subsection{Geomechanics fault surface stress problem}
\label{sec:num-res:fault-stress}
Consider injection through a well into a porous medium with unknown
permeability and unknown location of a fault. We employ a simplified
model for the surface stress at a fault assuming single-phase flow and
no poro-elastic effects of injection, as
in~\cite{Nordbotten_Celia_11}. The background stress field only
changes at the injection surface, and hydrostatic pressure is assumed.
The QoI is the stress threshold $S_{\text{thres}}$ that indicates fault stability, 
\[
S_{\text{thres}} = \left\{
\begin{array}{ll}
\Delta \tau - \Delta \sigma \mu_{F} \Delta F & \mbox{ if } \Delta \text{CFF} < 0, \\
0 & \mbox{ otherwise}
\end{array}
\right.
\]
where $\Delta \text{CFF} = \Delta \sigma \mu_{F} - \Delta \tau$ is a
Coulomb fault failure criterion. The friction drop $\Delta F$ is 0.8,
the shear stress $\Delta \tau$ is 20 MPa, the normal stress
$\Delta \sigma$ is the difference between an original stress assumed
to be 50 MPa and the pressure $p = \rho_{w} g ((d + 0.5 H) + h)$;
where $\rho_{w}=1000$ kg/m$^3$ (density water), $g=9.81$
(gravitational constant), $d=2000$ m (depth of caprock), $H=100$ m
(height of injection formation), and the hydraulic head $h$ is a
function of the well distance and given by Eq.~(2.46)
in~\cite{Nordbotten_Celia_11}.  The stochastic input parameters are
the well distance (in meters) with distribution $U[10,1000]$, and the
friction coefficient $\mu_{F} $ (dimensionless) with lognormal(0.2,
0.7) distribution.

Figure~\ref{fig:speedups_P4} shows the variance ratios for the fault
surface problem as a function of total number of samples, for
$\alpha=0, 0.9$ and dynamic $\alpha$ using hyperrectangles and
simplices for the adaptive stratifications. For this problem, the best
performance is achieved for hyperrectangular stratifications with
proportional allocation ($\alpha=0$) with speedups exceeding three
orders of magnitude compared to Monte Carlo for the largest sample
sizes investigated. This contrasts with the hyperspherical test cases
previously presented, where simplices and nearly optimal allocation
yield the best results.

\begin{figure}[h]
\centering
\subfigure[Hyperrect., $\alpha=0$.]
{\includegraphics[width=0.3\textwidth]{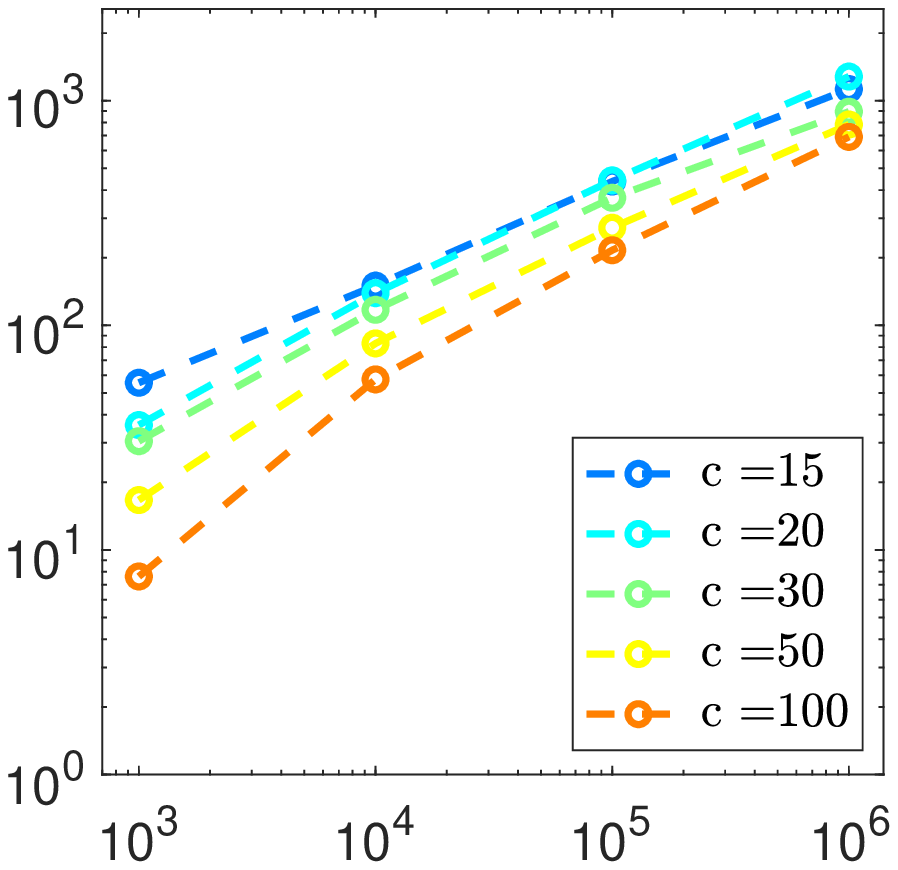}
\label{fig:example:P1:hyperrect:QoI:N3}
}
~
\subfigure[Hyperrect., $\alpha=0.9$.]
{\includegraphics[width=0.3\textwidth]{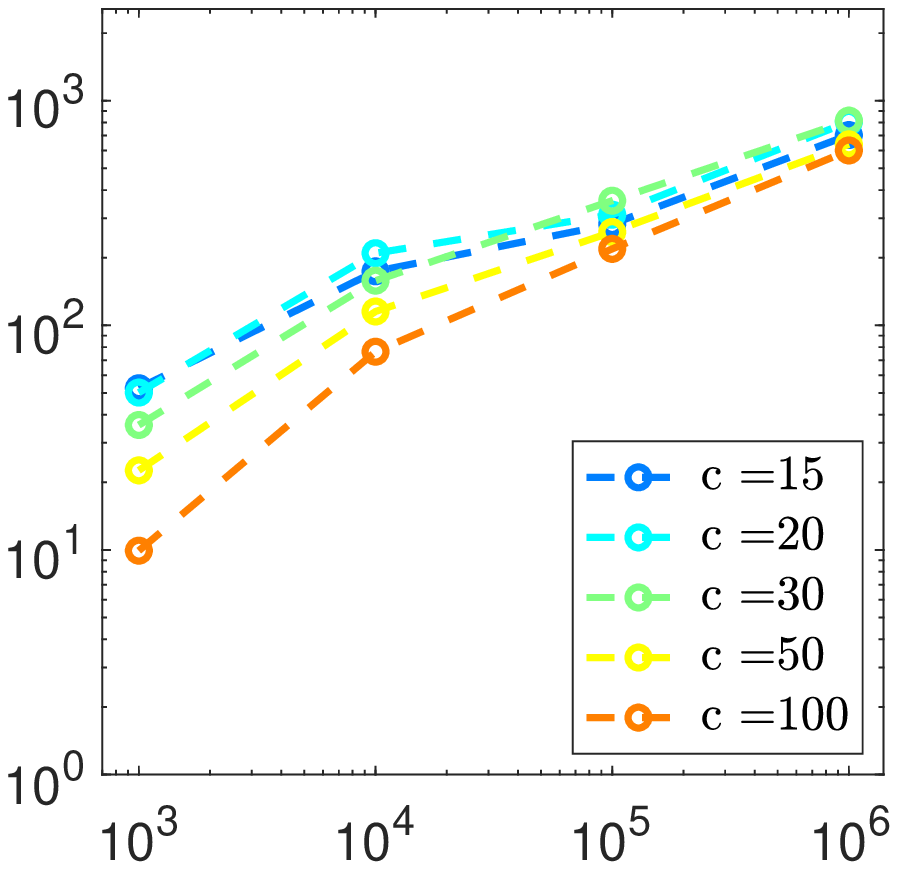}
\label{fig:example:P1:hyperrect:QoI:N4}
}
~
\subfigure[Hyperrect., dynamic $\alpha$.]
{\includegraphics[width=0.3\textwidth]{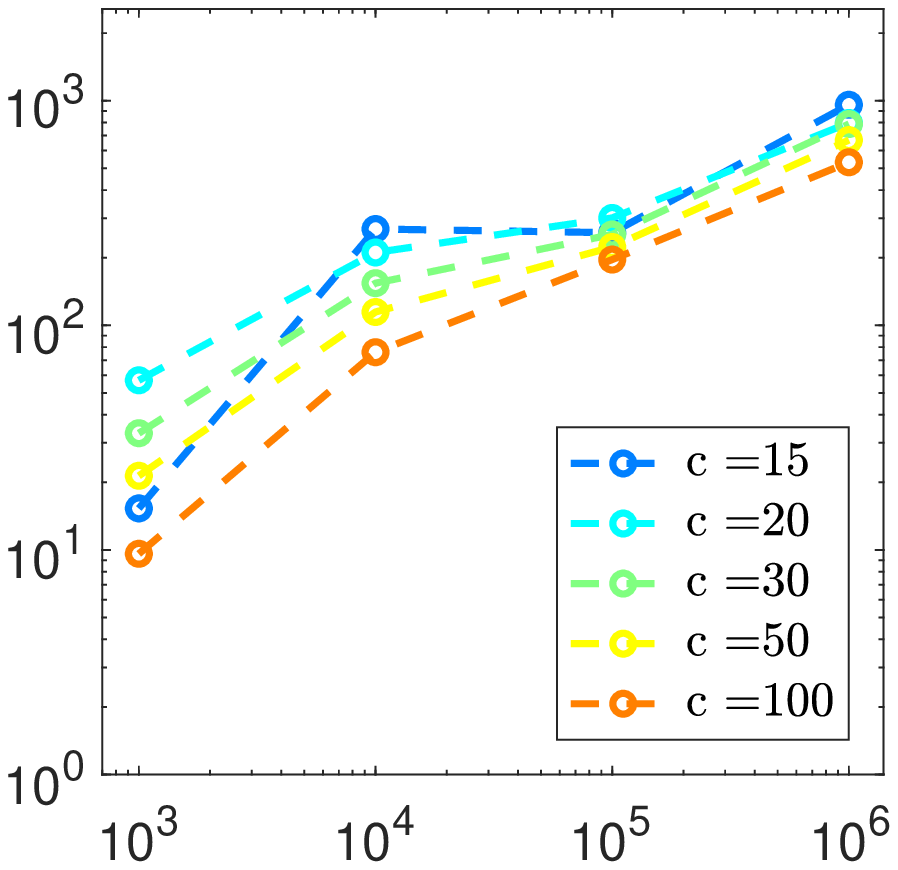}
\label{fig:example:P1:hyperrect:QoI:N4}
}
~
\subfigure[Simplices, $\alpha=0$.]
{\includegraphics[width=0.3\textwidth]{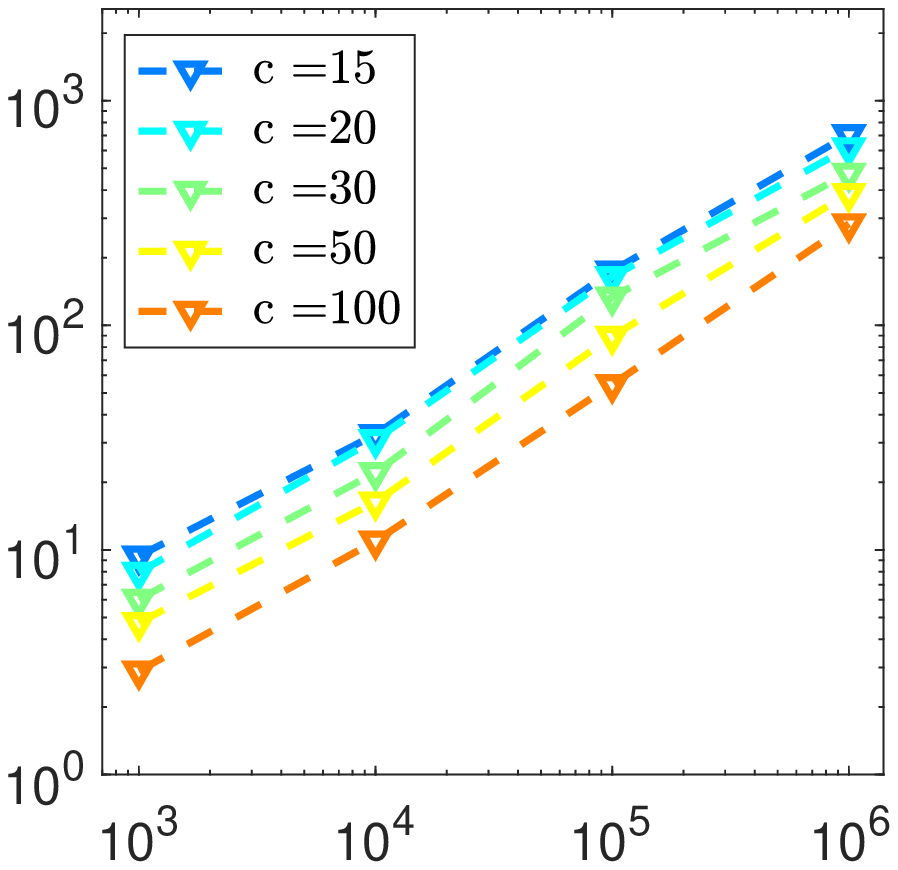}
}
~
\subfigure[Simplices, $\alpha=0.9$.]
{\includegraphics[width=0.3\textwidth]{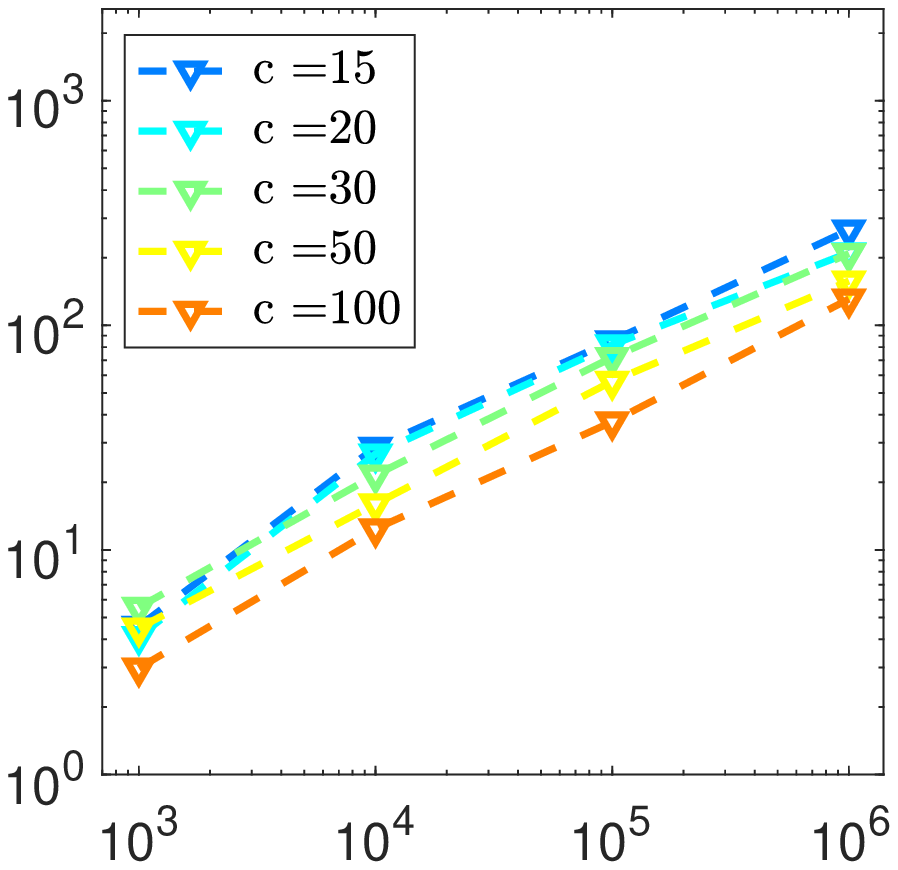}
}
~
\subfigure[Simplices, dynamic $\alpha$.]
{\includegraphics[width=0.3\textwidth]{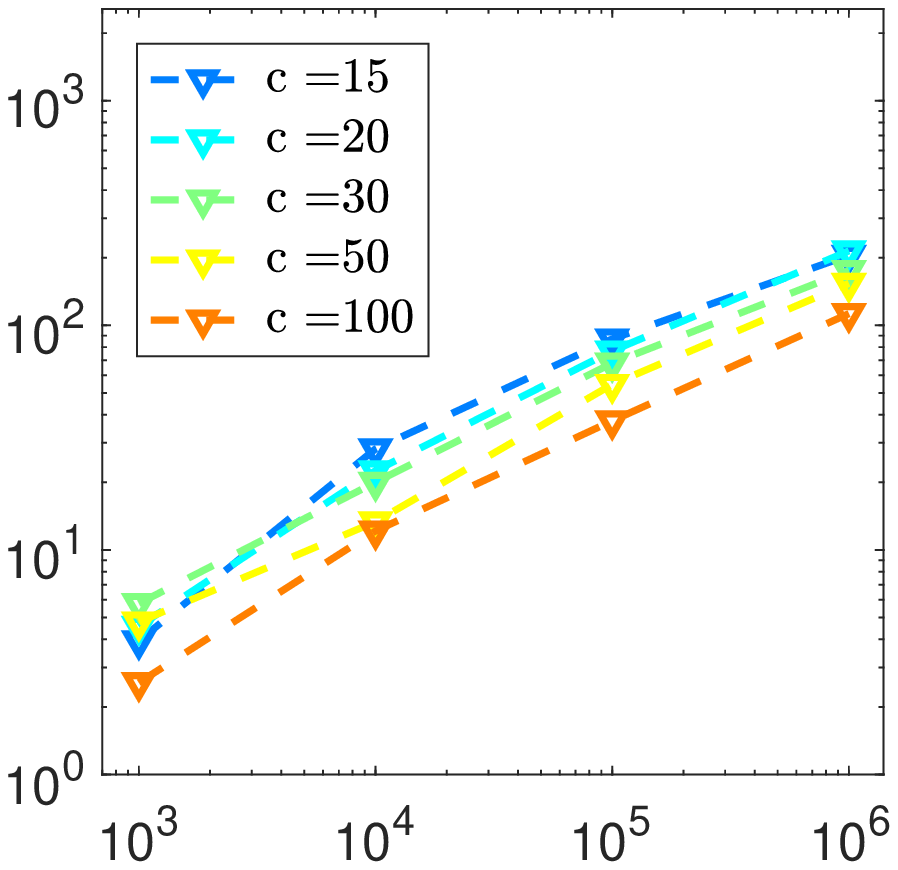}
}
\caption{Fault stress problem. Hyperrectangular and simplex tessellation, sampling with $\alpha$ being 0, 0.9 or set dynamically, and 1000 repetitions, variable $N_{\textup{max}}$ and sampling constant~$c$.}
\label{fig:speedups_P4}
\end{figure}

\subsection{Sod shock tube problem}
\label{sec:num-res:sod}
The Sod test case describes the time-dependent behavior of two gas
phases with different phase properties, separated by a membrane that
is instantly removed at time 0. The problem can be modeled by the
one-dimensional Euler equations and has been investigated in the
uncertainty quantification literature to illustrate challenges related
to discontinuities in stochastic space~\cite{Poette_etal_09,
  Tryoen_etal_10, Pettersson_etal_14}.  We assume the following three
uncertain parameters: left state density $U[0.7, \ 1.3]$, right state
density $U[0.05,\ 0.2]$, and initial location of the membrane with
distribution $U[0.45,\ 0.55]$. The quantity of interest is the density
at $(x,t)=(0.7, 0.1)$.

The speedups for hyperrectangles and simplices, hybrid sampling with $\alpha=0$, $\alpha=0.9$, and dynamic $\alpha$, are shown in Fig.~\ref{fig:speedups_P101}. The best performance is obtained with hyperrectangular stratification, despite a curved hypersurface discontinuity, see Fig.~\ref{fig:surfaces:Sod}.
\begin{figure}[h]
\centering
\subfigure[Hyperrect., $\alpha=0$.]
{\includegraphics[width=0.3\textwidth]{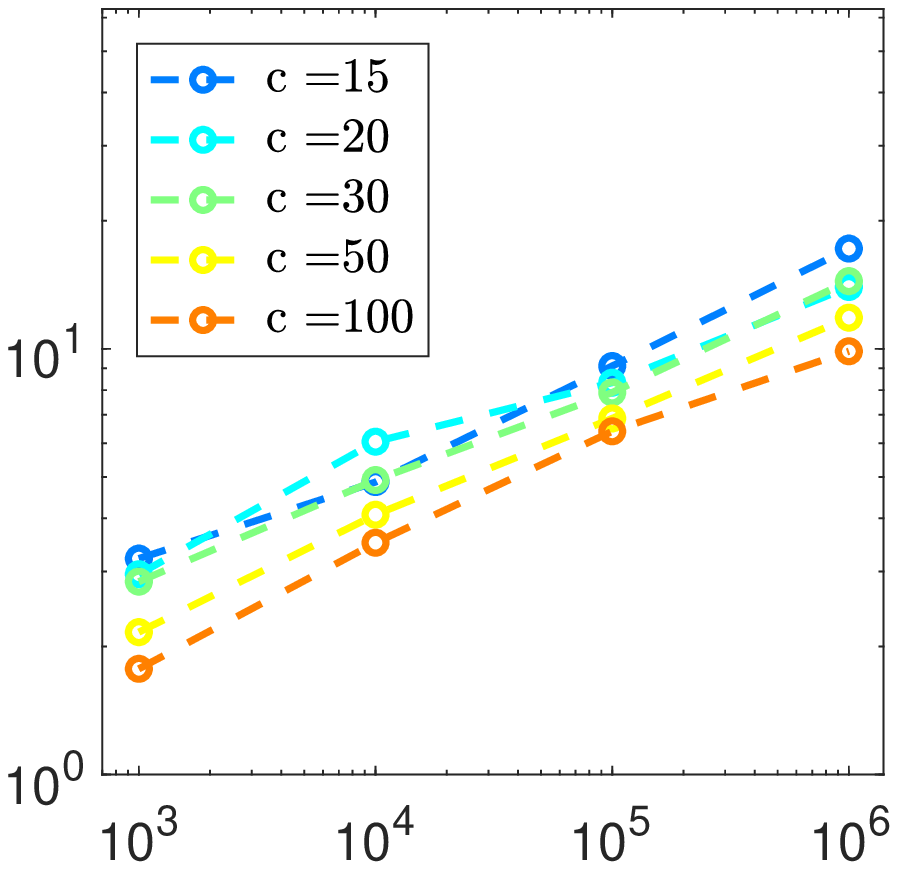}
}
~
\subfigure[Hyperrect., $\alpha=0.9$.]
{\includegraphics[width=0.3\textwidth]{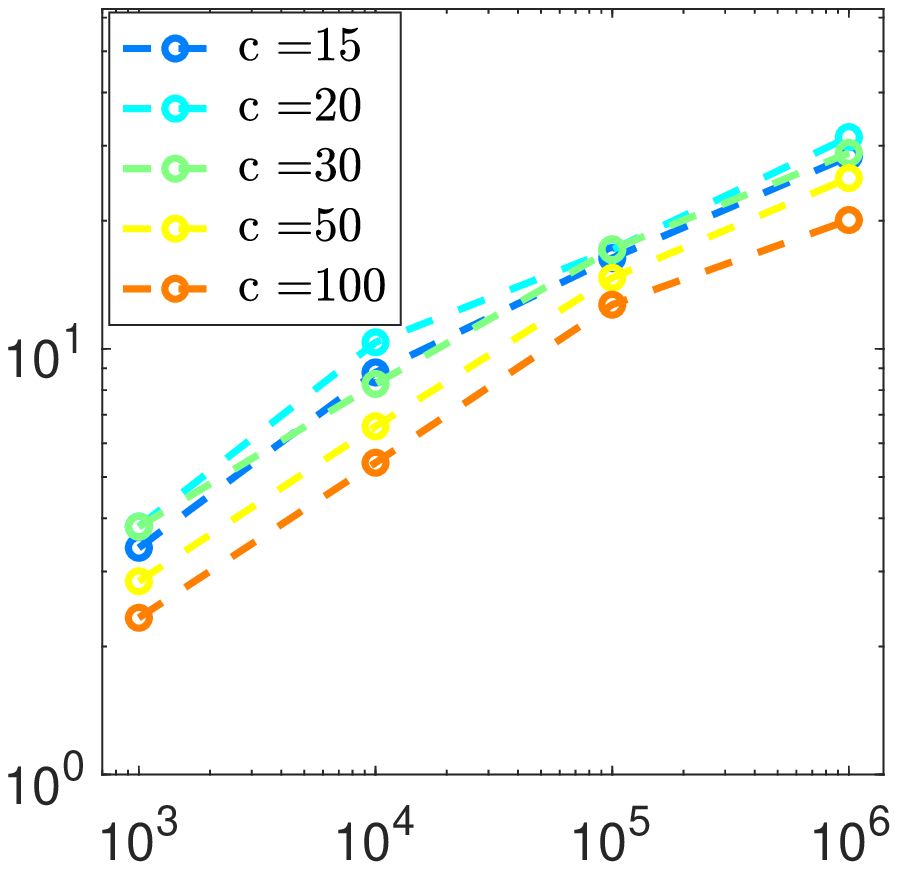}
}
~
\subfigure[Hyperrect., dynamic $\alpha$.]
{\includegraphics[width=0.3\textwidth]{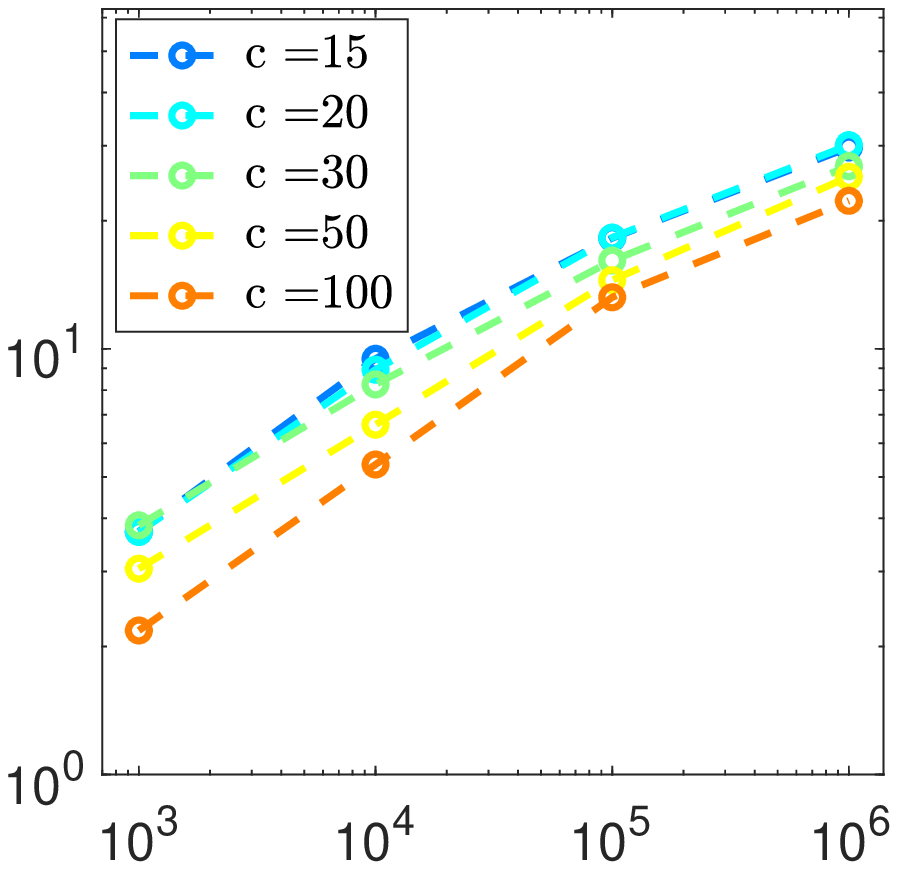}
}
~
\subfigure[Simplices, $\alpha=0$.]
{\includegraphics[width=0.3\textwidth]{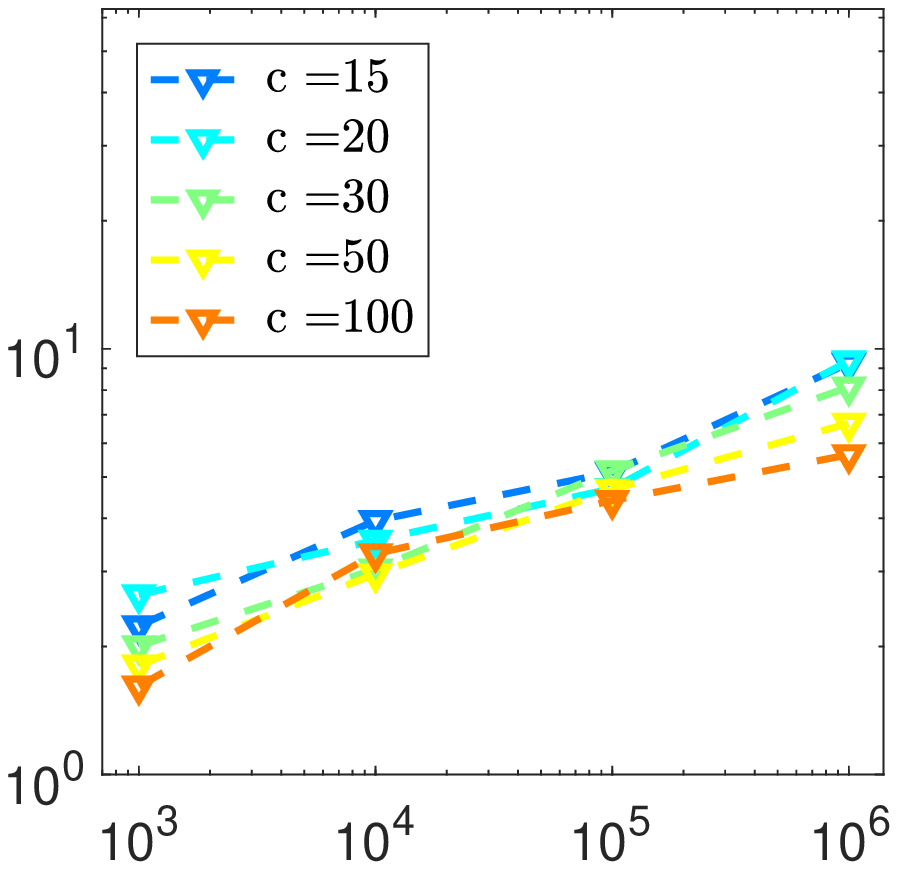}
}
~
\subfigure[Simplices, $\alpha=0.9$.]
{\includegraphics[width=0.3\textwidth]{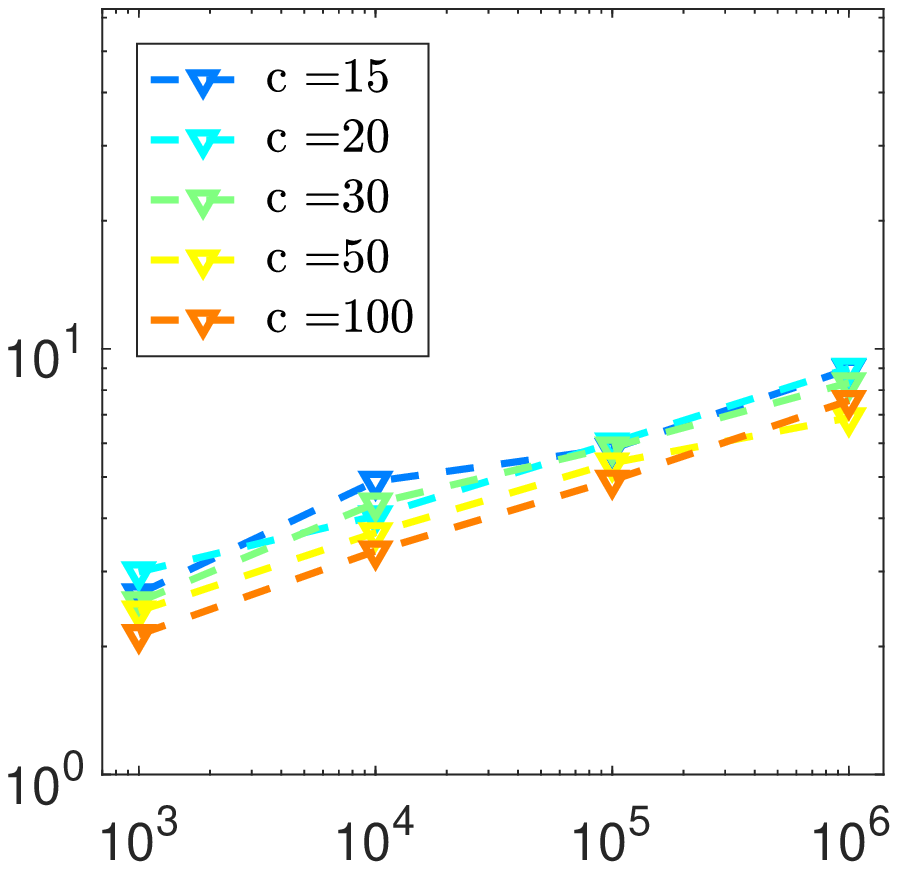}
}
~
\subfigure[Simplices, dynamic $\alpha$.]
{\includegraphics[width=0.3\textwidth]{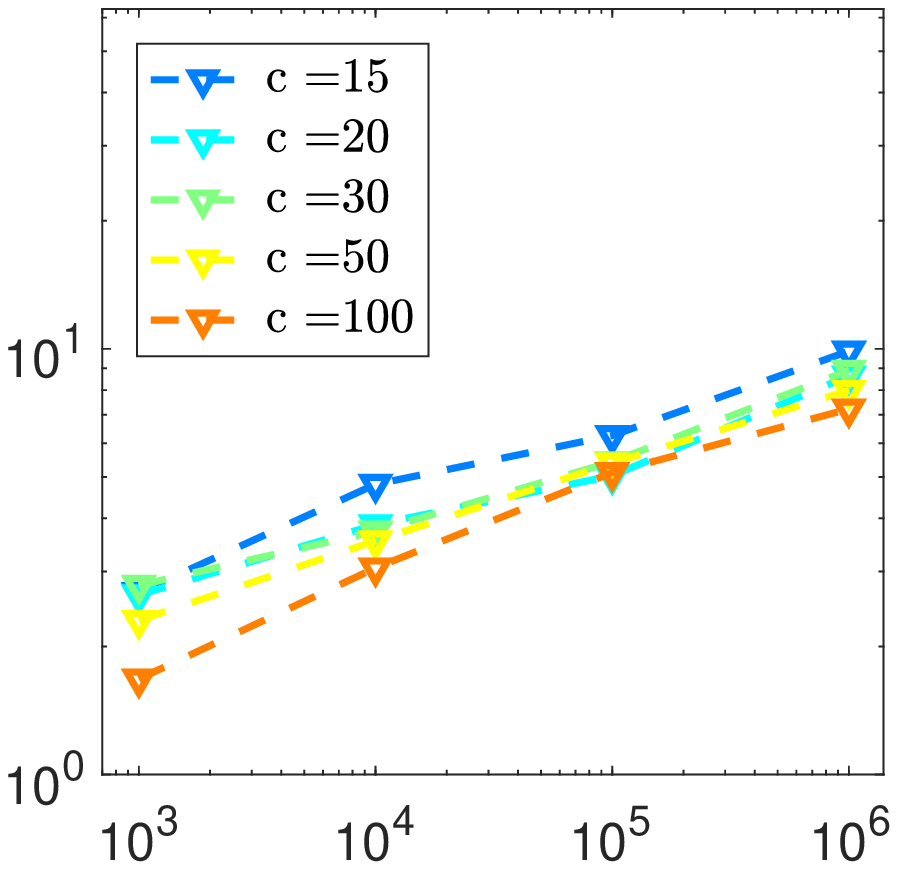}
}
\caption{Sod test case described by the Euler equations. Hyperrectangular and simplex tessellation, sampling with $\alpha$ being 0, 0.9 or set dynamically, and 1000 repetitions, variable $N_{\textup{max}}$ and sampling constant~$c$.}
\label{fig:speedups_P101}
\end{figure}

\subsection{Vertical equilibrium model of CO$_2$ storage}
\label{sec:num-res:CO2}
Subsurface permanent CO$_2$ storage involves injection of large
quantities of CO$_2$ into porous reservoirs with largely unknown
physical properties. The varying temporal and physical scales calls
for simplified-physics models, achieved, e.g., by vertical equilibrium
assumptions and dimension reduction by integration over the vertical
direction~\cite{Nordbotten_Celia_11}.  The result is a nonlinear
hyperbolic conservation law with discontinuous flux function, where
the solution represents the saturation of CO$_2$ as a function of
space and time. For reproducibility, Table~\ref{tab:co2-setup}
contains all values relevant for the setup of the problem, including
distributions for four random parameters: permeability, brine
mobility, CO$_2$ mobility, and background flow rate.  A more detailed
description of the problem, including the conservation law and its
time-dependent solution, can be found in~\cite{MacMinn_etal_10,
  Pettersson_16}. Although not explicitly referred to in this paper,
Table~\ref{tab:co2-setup} lists notation for various parameters
consistent with~\cite{Pettersson_16}, also to facilitate
reproducibility.

\begin{table}
  \centering
  \begin{tabular}{lll}  
    \toprule
    \cmidrule(r){2-3}
    Parameter & Notation & Value/Distribution \\
    \midrule
    Porosity & $\phi$ & 0.15 \\
    Residual brine saturation     & $S_{\textup{br}}$ & 0.1 \\
    Residual CO$_2$ saturation     & $S_{\textup{cr}}$ & 0.1 \\
    Slope angle & $\theta$ & 0.005 \\
    Injection time & $\tau$ & 20 years \\
    Injection rate & $Q_{\textup{inj}}$ & $1 \times 10^{-7}$ m/s \\
    CO$_2$ mobility & $\lambda_{\textup{c}}$ & Uniform $[0.7,\ 1.3] \times 6.25 \times 10^{-5}$ ms/kg\\
    Brine mobility & $\lambda_{\textup{b}}$ & Uniform $[0.8,\ 1.2] \times 5 \times 10^{-4}$ ms/kg \\
    Background flow & $Q$ & Exponential, mean $1 \times 10^{-9}$ m/s \\
    Permeability & $k$ & Lognormal, mean 200 mD, std 50 mD \\
    \bottomrule
  \end{tabular}
  \caption{\label{tab:co2-setup} Parameter setup for the CO$_2$ storage test problem.}
\end{table}
The QoI is the height of the CO$_2$ plume 100 m downstream of the
injection location, 600 years after the end of injection. Speedups for
the CO$_2$ problem as a function of the total number of samples are
shown in Fig.~\ref{fig:speedups_P100}. Speedup is observed for all
sampling schemes, even for the smallest sizes of sample sets. This
problem exhibits zero or small variability away from the
discontinuities, as indicated in Fig.~\ref{fig:surfaces:CO2}, making
it a particularly suitable candidate for the proposed method. At the
same time, the number of dimensions ($n=3$ due to $Q$ being a function
of the ratio of the phase mobilities, rather than their individual
values) already requires a large number of samples to give speedups
greater than an order of magnitude.

\begin{figure}[h]
\centering
\subfigure[Hyperrect., $\alpha=0$.]
{\includegraphics[width=0.3\textwidth]{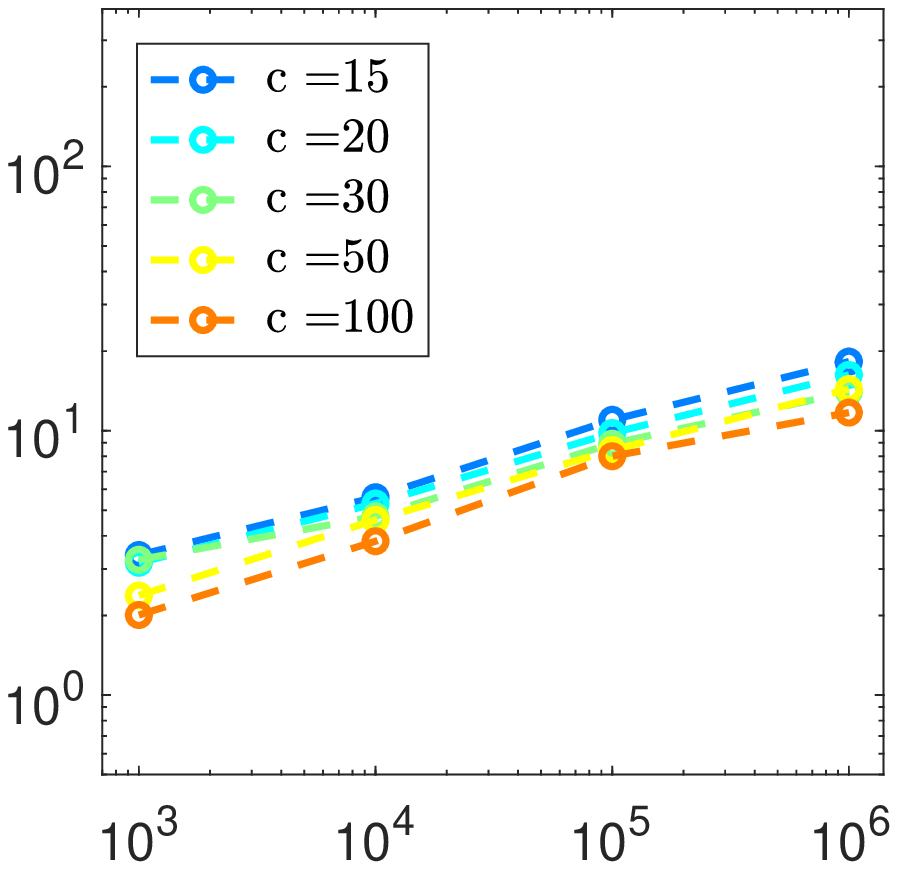}
}
~
\subfigure[Hyperrect., $\alpha=0.9$.]
{\includegraphics[width=0.3\textwidth]{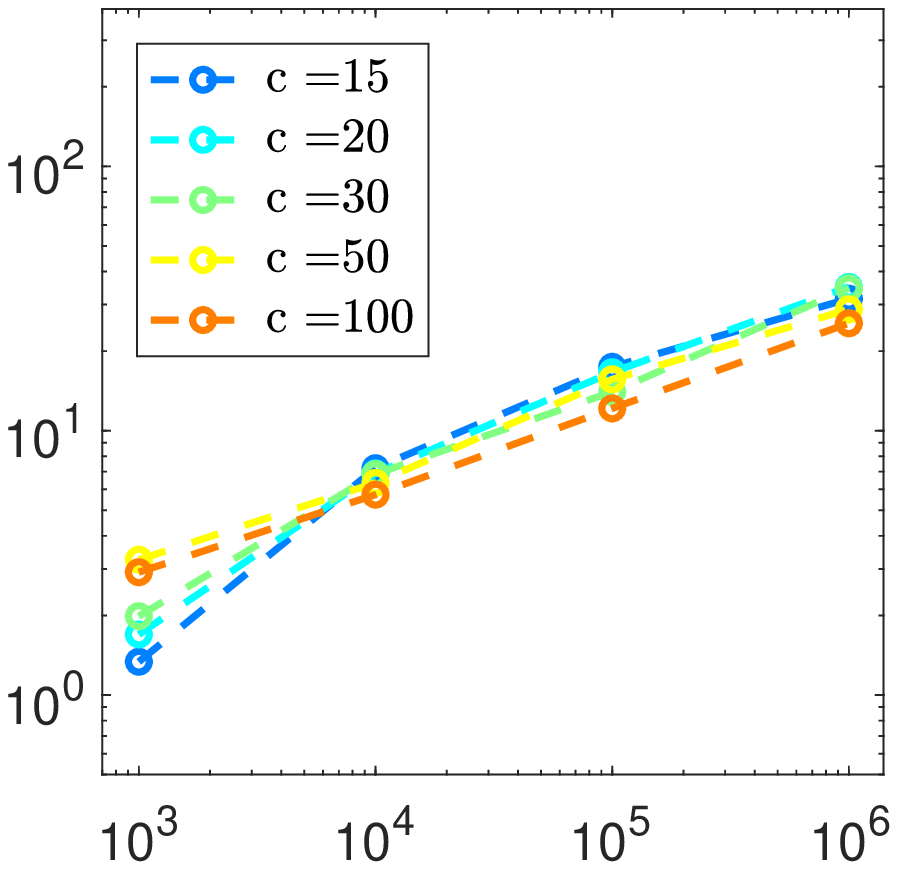}
}
~
\subfigure[Hyperrect., dynamic $\alpha$.]
{\includegraphics[width=0.3\textwidth]{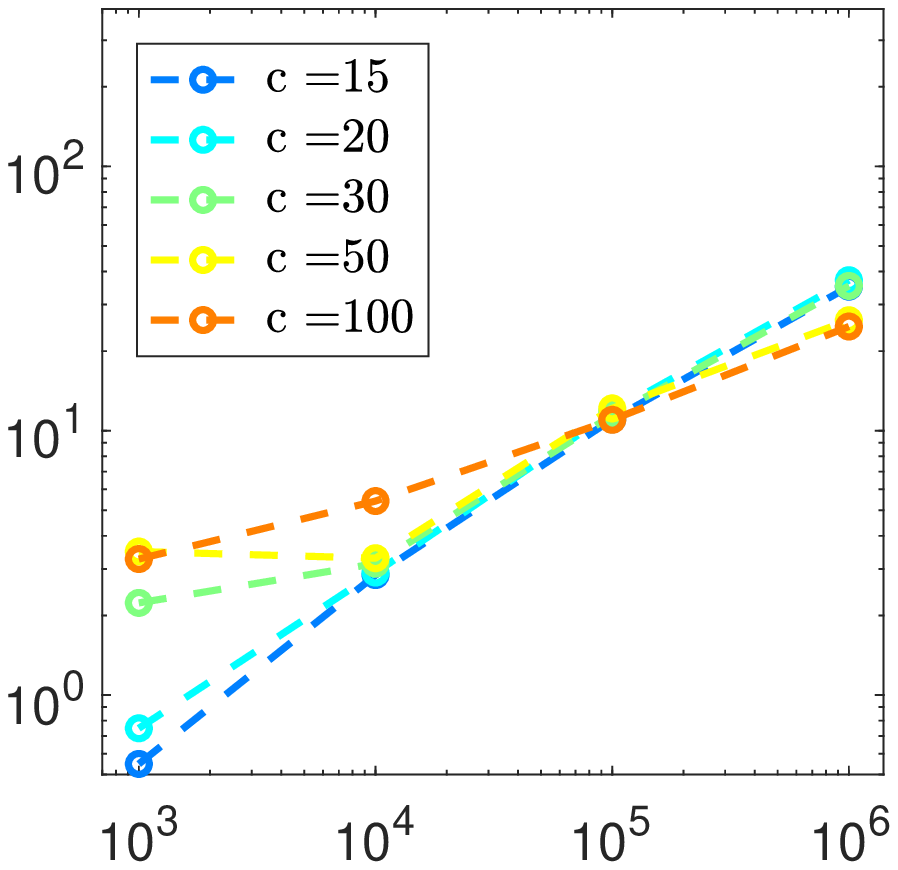}
}
~
\subfigure[Simplices, $\alpha=0$.]
{\includegraphics[width=0.3\textwidth]{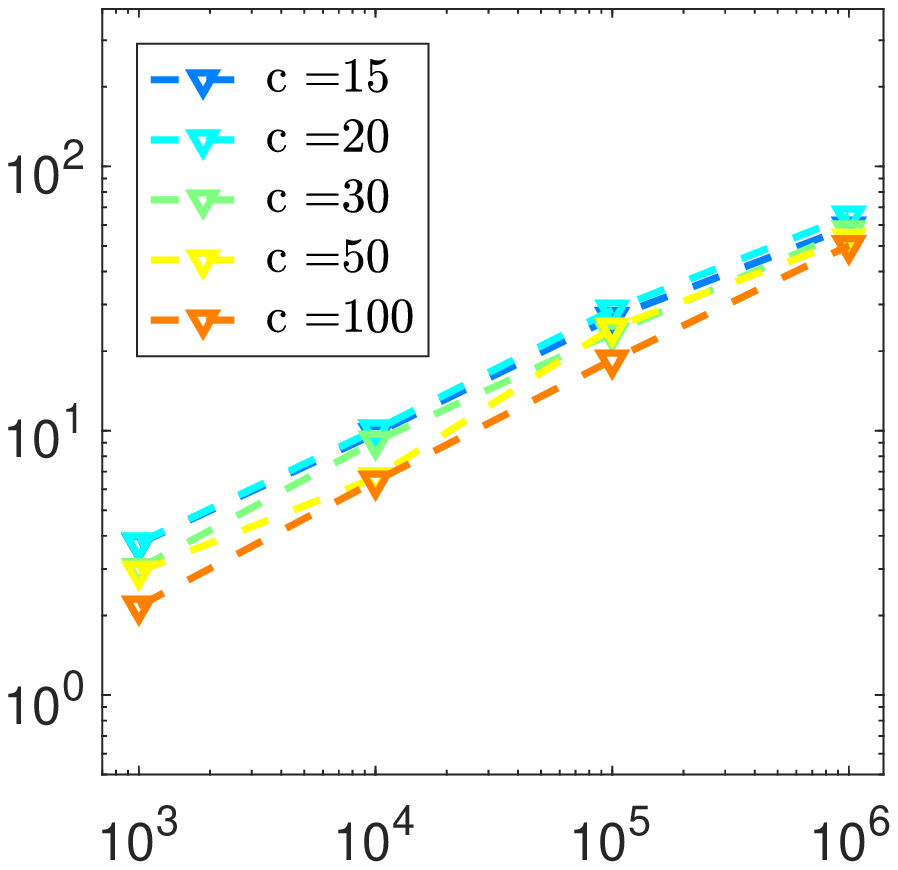}
}
~
\subfigure[Simplices, $\alpha=0.9$.]
{\includegraphics[width=0.3\textwidth]{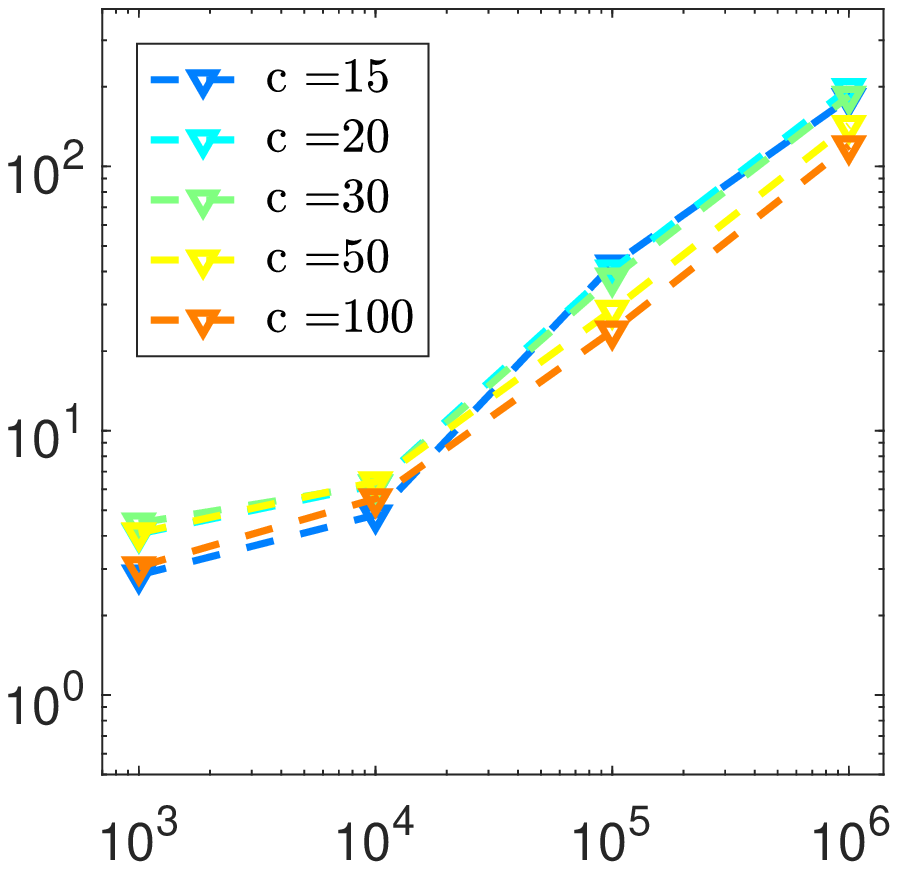}
}
~
\subfigure[Simplices, dynamic $\alpha$.]
{\includegraphics[width=0.3\textwidth]{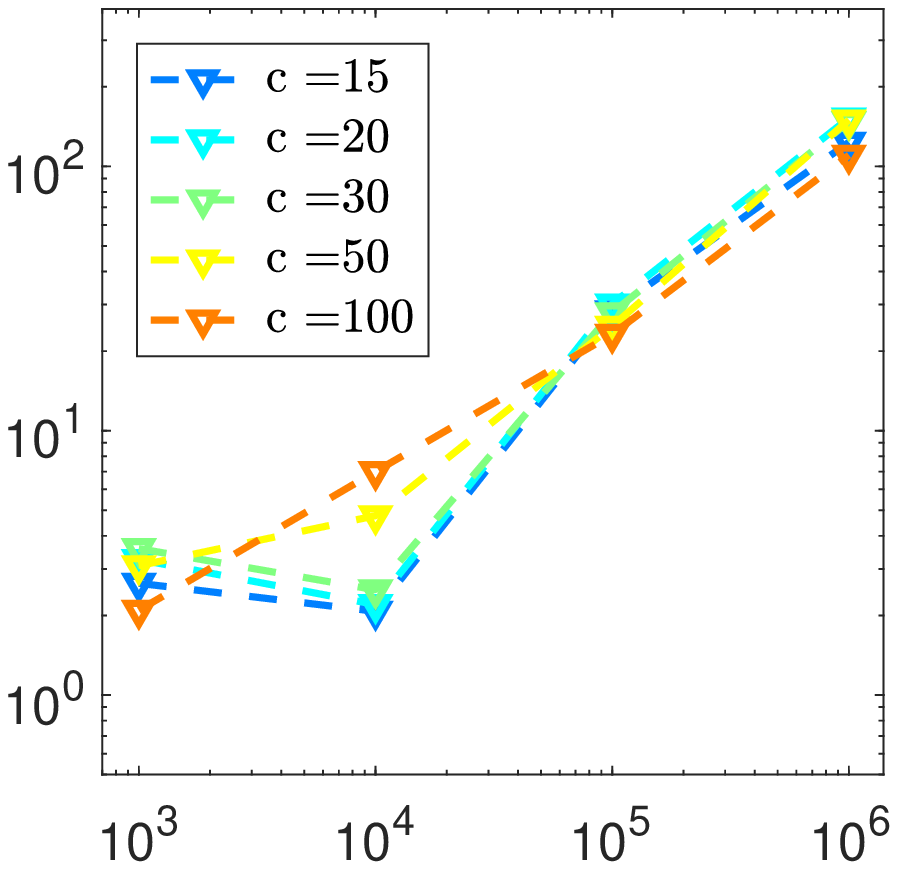}
}
\caption{CO$_2$ storage problem. Hyperrectangular and simplex tessellation, sampling with $\alpha$ being 0, 0.9 or set dynamically, and 1000 repetitions, variable $N_{\textup{max}}$ and sampling constant~$c$.}
\label{fig:speedups_P100}
\end{figure}


\section{Conclusions}
\label{sec:concl}

We have introduced a novel stratified sampling method with adaptive
stratification and sequential, hybridized allocation of samples. For a
fixed stratification, the sample allocation asymptotically approaches
a prescribed linear combination of proportional and optimal
allocation. Letting a fraction of samples be allocated proportionally
adds robustness to the adaptive method, as the stratum standard
deviations are not a priori known but estimated on-the-fly as samples
are added. Moreover, a greedy approach is used to split the stratum
that results in the largest reduction of the stratified sampling
estimator's current variance. Kurtosis dependent estimates are
provided that allow to quantify the probability of drastically
underestimating local strata variances and, thus, of failing to
identify strata that should be refined.

To maintain prescribed stratum sampling rates and for ease of
implementation, a new stratum that arises after re-stratification
should be contained within a single parent stratum, and it should
retain the same geometrical shape. In this work, we suggest using
hyperrectangles or simplices for the stratification, as either of
these shapes can be bisected and the result is two new hyperrectangles
or simplices. In fact, both classes of tessellations allow for a
flexible partition of the stochastic domain.

The proposed method is anticipated to result in significant speedups
for problems where the variability is localized in random space, e.g.,
the PDE solutions describing physical problems with uncertain
parameters and exhibiting steep gradients or discontinuities. In
contrast to, e.g., localized response surface methods based on function
approximations, an advantage of the proposed method is that accurate
identification of steep features is not necessary to obtain good
results. Confining sharp features to strata of small measure is often
sufficient to vastly outperform standard Monte Carlo sampling. This
has been verified experimentally through various test cases,
exhibiting speedups of up to three orders of magnitude compared to
standard Monte Carlo.




\section*{Acknowledgements}

The authors would like to thank Eirik Keilegavlen at the University of Bergen for the setup of the fault  surface stress problem.
The first author was funded by the Research Council of Norway through the project Quantification of fault-related leakage risk (FRISK) under project number 294719. A CC BY or equivalent licence is applied to the author accepted manuscript arising from this submission, in accordance with the grant's open access conditions.

\appendix
\section{Derivation of the gradient of the variance constant}
\label{sec:gradient:varconst}

Introduce the functions 
\begin{equation*}
  g_S\colon \mathbb{R}^{\lvert \mathcal{S}\rvert} \to \mathbb{R}\;,\quad \boldsymbol{\sigma}\mapsto g_S(\boldsymbol{\sigma}) = \frac{\sigma_S}{\langle \boldsymbol{p}, \boldsymbol{\sigma}\rangle }  
\end{equation*}
for any $S\in\mathcal{S}$. Then 
\begin{equation*}
  \boldsymbol{\sigma}\mapsto C_\alpha(\boldsymbol{\sigma}) =  \sum_{S\in\mathcal{S}}\frac{p_S \sigma_S^2}{1+\alpha\left(g_S(\boldsymbol{\sigma}) - 1 \right)}\;,
\end{equation*}
so that $V_\alpha = C_{\alpha}(\boldsymbol{\sigma})/N$. To
quantify the robustness of the variance reduction with respect to
perturbation in $\boldsymbol{\sigma}$, we compute the gradient of $C_{\alpha}$.
Therefore, we first compute the partial derivative of $g_S$ with
respect to the component $\sigma_U$ of $\boldsymbol{\sigma}$ as
\begin{equation*}
    \partial_{\sigma_U} g_S(\boldsymbol{\sigma}) =  \frac{\delta_{S,U}}{\langle \boldsymbol{p}, \boldsymbol{\sigma}\rangle} - \frac{p_U \sigma_S }{{\lvert \langle \boldsymbol{p}, \boldsymbol{\sigma}\rangle \rvert}^2}\;,
\end{equation*}
where $\delta_{S,U}$ denotes the Kronecker delta: $\delta_{S,U} = 1$
if $S=U$ and $\delta_{S,U} = 0$ else. Next we compute the component of
the gradient of $C_{\alpha}$ with respect to $\sigma_U$, which, after
some algebra, can be written as:
\begin{equation*}
    \partial_{\sigma_U} C_{\alpha}(\boldsymbol{\sigma}) 
  = \frac{p_U\sigma_U}{1+\alpha\left(g_U(\boldsymbol{\sigma}) - 1 \right)}\left(1  
    +\frac{1-\alpha}{1+\alpha\left(g_U(\boldsymbol{\sigma}) - 1 \right)}\right)
  + \frac{\alpha p_U}{{\lvert \langle \boldsymbol{p}, \boldsymbol{\sigma}\rangle \rvert}^2} \sum_{S\in\mathcal{S}}\frac{p_S\sigma_S^3 }{{\left(1+\alpha\left(g_S(\boldsymbol{\sigma}) - 1 \right)\right)}^2}\;.
\end{equation*}
Finally, using the fact that
\begin{equation*}
  1+\alpha\left(g_S(\boldsymbol{\sigma}) - 1 \right) = \frac{\alpha \sigma_S + (1-\alpha) \langle \boldsymbol{p}, \boldsymbol{\sigma}\rangle}{\langle \boldsymbol{p}, \boldsymbol{\sigma}\rangle}
\end{equation*}
we can eventually write the partial derivative of $C_{\alpha}$ as
\begin{equation*}
\partial_{\sigma_U} C_{\alpha}(\boldsymbol{\sigma}) = \frac{p_U\sigma_U \langle \boldsymbol{p}, \boldsymbol{\sigma}\rangle}{\alpha \sigma_U + (1-\alpha) \langle \boldsymbol{p}, \boldsymbol{\sigma}\rangle}\left(1  
    +\frac{(1-\alpha)\langle \boldsymbol{p}, \boldsymbol{\sigma}\rangle}{\alpha \sigma_U + (1-\alpha) \langle \boldsymbol{p}, \boldsymbol{\sigma}\rangle}\right)
  + \alpha p_U\sum_{S\in\mathcal{S}}\frac{p_S\sigma_S^3 }{{\left(\alpha \sigma_S + (1-\alpha) \langle \boldsymbol{p}, \boldsymbol{\sigma}\rangle \right)}^2}\;.
\end{equation*}


\section{Further motivation for splitting adaptively}
\label{sec:further:motivation:splitting}
\subsection{Variance reduction by splitting}
Consider an arbitrary stratum $S$ with a partition
$S=S_{+} \cup S_{-}$ defining a refined stratification. We are
interested in the effect of splitting on the variance of the 
estimator. Without loss of generality, assume that
$\E{f(\bY)| \bY\in S} = 0$ (we may always subtract a constant without
changing the variance), and let the relative measure of $S_{+}$ and
$S_{-}$ be $q$ and $1-q$ (the measure of $S$ itself does not
matter). Then,
\begin{multline}
\label{eq:effect_splitting}
\var{f(\bY) | \bY \in S} = \var{f(\bY) | \bY \in S_{+}} + \var{f(\bY) | \bY \in S_{-}} +
2 \cov{f(\bY) |\bY\in S_{+}}{f(\bY) |\bY\in S_{-}} \\
= \var{f(\bY) | \bY \in S_{+}} + \var{f(\bY) | \bY \in S_{-}} + \frac{2(1-q)}{q} \left( \E{f(\bY) |\bY\in S_{-}}\right)^2 \\
\geq q \var{f(\bY) | \bY \in S_{+}} + (1-q) \var{f(\bY) | \bY \in S_{-}},  
\end{multline}
where the second equality follows from $S_{+}$ and $S_{-}$ being
disjoint and the relation
$$\E{f(\bY)|\bY\in S} = q\E{f(\bY)|\bY\in S_{+}}+(1-q)\E{f(\bY)|\bY\in
  S_{-}}=0\;.$$ The right-hand side of Eq.~\eqref{eq:effect_splitting}
equals the variance contribution to the estimator 
from the two newly created strata. Thus, we have shown that splitting
of a stratum always reduces the variance of the estimator. 
that optimal allocation is maintained after splitting.

\subsection{Stratification of domains of discontinuous functions}

Let $S_{+}$ and $S_{-}$ be two subdomains that forms a partition of $\mathfrak{U} = [0,\ 1]^{n}$, separated by a hypersurface $H \subset \mathbb{R}^{n-1}$.
Let $f$ be a piecewise constant function on $\mathbb{R}^{n}$, defined by
\begin{equation}
f(\bY) = \left\{
\begin{array}{ll}
c_{+} & \mbox{if } \bY\in S_{+} \\
c_{-} & \mbox{if } \bY\in S_{-}
\end{array}
\right.,
\end{equation}
where $c_{+}$ and $c_{-}$ are constants, and consider a neighborhood
$S_{H}$ around $H$.  Denote the measure of $S_{+}$ and $S_{-}$ by
$p_{+} = \Pp{S_{+}}$ and $p_{-}=\Pp{S_{-}}$, respectively, and
$\tilde{p}_{+} = \Pp{S_{+} \cap S_H}$ and
$\tilde{p}_{-} = \Pp{S_{-} \cap S_H}$. Setting
$p = \tilde{p}_{+}/p_{S_{H}}$ we have
\begin{multline}
\label{eq:var_disc_fun}
\var{f(\bY) | \bY \in S_{H}} = \frac{\tilde{p}_{+}}{p_{S_H}} \left( 1-\frac{\tilde{p}_{+}}{p_{S_H}} \right) c_{+}^{2} + \frac{\tilde{p}_{-}}{p_{S_H}} \left( 1-\frac{\tilde{p}_{-}}{p_{S_H}} \right) c_{-}^{2} - 2 \frac{\tilde{p}_{+} \tilde{p}_{-} }{p_{S_H}^2} c_{+} c_{-} \\
= p(1-p)(c_{+} - c_{-})^2,
\end{multline}
where we have used that $p_{+}+p_{-}=1$, and $\tilde{p}_{+}+\tilde{p}_{-} = p_{S_{H}}$. Note that~\eqref{eq:var_disc_fun} is similar to the expression for the global variance,
\[
\var{f(\bY)}=p_{+}(1-p_{+})(c_{+} - c_{-})^2.
\]
To minimize the variance, we see from~\eqref{eq:var_disc_fun} that the stratum $S_{H}$ should be small, whereas the local variances in $S_{+}\setminus S_{H}$ and $S_{-}\setminus S_{H}$ are identically zero (since $f$ is constant).

\bibliographystyle{plain}
\bibliography{strat_samp_bib.bib}

\begin{thebibliography}{10}

\bibitem{Asmussen_Glynn_07}
S.~Asmussen and P.~W. Glynn.
\newblock {\em Stochastic simulation: algorithms and analysis}, volume~57.
\newblock Springer, New York, 2007.

\bibitem{Pereira_etal_16}
L.~Cabral~Pereira, M.~S\'{a}nchez, and L.~J. do~Nascimento Guimar\~{a}es.
\newblock Uncertainty quantification for reservoir geomechanics.
\newblock {\em Geomech. Energy Envir.}, 8:76--84, 2016.
\newblock Themed Issue on Selected Papers Symposium of Energy Geotechnics 2015
  — Part II.

\bibitem{Christie_etal_06}
M.~Christie, V.~Demyanov, and D.~Erbas.
\newblock Uncertainty quantification for porous media flows.
\newblock {\em J. Comput. Phys.}, 217(1):143--158, 2006.

\bibitem{Cliffe_etal_11}
K.~A. Cliffe, M.~B. Giles, R.~Scheichl, and A.~L. Teckentrup.
\newblock Multilevel {M}onte {C}arlo methods and applications to elliptic
  {PDE}s with random coefficients.
\newblock {\em Comput. Visual Sci.}, 14(1):3--15, 2011.

\bibitem{Cochran_77}
W.~G. Cochran.
\newblock {\em Sampling Techniques}.
\newblock John Wiley \& Sons, New York, 3rd edition, 1977.

\bibitem{Cuvelier_Scarella_18}
Fran{\c c}ois Cuvelier and Gilles Scarella.
\newblock {Vectorized algorithms for regular tessellations of d-orthotopes and
  their faces}.
\newblock working paper or preprint, November 2017.

\bibitem{DasGupta_08}
A.~DasGupta.
\newblock {\em Asymptotic theory of statistics and probability}.
\newblock Springer, New York, 2008.

\bibitem{DeLuigi_Maire_10}
C.~De~Luigi and S.~Maire.
\newblock Adaptive integration and approximation over hyper-rectangular regions
  with applications to basket options pricing.
\newblock {\em Monte Carlo Methods Appl.}, 16:265--282, 12 2010.

\bibitem{Dick2013}
J.~Dick, F.~Y. Kuo, and I.~H. Sloan.
\newblock High-dimensional integration: {The quasi-Monte Carlo way}.
\newblock {\em Acta Numerica}, 22:133--288, 2013.

\bibitem{Etore_etal_11}
P.~{\'E}tor{\'e}, G.~Fort, B.~Jourdain, and E.~Moulines.
\newblock On adaptive stratification.
\newblock {\em Ann. Oper. Res.}, 89(1):127--154, 2011.

\bibitem{Etore_Jourdain_10}
P.~{\'E}tor{\'e} and B.~Jourdain.
\newblock Adaptive optimal allocation in stratified sampling methods.
\newblock {\em Methodol. Comput. Appl.}, 12(3):335--360, 2010.

\bibitem{Giles_2015}
M.~B. Giles.
\newblock Multilevel {M}onte {C}arlo methods.
\newblock {\em Acta Numer.}, 24:259--328, 2015.

\bibitem{Giles_2019}
M.~B. Giles and A.-L. Haji-Ali.
\newblock Multilevel nested simulation for efficient risk estimation.
\newblock {\em SIAM/ASA J. Uncertain. Quantif.}, 7(2):497--525, 2019.

\bibitem{Gorodetsky_Marzouk_14}
A.~Gorodetsky and Y.~Marzouk.
\newblock Efficient localization of discontinuities in complex computational
  simulations.
\newblock {\em SIAM J. Sci. Comput.}, 36(6):A2584--A2610, 2014.

\bibitem{Haji-Ali_etal_16}
A.-L. Haji-Ali, F.~Nobile, and R.~Tempone.
\newblock Multi {I}ndex {M}onte {C}arlo: when sparsity meets sampling.
\newblock {\em Numer. Math.}, 132(4):767--806, 2016.

\bibitem{Hosder_etal_08}
S.~Hosder, R.~Walters, and M.~Balch.
\newblock Efficient uncertainty quantification applied to the aeroelastic
  analysis of a transonic wing.
\newblock In {\em 46th AIAA Aerospace Sciences Meeting and Exhibit}, 2008.

\bibitem{Jakeman_etal_13}
J.~D. Jakeman, A.~Narayan, and D.~Xiu.
\newblock Minimal multi-element stochastic collocation for uncertainty
  quantification of discontinuous functions.
\newblock {\em J. Comput. Phys.}, 242:790-- 808, 2013.

\bibitem{Kroese_etal_11}
D.~P. Kroese, T.~Taimre, and Z.~I. Botev.
\newblock {\em Handbook of Monte Carlo methods}.
\newblock Wiley New Jersey, 2011.

\bibitem{Krumscheid2018}
S.~Krumscheid and F.~Nobile.
\newblock Multilevel {M}onte {C}arlo {A}pproximation of {F}unctions.
\newblock {\em SIAM/ASA J. Uncertain. Quantif.}, 6(3):1256--1293, 2018.

\bibitem{Pisaroni_2020}
S.~Krumscheid, F.~Nobile, and M.~Pisaroni.
\newblock Quantifying uncertain system outputs via the multilevel {M}onte
  {C}arlo method---{P}art {I}: {C}entral moment estimation.
\newblock {\em J. Comput. Phys.}, 414, 2020.

\bibitem{Kuhn_60}
H.~W. Kuhn.
\newblock Some combinatorial lemmas in topology.
\newblock {\em IBM J. Res. Dev.}, 4(5):518--524, November 1960.

\bibitem{MacMinn_etal_10}
C.~W. MacMinn, M.~L. Szulczewski, and R.~Juanes.
\newblock {CO}$_2$ migration in saline aquifers. part 1. capillary trapping
  under slope and groundwater flow.
\newblock {\em J. Fluid Mech.}, 662:329--351, 2010.

\bibitem{Muller_etal_13}
F.~M\"{u}ller, P.~Jenny, and D.~W. Meyer.
\newblock Multilevel {M}onte {C}arlo for two phase flow and
  {B}uckley-{L}everett transport in random heterogeneous porous media.
\newblock {\em J. Comput. Phys.}, 250:685--702, 2013.

\bibitem{Niederreiter_1992}
H.~Niederreiter.
\newblock {\em Random number generation and quasi-{M}onte {C}arlo methods}.
\newblock Society for Industrial and Applied Mathematics (SIAM), Philadelphia,
  PA, 1992.

\bibitem{Nordbotten_Celia_11}
J.M. Nordbotten and M.A. Celia.
\newblock {\em Geological Storage of {CO}$_2$: Modeling Approaches for
  Large-Scale Simulation}.
\newblock Wiley, 2011.

\bibitem{Palmer_00}
T.~N. Palmer.
\newblock Predicting uncertainty in forecasts of weather and climate.
\newblock {\em Rep. Prog. Phys.}, 63(2):71--116, 2000.

\bibitem{Pettersson_16}
P.~Pettersson.
\newblock Stochastic {G}alerkin formulations for {CO}$_2$ transport in
  aquifers: Numerical solutions with uncertain material properties.
\newblock {\em Transport Porous Med.}, 114(2):457--483, 2016.

\bibitem{Pettersson_etal_14}
P.~Pettersson, G.~Iaccarino, and J.~Nordstr\"{o}m.
\newblock A stochastic {G}alerkin method for the {E}uler equations with {R}oe
  variable transformation.
\newblock {\em J. Comput. Phys.}, 257:481--500, 2014.

\bibitem{Poette_etal_09}
G.~Po{\"e}tte, B.~Despr{\'e}s, and D.~Lucor.
\newblock Uncertainty quantification for systems of conservation laws.
\newblock {\em J. Comput. Phys.}, 228(7):2443--2467, 2009.

\bibitem{Press_Farrar_90}
W.~H. Press and G.~R. Farrar.
\newblock Recursive stratified sampling for multidimensional {M}onte {C}arlo
  integration.
\newblock {\em Comput. Phys.}, 4(2):190--195, 1990.

\bibitem{Rushdi_etal_2017}
A.~Rushdi, L.~P. Swiler, E.~T. Phipps, M.~D'Elia, and M.~S. Ebeida.
\newblock Vps: Voronoi piecewise surrogate models for high-dimensional data
  fitting.
\newblock {\em Int. J. Uncertain. Quan.}, 7(1):1--21, 2017.

\bibitem{Serfling_80}
R.~J. Serfling.
\newblock {\em Approximation theorems of mathematical statistics}.
\newblock John Wiley \& Sons, Inc., New York, 1980.

\bibitem{Shields_16}
M.~D. Shields.
\newblock Refined latinized stratified sampling: A robust sequential sample
  size extension methodology for high-dimensional latin hypercube and
  stratified designs.
\newblock {\em Int. J. Uncertain. Quan.}, 6(1):79--97, 2016.

\bibitem{Shields_etal_15}
M.~D. Shields, K.~Teferra, A.~Hapij, and R.~P. Daddazio.
\newblock Refined stratified sampling for efficient {M}onte {C}arlo based
  uncertainty quantification.
\newblock {\em Reliab. Eng. Syst. Safe.}, 142:310--325, 2015.

\bibitem{Tartakovsky_Broyda_11}
D.~M. Tartakovsky and S.~Broyda.
\newblock {PDF} equations for advective-reactive transport in heterogeneous
  porous media with uncertain properties.
\newblock {\em J. Contam. Hydrol.}, 120-121:129--140, 2011.

\bibitem{Tryoen_etal_10}
J.~Tryoen, O.~Le~Ma\^{i}tre, M.~Ndjinga, and A.~Ern.
\newblock Intrusive {G}alerkin methods with upwinding for uncertain nonlinear
  hyperbolic systems.
\newblock {\em J. Comput. Phys.}, 229(18):6485--6511, 2010.

\bibitem{Vershynin_18}
R.~Vershynin.
\newblock {\em High-dimensional probability}, volume~47 of {\em Cambridge
  Series in Statistical and Probabilistic Mathematics}.
\newblock Cambridge University Press, Cambridge, 2018.
\newblock An introduction with applications in data science.

\bibitem{Welford_1962}
B.~P. Welford.
\newblock Note on a method for calculating corrected sums of squares and
  products.
\newblock {\em Technometrics}, 4:419--420, 1962.

\bibitem{Witteveen_Iaccarino_12}
J.~A.~S. Witteveen and G.~Iaccarino.
\newblock Simplex stochastic collocation with random sampling and extrapolation
  for nonhypercube probability spaces.
\newblock {\em SIAM J. Sci. Comput.}, 34(2):A814--A838, 2012.

\bibitem{Xiu_2010}
D.~Xiu.
\newblock {\em Numerical methods for stochastic computations: A spectral method
  approach}.
\newblock Princeton University Press, Princeton, NJ, 2010.

\end{thebibliography}

\end{document}